\documentclass[11pt,reqno]{amsart}

\usepackage[active]{srcltx}
\usepackage{graphicx}
\usepackage{color}
\usepackage{enumerate,cleveref}
\usepackage{amssymb}
\usepackage{overpic}
\usepackage{amsfonts,amsmath,amssymb}
\usepackage{epstopdf}
\usepackage{algorithmic}
\usepackage{mathrsfs}
\usepackage{amsopn}
\usepackage{bm}
\usepackage{algorithm2e}
\usepackage{array,booktabs}
\usepackage{subcaption}

\newtheorem{lemma}{Lemma}[section]
\newtheorem{corollary}[lemma]{Corollary}
\newtheorem{theorem}[lemma]{Theorem}
\newtheorem{proposition}[lemma]{Proposition}

\newcommand{\Bx}{\mathbf{x}}
\newcommand{\Bv}{\mathbf{v}}
\newcommand{\By}{\mathbf{y}}
\newcommand{\Be}{\mathbf{e}}
\newcommand{\Br}{\mathbf{r}}
\newcommand{\Bz}{\mathbf{z}}
\newcommand{\Beps}{\bm{\epsilon}}

\newcommand{\Hd}{\mathbf{\hat{d}}}
\newcommand{\Hx}{\mathbf{\hat{x}}}

\title{Nonlinear Iterative Hard Thresholding for Inverse Scattering}

\author{Anna C.~Gilbert}
\address{Department of Mathematics \\ University of Michigan \\ Ann Arbor, MI 48109}
\email{annacg@umich.edu}
\author{Howard W.~Levinson}
\address{Department of Mathematics \\ University of Michigan \\ Ann Arbor, MI 48109}
\email{levh@umich.edu}
\author{John C.~Schotland}
\address{Department of Mathematics and Department of Physics \\ University of Michigan \\ Ann Arbor, MI 48109}
\email{schotland@umich.edu}

\begin{document}

\begin{abstract}
We consider the inverse scattering problem for sparse scatterers. An image reconstruction algorithm is proposed that is based on a nonlinear generalization of iterative hard thresholding. The convergence and error of the method was analyzed by means of coherence estimates and compared to numerical simulations.
\end{abstract}

\maketitle

\section{Introduction}
\subsection{Background}
Scattering experiments are of fundamental importance in nearly every branch of physics. They also arise in numerous applied fields including optics, seismology and biomedical imaging. In this paper, we consider the following experimental setup. A wave field is incident on a medium  and the resulting scattered field is measured. The spatially-dependent properties of the medium are encoded in coefficients that arise in the wave equation. Depending on the physical setting, the associated inverse scattering problem is to reconstruct one or more coefficients from boundary measurements. Much is known about theoretical aspects of the problem, especially on matters of uniqueness, stability and reconstruction~\cite{cakoni_colton,cakoni_colton_again,chadan,colton_kress,colton_kress_again,
isakov,kirch}. By uniqueness we mean the injectivity of the forward map from the coefficients, sometimes referred to as the scattering potential, to the scattered field. Stability refers to continuity of the inverse map from the scattered field to the potential. We note that inverse scattering problems are typically ill-posed, which means that the inverse map must be suitably regularized to achieve stable inversion. 

The forward map depends nonlinearly on the scattering potential. Linearization of the forward map is referred to as the first Born approximation (FBA). The ISP within the FBA is relatively well understood. The development of reconstruction methods for the nonlinear ISP is an area of active research. These include optimization, qualitative and direct methods. There is also considerable interest in developing reconstruction algorithms for scatterers that are known a priori to be sparse~\cite{Borcea_2015,0266-5611-27-3-035013,0266-5611-33-6-060301,PMC2858419}. 
For instance, similar algorithms to iterative hard thresholding (IHT) have been applied to the ISP in the FBA~\cite{Fanngiang_Remote,0266-5611-26-3-035008}. We note that IHT has also proven to be effective in many applications other than inverse scattering~\cite{blu_iht}.

In this paper we apply a generalization of IHT to the nonlinear ISP. Our results provide nonlinear corrections to inversion within the FBA. In particular, we characterize the convergence and error of the method. The analysis is primarily carried out within the framework of coherence estimates. Extensions to restricted isometry property (RIP) estimates are also provided. Numerical simulations are used to illustrate the results.
 
\subsection{Related work}
A nonlinear version of IHT was proposed in \cite{blu_nliht}. An analysis of this algorithm was performed by making use of the restricted isometry property (RIP). However, many physical measurement systems of interest do not obey even moderately strong RIP bounds. Moreover, computationally efficient methods to compute RIP constants are not available, thereby preventing the practical verification of RIP bounds. Thus, sparse imaging in deterministic settings is usually characterized by the mutual coherence. We note that coherence calculations for various scattering experiments have been reported~\cite{0266-5611-26-3-035008,5728925}.

\subsection{Organization}
The paper is organized as follows. In~\cref{sec:iht}, we state the nonlinear IHT algorithm and provide a convergence analysis based on coherence estimates. Relevant background on inverse scattering is then introduced in ~\cref{sec:scatt}.  In ~\cref{sec:invscatteringIHT} we apply the nonlinear IHT algorithm to inverse scattering.  An analysis of the resulting algorithm is presented in~\cref{sec:4}. Our results are illustrated with numerical simulations in~\cref{sec:coherence}. A related analysis of nonlinear IHT for inverse scattering using RIP is given in Appendix B.

\section{Preliminaries and Background}
\label{sec:prelim}
Throughout the paper, matrices will be denoted by uppercase letters and vectors will be denoted by bold lowercase letters.  For a matrix $A$ with entries $A_{ij}$, we denote its $j$th column by $A_j$.  We denote the adjoint of $A$ by $A^*$.  For an $N\times1$ vector $\Bx$, $D_\Bx$ is the $N\times N$ diagonal matrix with diagonal entries given by $\Bx$ and $\Hx$ is the vector of unit length in the same direction as $\Bx$.  We define the max norm of a matrix $A$ as $\|A\|_{\max}=\max_{ij}|A_{ij}|$. We also let $\Omega\subset \mathbb{R}^3$ denote a bounded domain with smooth boundary $\partial\Omega$.

\subsection{Iterative Hard Thresholding}
\label{sec:iht}
Consider the linear system 
\begin{equation}
\label{eq:linear}
	\By = \Phi \Bx + \Beps,
\end{equation}  
where $\Phi$ is an $M\times N$ matrix, $\Bx$ is an $s$-sparse vector of length $N$ with $s$ non-zero entries, and $\Beps$ is a vector of length $M$. We will refer to $\Phi$ as the sensing matrix and $\Beps$ accounts for the effects of noise. When $\Phi$ has fewer rows than columns, the linear system is under-determined and the problem of finding or recovering $\Bx$ is frequently referred to as \emph{compressive sensing}. Here the matrix $\Phi$ \emph{senses} $\Bx$, albeit in a noisy fashion, and the vector $\By$ is a set of noisy observations of $\Bx$. Despite this problem being under-determined, if there are enough rows in $\Phi$ or observations of $\Bx$ (roughly $O(s \log N)$), then $\Bx$ can be recovered from its observations $\By$, provided $\Bx$ is $s$-sparse, up to the level of the noise $\Beps$.

IHT is an algorithm for sparse signal recovery, first proposed by Blumensath and Davies in \cite{blu_iht}, and subsequently analyzed and developed in \cite{blu_aiht,blu_iht2,blu_niht}.  The IHT algorithm recovers $\Bx$ by the iterative process 
\begin{equation}
\label{eq:iht}
	\Bx_{n+1} = H_s \left( \Bx_{n} + \Phi^*(\By - \Phi \Bx_{n} ) \right),
\end{equation}
where $H_s(x)$ is the nonlinear thresholding operator that sets all but the largest (in absolute value) $s$ elements of $\Bx$ to zero.  This algorithm is the classical Richardson first-order iteration with the additional application of a thresholding operator to promote sparsity \cite{doi:10.1137/1.9781611970944}.  

Since not all compressive sensing problems are linear, it is of interest to consider the nonlinear problem 
\begin{equation}
	\By = \Phi(\Bx) + \Beps ,
\end{equation}
where $\Phi$ is a nonlinear function rather than a linear one.  Blumensath~\cite{blu_nliht} proposed solving the above problem with an analogous nonlinear IHT algorithm. This algorithm replaces the iterative step in Eq.~\eqref{eq:iht} by
\begin{equation}
\label{eq:iht_nl}
	\Bx_{n+1} = H_s \left(\Bx_{n} + \Phi^*_{\Bx_n}(\By - \Phi_{\Bx_n} \Bx_{n}) \right) \ ,
\end{equation}
where $\Phi_{\Bx_n}$ is the linearization of $\Phi$ at the point $\Bx_n$.  Blumensath analyzed this algorithm by making use of the restricted isometry property (RIP), which gives a measure of how ``close'' to orthogonal the sensing matrix is when applied to $s$-sparse vectors.  In particular, a matrix $\Phi$ satisfies the RIP of order $s$ if there exists a constant $\delta_s$ with $0 < \delta_s < 1$
such that, for all $s$-sparse vectors $\Bx$, 
\begin{equation}
	(1-\delta_s)\|\Bx\|_2^2\le \|\Phi(\Bx)\|_2^2 \le (1+\delta_s)\|\Bx\|_2^2 \ .
\end{equation}
See~\cite{CandesTao2005} for further details. Rudelson and Vershynin~\cite{RudelsonVershynin2008} have shown that there are several randomized algorithms for generating matrices that satisfy the RIP, namely a random matrix with iid sub-Gaussian entries or a random submatrix of a matrix with bounded entries with a sufficient number of rows (drawn uniformly at random). There are also deterministic constructions of matrices that satisfy the RIP that have nearly but not quite optimal dimensions. 

While the analysis of both linear and nonlinear compressive sensing algorithms using RIP yields strong theoretical results, most scattering experiments with deterministic source and detector geometries do not yield sensing matrices for which RIP estimates are easy to verify either analytically or computationally.  Therefore, we utilize a weaker notion than RIP, namely coherence, to analyze the nonlinear IHT algorithm. The coherence of the sensing matrix $\Phi$ is defined by 
\begin{equation}
\label{eq:mu_def}
\mu(\Phi)=\max_{j\neq k} \frac{|\langle \Phi_j, \Phi_k\rangle|}{\|\Phi_j\|_2\|\Phi_k\|_2}.
\end{equation}
Equivalently, if $\Phi$ is column-normalized ($\|\Phi_j\|_2=1$ for all $j$), then the mutual coherence is the largest (in absolute value) off-diagonal entry in the matrix $\Phi^*\Phi$.  
This quantity is considerably easier to compute, does not depend on the putative sparsity of the underlying data, and does not require sophisticated algebraic constructions (for either random or deterministic matrices), in contrast to the RIP constructions. There are, however, deterministic constructions of RIP matrices based upon simple coherence estimates but these have far from optimal dimensions~\cite{NguyenShin2013}. 

There are convergence results for linear IHT expressed in terms of the coherence bound of the sensing matrix and the sparsity level of the unknown vector $\Bx$ in Eq.~\eqref{eq:linear} (see \cite{wang}). Our analysis, using coherence rather than RIP of the sensing matrix, of the nonlinear IHT algorithm shows that it behaves similarly to the linear IHT algorithm under two conditions: (1) the linearization $\Phi_{\Bx_n}$ obeys a coherence bound for all iterations and (2) the error in the linearization is sufficiently small.  

We now state our main result on the convergence of nonlinear IHT using coherence rather than RIP.  The proof is given in Appendix A.  
\begin{theorem}
\label{thm:coherence} 
{\it Let $\By=\Phi(\Bx)+\Beps$ and let $\{\Bx_n\}$ be the sequence generated  by the iteration 
\begin{equation}
\label{eq:nliht}
\Bx_{n+1}=H_s\left(\Bx_{n}+\Phi^*_{\Bx_n}(\By-\Phi_{\Bx_n} \Bx_{n})\right) \ ,
\end{equation}
where $\Phi_{\Bx_n}$ is the linearization of $\Phi$ at $\Bx_n$.  Suppose that $\Phi_{\Bx_n}$ obeys the coherence estimate $\mu(\Phi_{\Bx_n})\le \mu_0$ for all $n\ge1$.  If $\Bx$ is $s$-sparse and $\mu_0<\frac{1}{3s+1}$, then
\begin{equation}
\label{eq:iteration_thm_2.1}
	\| \Bx_n-\Bx \|_1 \le \rho \| \Bx_{n-1} -  \Bx \|_1 + (3s+1) \| \Phi^*_{\Bx_n} \Be_n \|_\infty \ ,
\end{equation}
where $\Be_n = \Phi(\Bx) - \Phi_{\Bx_n}\Bx + \Beps$ and $\rho = \mu_0(3s+1)$.
}
\end{theorem}
We note that the above result reduces to the known linear IHT convergence result when $\Phi$ is linear~\cite{wang}. Iterating \eqref{eq:iteration_thm_2.1}, we obtain for all $n\ge 1$,
\begin{equation} 
\label{eq:iterates}
\|\Bx_n-\Bx\|_1\le \rho^n\|\Bx_0-\Bx\|_1+\sum_{j=0}^n\rho^j (3s+1)  \|\Phi^*_{\Bx_j}\Be_j\|_\infty \ .
\end{equation}
Eq.~\eqref{eq:iterates} illustrates the importance of the error term in Theorem 2.1. Since by hypothesis $\rho < 1$, the condition that $\|\Phi^*_{x_j}e_j\|_\infty$ is bounded for all $j$ leads to a convergent geometric series in \eqref{eq:iterates}, which is summarized in the following corollary.  
\begin{corollary}
Let $\By=\Phi(\Bx)+\Beps$ and $\{\Bx_n\}$ be the sequence generated by \eqref{eq:nliht}. Suppose that $\Bx$ is $s$-sparse and $\rho=\mu_0(3s+1)<1$. If there exists a constant $M$ such that for all $j$, $\|\Phi^*_{\Bx_j}\Be_j\|_\infty<M$, then the error estimate \begin{equation}
\|\Bx_n-\Bx\|_1\le \rho^n\|\Bx_0-\Bx\|_1+\frac{M(3s+1)}{1-\rho}
\end{equation}
holds for all $n\ge 1$.
\end{corollary}
The above results are applied in Section \ref{sec:analysis} to the inverse scattering problem (ISP), which we describe next. 

\subsection{Inverse scattering problem}
\label{sec:scatt}
Here we describe the formulation of the ISP that we require in this paper. For simplicity, we restrict our attention to scalar waves. Our results extend naturally to other types of scattering problems, though some modifications are required \cite{Born99a}.

We begin with the forward scattering problem. We consider the scattering of a time-harmonic scalar wave $u$ from a medium with scattering potential $\eta$, where we assume that $\eta$ is compactly supported in a bounded domain $\Omega$. The field satisfies the Helmholtz equation
\begin{equation}
\label{eq:helmholtz}
\nabla^2 u(\Bx)+k^2 (1+\eta(\Bx)) u(\Bx) = -S(\Bx) \ ,
\end{equation}
where $k$ is the wavenumber and $S(\Bx)$ is the source term.  The total field $u$ can be decomposed into the incident field $u_i$ and scattered field $u_s$ by $u=u_i+u_s$.  The incident field obeys Eq.~\eqref{eq:helmholtz} in the absence of scattering:
\begin{equation}
\nabla^2 u_i(\Bx)+k^2 u_i(\Bx) = -S(\Bx) \ .
\end{equation}
It follows that the scattered field obeys
\begin{equation}
\label{eq:scattered_field}
\nabla^2 u_s(\Bx)+k^2u_s(\Bx) = -\eta(\Bx) u(\Bx) \ ,
\end{equation}
which we supplement with the Sommerfeld radiation condition
\begin{equation}
\label{eq:radiation}
\lim_{r\to\infty} r\left(\frac{\partial u_s}{\partial r}-iku_s  \right)=0 \ ,
\end{equation}
where $r=|\Bx|$.  
Eq.~\eqref{eq:helmholtz} can be converted to an equivalent integral equation by introducing a Green's function. The Green's function $G$ satisfies
\begin{equation}
\nabla^2G(\Bx,\By)+k^2G(\Bx,\By)=-\delta(\Bx-\By) \ ,
\end{equation}
together with the Sommerfeld radiation condition. In three dimensions, $G$ is given by\begin{align}
G(\Bx,\By) =\frac{e^{ik|\Bx-\By|}}{4\pi|\Bx-\By|} \ .
\end{align}
It follows that $u_s$ obeys the integral equation
\begin{equation}
\label{eq:LS}
u_s(\Bx) =  k^2\int_\Omega G(\Bx,\By)\eta(\By)u(\By) d\By \ .
\end{equation}
By iterating the above, we obtain the Born series
\begin{align}
\label{eq:Born_series}
u_s(\Bx) = & k^2\int_\Omega G(\Bx,\By_1)\eta(\By_1)u_i(\By_1) d\By_1 \nonumber\\ & +k^4\int_\Omega\int_\Omega G(\Bx,\By_1)\eta(\By_1)G(\By_1,\By_2)\eta(\By_2)u_i(\By_2) d\By_1 d\By_2+\cdots \ , 
\end{align}
which yields an explicit formula for the scattered field.  Convergence of the Born series for scalar waves is considered in \cite{10.2307/43693726}.  In the case of weak scattering, the scattered field can be well approximated by the first term of  the Born series, which  yields the first Born approximation
\begin{equation}
u_s(\Bx) \approx k^2\int_\Omega G(\Bx,\By)\eta(\By)u_i(\By) d\By \ .
\end{equation} 
Multiple scattering is accounted for by including the higher order order terms in the Born series. 

It will prove useful to express the Born series in operator notation.  We define $\mathcal{G}$ to be the integral operator with kernel $G(\Bx,\By)$, where both $\Bx$ and $\By$ belong to $\Omega$, and define $\tilde{\mathcal{G}}$ to be the integral operator with the same kernel,  but with $\Bx$ outside of $\Omega$.  We also let $\mathcal{V}$ be the operator that corresponds to multipication by $k^2\eta$.  The Born series \eqref{eq:Born_series} thus becomes
\begin{align}
\label{eq:LS_ops}
u_s= & \tilde{\mathcal{G}}\left( \mathcal{V}+\mathcal{V}\mathcal{G}\mathcal{V}+ \mathcal{V}\mathcal{G}\mathcal{V}\mathcal{G}\mathcal{V} + \cdots \right)u_i \ .
\end{align}
By summing the geometric series we obtain
\begin{equation}
\label{eq:cont_for}
u_s = \tilde{\mathcal{G}}(1-\mathcal{V}\mathcal{G})^{-1}\mathcal{V}u_i  \ .
\end{equation} 
Introducing the $T$-matrix operator  which is defined by
\begin{equation}
\mathcal{T}=\mathcal{V}(1-\mathcal{G}\mathcal{V})^{-1}=(1-\mathcal{V}\mathcal{G})^{-1}\mathcal{V} \ ,
\end{equation} 
we find that 
\begin{equation}
u_s = \tilde{\mathcal{G}}\mathcal{T} u_i \ .
\end{equation}
  
The above formulation is quite general and holds for many experimental configurations.  For the remainder of the paper, we consider a specific scattering experiment in which the incident field is a plane wave, and measurements of the scattered field are taken in the far field on a ball of radius $R$ with $kR\gg 1$, as shown in Fig.~\ref{fig:experiment}.
That is, the incident field is given by
\begin{equation}
u_i(\Bx)=e^{ik\Hd\cdot\Bx } \ ,
\end{equation} 
where the unit vector $\Hd$ is the direction of propagation.
In the far field, the Green's function is asymptotically of the form
\begin{equation}
\frac{e^{ik|\Bx-\By|}}{4\pi|\Bx-\By|} \sim \frac{e^{ikR}}{4\pi R}e^{-ik \Hx \cdot \By} \ .
\end{equation}
Eq.~\eqref{eq:LS} then becomes
\begin{equation}
u_s(\Bx)\sim A(\Hx)\frac{e^{ikR}}{4\pi R} \ ,
\end{equation}
where the scattering amplitude $A$ is defined by
\begin{equation}
\label{eq:ampscat}
A(\Hx)= k^2\int_\Omega e^{-ik \Hx \cdot \By}\eta(\By)u(\By)d\By \ .
\end{equation}
Thus the scattered field behaves as an outgoing spherical wave with amplitude $A$.

The goal of the ISP is to reconstruct $\eta$ from measurements of $A$. To proceed, we suppose that $N_d$ measurements of the scattering amplitude in the directions $\Hx_1,\dots,\Hx_{N_d}$ are acquired for each of $N_s$ incident plane waves in the directions $\Hx_1,\dots,\Hx_{N_s}$.  We assemble these measurements in a $N_d\times N_s$ matrix that we denote by $Y$. We also discretize $\Omega$ into $N$ voxels of volume $h^3$ with centers $\Br_1,\dots,\Br_N$. The discretized version of \eqref{eq:cont_for} is then of the form
\begin{equation}
\label{eq:disc_forward}
Y=A(I-V\Gamma)^{-1}VB \ .
\end{equation}
Here $A$ is an $N_d\times N$ matrix, $B$ is an $N\times N_s$ matrix, $\Gamma$ is an $N\times N$ matrix, and $V$ is a $N\times N$ diagonal matrix.  These matrices are defined as 
\begin{align}
\label{eq:disc_ex}
&A_{mn}=e^{-ik \Hx_m \cdot \Br_m} \ , \\
&B_{mn}=e^{ik \Hd_n \cdot \Br_m} \ , \\
& \Gamma_{mn}=(1-\delta_{mn})G(\Br_m,\Br_n) \ , \\
& V_{mn} = \delta_{mn}k^2h^3\eta(\Br_m) \ .
\end{align}  

It is important to note that the dependence of $Y$ on $V$ is nonlinear. This can be seen by expanding the matrix inverse in \eqref{eq:disc_forward} to obtain
\begin{equation}
Y=A(I-V\Gamma)^{-1}VB = A\left(V+V\Gamma V+ V\Gamma V\Gamma V+\cdots\right)B \ ,
\end{equation}
which converges when $\|V\Gamma\|<1$. The first Born approximation corresponds to keeping only the first term in the series, so that \eqref{eq:disc_forward} becomes
\begin{equation}
\label{eq:AVB}
Y=AVB \ .
\end{equation}
Likewise, the $M$th Born approximation is given by
\begin{equation}
\label{eq:AVB_2}
Y =	A \left(\sum_{m=0}^{M-1} (V\Gamma)^m \right)V B  \ .
\end{equation}
Note that we can also introduce the $N\times N$ discretized $T$-matrix, which is defined by $T_{mn}=\mathcal{T}(\Br_m,\Br_n)$. Eq.~\eqref{eq:disc_forward} can then be written as
\begin{equation}
\label{eq:ATB}
Y = ATB \ .
\end{equation}

In the above discrete setting, the ISP consists of recovering the matrix $V$ from measurements of $Y$.   Our goal is to use suitably modified variants of IHT to solve the ISP when we assume that the desired matrix $V$ is sparse.   That is, we cast the ISP as a sparse signal recovery problem and use IHT to solve it.  Depending on the form of the approximation for the ISP, we use either linear or nonlinear IHT.  
   In the first Born approximation, this consists of solving the linear system \eqref{eq:AVB}. We solve this problem by making use of the linear IHT algorithm \eqref{eq:iht}. In the multiple scattering regime, we solve either \eqref{eq:AVB_2} or \eqref{eq:ATB} using  the nonlinear IHT algorithm \eqref{eq:iht_nl}.  
\begin{figure}[t]
\centering
\includegraphics[width=0.5\linewidth]{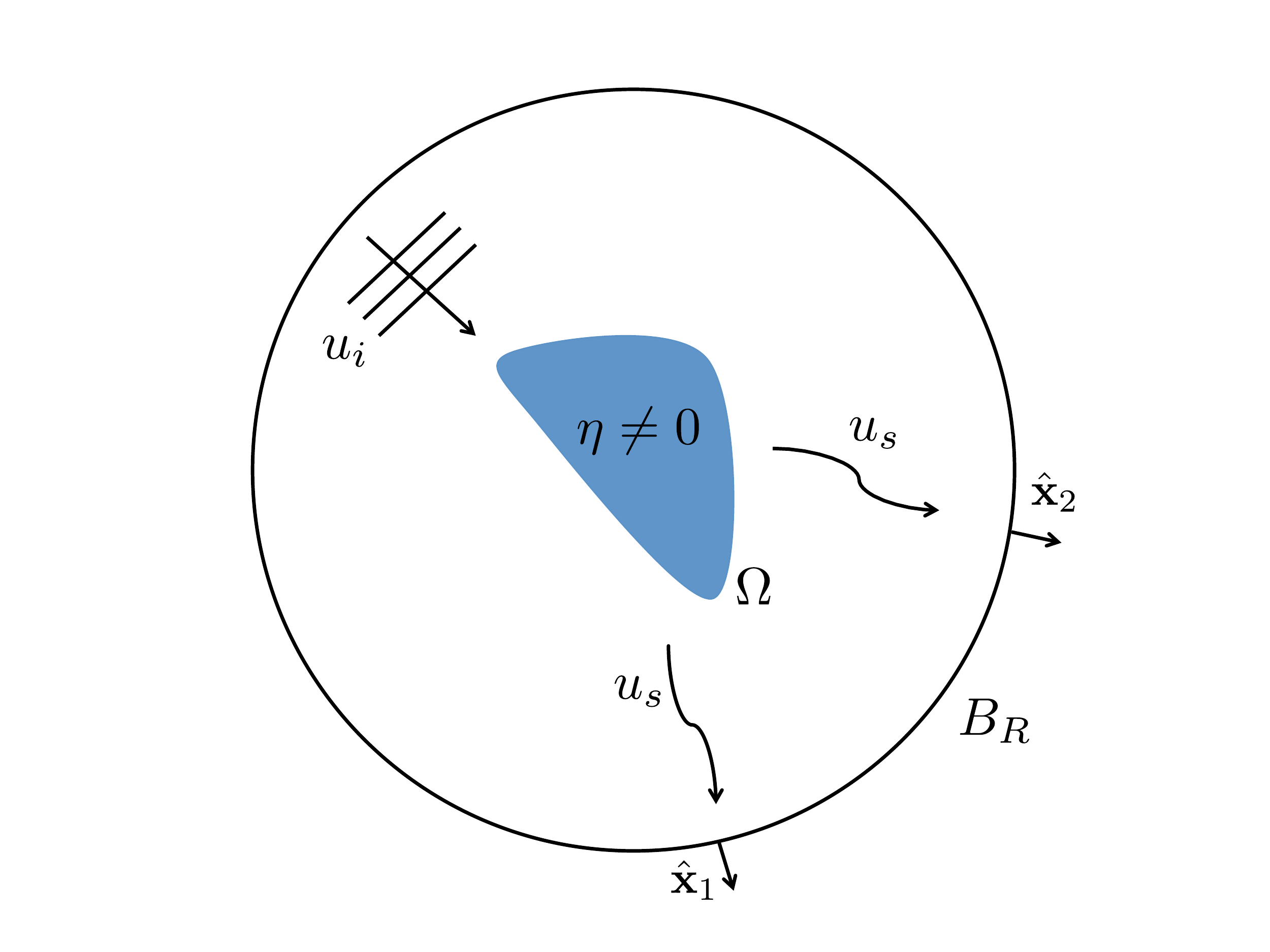}
\caption{Illustrating the scattering experiment.}
\label{fig:experiment}
\end{figure}

\section{Inverse Scattering and Iterative Hard Thresholding}
\label{sec:invscatteringIHT}

\subsection{Linear inverse problem}
\label{sec:invscatteringIHT_lin}
We now apply the linear IHT algorithm to the ISP in the first Born approximation.  Hence, we consider using the linearized forward equation \eqref{eq:AVB}, which is valid in weakly scattering regimes.  This equation, $AVB=Y$ where $V$ is diagonal, can be converted to a standard linear equation in vector form as
\begin{equation}
\label{eq:kv}
\Phi \Bv=\By \ ,
\end{equation}
where $\By$ is the $(N_sN_d)$-dimensional vector formed by stacking the columns of $Y$, $\Bv$ is the $N$-dimensional vector containing the diagonal entries of $V$, and $\Phi$ is the $ (N_sN_d)\times N$ matrix that is defined  by
\begin{equation}
\label{eq:Kdef}
\Phi_{(mn),j}=A_{mj}B_{jn} \ .
\end{equation}
Here $(mn)$ is a composite index where $1\le m\le N_d$ and $1\le n\le N_s$.

One only considers using the IHT algorithm if $\Bv$ is sparse.  Note that this corresponds to the diagonal matrix $V=D_\Bv$ being sparse along its diagonal. 
 Eq.~\eqref{eq:kv} can be solved by a straightforward application of the IHT algorithm \eqref{eq:iht}, but for our purposes, we are interested in applying the algorithm directly to the matrix equation \eqref{eq:AVB}.  While the two linearized forward equations are identical (as are the resulting formulations of IHT), this is not an empty exercise.  Using the full matrix form in \eqref{eq:AVB} yields a more tractable form of IHT, especially as we move to the nonlinear formulation of IHT and its analysis in the next section.  Additionally, and perhaps more importantly, these insights lead to computational advantages.   Fundamentally, these benefits come from treating the sensing matrix $\Phi$ for inverse scattering problems as two smaller separate sensing matrices.  Similar observations are made in \cite{Tmatrix1,hadamard}.     

To derive the matrix form of IHT for the linear ISP, we first note that 
\begin{align}
\label{eq:phiphi}
(\Phi ^*\Phi )_{jk} = & \sum_{\ell=1}^{N_sN_d} \Phi ^*_{j\ell}\Phi _{\ell k} 
=  \sum_{m=1}^{N_d}\sum_{n=1}^{N_s}\Phi ^*_{j,(mn)}\Phi _{(mn),k} \nonumber\\ 
= & \sum_{m=1}^{N_d} A^*_{jm}A_{mj} \sum_{n=1}^{N_s} B_{kn}B^*_{nj}
=  (A^*A)_{jk}(BB^*)_{kj} \ .
\end{align}
Thus $\Phi ^*\Phi =(A^*A)\circ(BB^*)^T$, where \ $\circ$ \ denotes the Hadamard (entrywise) product of two matrices.

To formulate IHT in terms of the diagonal matrix $V_n$ instead of the vector $\Bv_n$, we take advantage of the natural interplay between Hadamard products and diagonal matrices, namely 
\begin{equation}
\label{eq:had}
((A\circ B^T)\Bx)_j=(AD_{\Bx}B)_{jj} \ .
\end{equation}
Combining \eqref{eq:phiphi} and \eqref{eq:had}, the matrix product $\Phi ^*\Phi \Bv$ in \eqref{eq:iht} can be rewritten as 
\begin{equation}
(\Phi ^*\Phi \Bv)_j=\left(\left((A^*A)\circ(BB^*)^T\right)\Bv\right)_j=\left( A^*AVBB^* \right)_{jj} \ .
\end{equation}
Similarly, for the matrix vector product $\Phi^*y$, we have 
\begin{equation}
(\Phi ^*\By)_j = \sum_{m=1}^{N_d}\sum_{n=1}^{N_s}A^*_{jm}\By_{(mn)}B^*_{nj} = (A^*Y B^*)_{jj } \ .
\end{equation}
We introduce the linear ``diagonalizing'' operator $\mathcal{D}$ that acts on square matrices according to
\begin{equation}
\label{eq:Ddef}
\mathcal{D}(M)_{mn} := \delta_{mn}M_{mn} \ ,
\end{equation}
and can write the IHT algorithm of matrix equation \eqref{eq:AVB} by 
\begin{equation}
\label{eq:ihtmatrix}
V_{n+1}=H_s^\mathcal{D}\Big(\mathcal{D}\big(V_n+A^*\left(Y-AV_nB \right)B^*\big)\Big) \ ,
\end{equation}
where the nonlinear thresholding operator $H_s^\mathcal{D}$ only acts on the diagonal entries of the resulting matrix.

It is instructive to compare the matrix formulation of IHT in \eqref{eq:ihtmatrix} with the standard form of IHT in \eqref{eq:iht}. Recall that the dimensions of the matrices $\Phi, A$, and $B$ are $(N_sN_d)\times N, \  N_d\times N$, and $N\times N_s$ respectively.  Thus the computational complexity of IHT in the form \eqref{eq:iht} is $O(N_sN_dN)$.  This result is obtained by carrying out two sequential matrix-vector products. In contrast, naively forming the matrix $\Phi^*\Phi$ requires $O(N_sN_dN^2)$ operations. However, in matrix form \eqref{eq:ihtmatrix}, the matrix multiplications have complexity $O(\max\{N_dN^2,N_sN^2\})$. Depending on the relative values of $N_s,N_d,$ and $N$, this approach can be computationally advantageous.  

We note that the general framework of treating the fundamental unknown as a matrix instead of a vector provides flexibility for algorithmic developments. The unknown $V$, while known to be diagonal, has nonzero off-diagonal terms before the diagonalizing operator $\mathcal{D}$ acts during each iteration.  \cite{Tmatrix1} explores a similar algorithm with such flexibility. The definition used in \eqref{eq:Ddef} is a particular choice, and there may indeed be better choices for defining the diagonalizing operator.

\subsection{Nonlinear inverse problem}
\label{sec:invscatteringIHT_nlin}
We now consider nonlinear IHT, which applies the ISP in the multiple-scattering regime.  Returning to the forward equation $Y=A(I-V\Gamma)^{-1}VB$, we add additional corrections beyond the first Born approximation via the Born series.  
Our goal is to develop an IHT algorithm for Eq.~\eqref{eq:AVB_2}. 
 
The derivation of nonlinear IHT follows by making the transformation
\begin{equation}
\tilde{A}= A\left(\sum_{m=0}^{M-1}(V\Gamma)^m \right)  
\end{equation}
in the linear IHT algorithm.
Note that we could also rewrite the forward equation as 
\begin{equation}
AV\left(\sum_{m=0}^{M-1}(\Gamma V)^m \right)B=Y \ ,
\end{equation} 
and then make the substitution
\begin{equation}
\tilde{B}=\left(\sum_{m=0}^{M-1}(\Gamma V)^m \right)B \ .
\end{equation} 
This formulation is justified by the reciprocity of sources and detectors, which allows for interchanging $A$ and $B$, while leaving the forward equation unchanged (with the data matrix $Y$ becoming $Y^T$). Note that we can still think of having one larger sensing matrix $\Phi_\Bv$ (which analogous to Eq.~\eqref{eq:iht_nl} is the linearization of $\Phi$ at $V=D_\Bv$), which obeys the equation
\begin{equation}
\label{eq:phiphinl}
\Phi^*_\Bv\Phi_\Bv = (\tilde{A}^*\tilde{A})\circ(BB^*)^T = (A^*A)\circ(\tilde{B}\tilde{B}^*)^T \ .
\end{equation}

Making either of the above substitutions, we obtain the nonlinear IHT algorithm in the form\begin{align}
\label{eq:nlihtisp}
 V_{n+1}=H_s^\mathcal{D}\Big(\mathcal{D}\big(V_n+\tilde{A}_n^*\left(\Phi-\tilde{A}_nV_nB \right)B^*\big)\Big)\ , \quad \tilde{A}_n=A\left(\sum_{m=0}^M(V_n\Gamma)^m \right)\ ,
\end{align}
or alternatively
\begin{align}
\label{eq:nlihtisp2}
 V_{n+1}=H_s^\mathcal{D}\Big(\mathcal{D}\big(V_n+A^*\left(\Phi-AV_n\tilde{B}_n) \right)\tilde{B}_n^*\big)\Big) \ ,  \quad \tilde{B}_n=\left(\sum_{m=0}^M(\Gamma V_n)^m \right)B \ . 
\end{align}
This algorithm is summarized below as Algorithm \ref{algorithm:1}.
As $M\to\infty$, we can replace $\tilde{A}$ and $\tilde{B}$ in Eqs.~\eqref{eq:nlihtisp} and \eqref{eq:nlihtisp} with
\begin{equation}
\tilde{A}=A(I-V_n\Gamma)^{-1} \qquad\qquad \tilde{B}=(I-\Gamma V_n)^{-1}B \ ,
\end{equation}
or this fully nonlinear IHT algorithm can be written in T-matrix form
\begin{equation}
\label{eq:nlfull}
V_{n+1}=H_s^\mathcal{D}\Big(\mathcal{D}\big(V_n+V_n^{-1*}T_n^*A^*\left(\Phi-AT_nB \right)B^*\big)\Big) \ .
\end{equation}
Note that $V_n^{-1}$ is not well-defined since it is sparse along its diagonal. We thus compute $V_n^{-1}$ only for the nonzero entries of $V_n$, setting all other matrix elements of $V_n^{-1}$ to zero.

\RestyleAlgo{boxruled}
          \begin{algorithm}
     \caption{Nonlinear Iterative Hard Thresholding Algorithm}
     \label{algorithm:1}
     \DontPrintSemicolon
     \KwIn{$A,B,\Gamma, Y$} 
     \KwIn{$M, s$} 
     \KwOut{$V$} 
     \BlankLine
     Initialize $V_0$, $n=0$
     \BlankLine
     \While{stopping criteria not met}{
    $\tilde{A}_n=A\left(\sum_{m=0}^{M-1} (V_n \Gamma)^m \right)$\;
    $X = \mathcal{D}\left(  V_n+\tilde{A}_n^*\left(Y-\tilde{A}_nV_nB \right)B^*   \right) $\; 
    $V_{n+1}=H^{\mathcal{D}}_s(X)$\; 
    $n\gets n+1$\;
    }
     \end{algorithm}

We now remark on the computational complexity of the nonlinear IHT algorithm.  Since $V$ is $s$-sparse along its diagonal, the product $\Gamma V$ will be block sparse (up to permutation of rows), with only $s$ nonzero rows. Thus, the multiplication $(\Gamma V)^m$ requires $O(Nms^2)$ operations, and the subsequent formation of $\tilde{A}$ requires $O(sN_dN)$. Done in correct order, the remaining matrix products again require  $O(\max\{N_dN^2,N_sN^2\})$ operations.  Note that in the case when $M=\infty$, due to the sparseness of $V$, computation of the $T$-matrix has only complexity $O(s^3)$.

\section{Coherence Estimates for Inverse Scattering}
\label{sec:4}

In this section we apply Theorem \ref{thm:coherence} on convergence of the nonlinear IHT algorithm to the ISP. We begin by restating Theorem \ref{thm:coherence} using the notation from Section \ref{sec:invscatteringIHT}. 

\begin{theorem} 
\label{cor:thm2.1}
Let $Y=A(I-V\Gamma)^{-1}VB+E$ and $\{V_n\}$ be the sequence of diagonal matrices generated  by the nonlinear IHT iteration  
\begin{equation}
\label{eq:ref_scat_its}
V_{n+1}=H_s^\mathcal{D}\Big(\mathcal{D}\big(V_n+\tilde{A}_n^*\left(Y-\tilde{A}_nV_nB \right)B^*\big)\Big) \ ,
\end{equation}
where $\tilde{A}_n=A\left(\sum_{m=0}^{M-1}(V_n\Gamma)^m \right)$.  Let $\Phi_{\Bv_n}$ be the linearization about $V_n$ corresponding to $\tilde{A}_n$.  Assume that for all $n\ge1$, the coherence of $\Phi_{\Bv_n}$ is bounded by $\mu_0$.  If $V$ is $s$-sparse along its diagonal, then
\begin{equation}
	\| \Bv_n-\Bv \|_1 \le \rho \| \Bv_{n-1} -  \Bv \|_1 + (3s+1) \| \mathcal{D}[\tilde{A}^*  E_n B^*] \|_\infty ,
\end{equation}
where $E_n = A(I-V\Gamma)^{-1}VB-\tilde{A}_nVB+E$ and $\rho = \mu_0(3s+1)$.
\end{theorem}

To proceed, we must obtain bounds on the coherence suitable for use in the nonlinear IHT algorithm.  We investigate two specific cases: double scattering ($M=2$) and multiple scattering ($M=\infty$). 

\subsection{Mutual coherence of sensing matrices}
\label{sec:mut_coh}

For the linearized problem, it follows from Eq.~\eqref{eq:phiphi} that 
normalizing the columns of $\Phi$ is equivalent to normalizing the columns of both $A$ and $B^*$.  We note that normalizing the column $\Phi_j$ is performed by dividing all of its entries by the square root of 
\begin{equation}
\|\Phi_j\|_2^2=\sum_{i=1}^{N_sN_d} |\Phi_{ij}|^2 = \sum_{m=1}^{N_d}\sum_{n=1}^{N_s} |A_{mj}|^2 |B_{jn}|^2=\|A_j\|_2^2\|B^*_j\|_2^2 \ .
\end{equation}
Since we have normalized the columns of $A$ and $B^*$ (as well as $\Phi$) and because the coherence of a column-normalized matrix is the largest off-diagonal entry in its Gram matrix, we conclude that
\begin{equation}
\label{eq:mu_lin}
\mu(\Phi) \le \mu(A)\mu(B^*) \ ,
\end{equation}
where we have used the fact that $\Phi ^*\Phi =(A^*A)\circ(BB^*)^T$.
To extend this result to the nonlinear case, we use that fact that Eq.~\eqref{eq:phiphinl} implies
\begin{equation}
\mu(\Phi_\Bv )\le\min\{\mu(A)\mu(\tilde{B}^*),\mu(\tilde{A})\mu(B^*)\} \ ,
\end{equation}  
where $\tilde{A}$ and $\tilde{B}$ are defined as in Eqs.~\eqref{eq:nlihtisp} and \eqref{eq:nlihtisp2}.  
Since the mutual coherence is bounded above by 1, we have the immediate (pessimistic) bound that for any number of nonlinear terms $M$, 
\begin{equation}
\label{eq:bound}
\mu(\Phi_\Bv )\le\min\{\mu(A),\mu(B^*)\} \ .
\end{equation}
The above bound is most likely not optimal.  However, it can be used to easily compute (or accurately estimate) the coherence of higher order linearizations based solely on the geometry of the experiment. Moreover, Eq.~\eqref{eq:bound} is applicable to the nonlinear problem.

For the remainder of the paper, we assume that $A$ and $B^*$ have been column-normalized. We require the following result on the coherence of products of matrices. 
\begin{proposition}
\label{prop:ABcoh}
Let $A$ be an $(L\times N)$ matrix with normalized columns and $H$ be an $(N \times P)$ matrix.  Then the mutual coherence of the product $AH$ obeys the following estimate:
 \begin{equation}
 \mu(AH)\le \frac{\mu(H) + N\mu(A)}{ \left|1 - N\mu(A) \right|}
 \end{equation}
\end{proposition}
\begin{proof}
We first compute an upper bound on $|(H^*A^*AH)_{ij}|$.
\begin{align}
|(H^*A^*AH)_{ij}|= & \left|\sum_{k,\ell,m} H^*_{ik}A^*_{k\ell}A_{\ell m}H_{mj}\right| =  \left|\sum_{k,m} H^*_{ik} H_{mj} \sum_{\ell} A^*_{k\ell}A_{\ell m}\right|\\
\le &  \left|\sum_{k} H^*_{ik} H_{kj} \sum_{\ell} A^*_{k\ell}A_{\ell k}\right|  + \left|\sum_{k,  m\neq k} H^*_{ik} H_{mj} \sum_{\ell} A^*_{k\ell}A_{\ell m} \right|\label{sub-eq-1:1} \\
\le & |\langle H_j, H_i\rangle| + \mu(A) \sum_{k}\left| H^*_{ik}\right| \sum_{m\neq k} \left|H_{mj}\right| \label{sub-eq-1:2} \\ 
\le &  |\langle H_j, H_i\rangle| + \mu(A)\|H\|_1^2 \ , 
\end{align}
where \eqref{sub-eq-1:2} is obtained from \eqref{sub-eq-1:1} by replacing the inner products between columns of $H$ and $A$ respectively with the upper bound on the coherence.   
We now compute a lower bound on $|(H^*A^*AH)_{ii}|$.
\begin{align}
|(H^*A^*AH)_{ii}|= & \left|\sum_{k,\ell,m} H^*_{ik}A^*_{k\ell}A_{\ell m}H_{mi}\right| \\
= &  \left|\sum_{k} H^*_{ik} H_{ki} \sum_{\ell} A^*_{k\ell}A_{\ell k} + \sum_{k\neq m} H^*_{ik} H_{mi} \sum_{\ell} A^*_{k\ell}A_{\ell m} \right|\\
\ge & \left|\langle H_i, H_i\rangle - \mu(A) \sum_{k\neq m} H^*_{ik} H_{mi}\right| \label{sub-eq-1:3}\\
\ge &  \left|\langle H_i, H_i\rangle - \mu(A)\|H\|_1^2 \right| \ .
\end{align}
Now, by definition of coherence,
\begin{align*}
\mu(AH)= & \max_{i\neq j}\frac{|(H^*A^*AH)_{ij}|}{\sqrt{(H^*A^*AH)_{ii}(H^*A^*AH)_{jj}}} \\
\le & \frac{|\langle H_j, H_i\rangle| + \mu(A)\|H\|_1^2}{\sqrt{\left|\langle H_i, H_i\rangle - \mu(A) \|H\|_1^2\right|\left|\langle H_j, H_j\rangle - \mu(A) \|H\|_1^2\right|}} \\
\le & \frac{\mu(H) + \mu(A)\frac{\|H\|_1^2}{\sqrt{\langle H_i,H_i \rangle\langle H_j,H_j \rangle}}}{\sqrt{\left|	1 - \mu(A) \frac{\|H\|_1^2}{\langle H_i,H_i \rangle}\right|\left|1 - \mu(A) \frac{\|H\|_1^2}{\langle H_j,H_j \rangle}\right|}} \\
\le & \frac{\mu(H) + N\mu(A)}{|1-N\mu(A)|} \ , 
\end{align*}
where to obtain the final result we have used inequality
\begin{equation*}
\|\Bx\|_1\le \sqrt{N}\|\Bx\|_2 \ .
\end{equation*}
\end{proof}

The following corollary will prove to be of use.
\begin{corollary}
\label{cor:muAH}
Let $A$ be an $(L\times N)$ matrix  with normalized columns, and $H$ be an $(N \times P)$ matrix.  Furthermore, assume that $H$ has only $s$ nonzero entries in any column.   Then the mutual coherence of the product $AH$ obeys the following bound:
 \begin{equation}
 \label{eq:muAH}
 \mu(AH)\le \frac{\mu(H) + s\mu(A)}{ \left|1 - s\mu(A) \right|}
 \end{equation}
\end{corollary}

The simplest case of interest arises when $H=I+VG$ and $V$ is $s$-sparse along its diagonal. Suppose that the first $s$ entries of $V$ are nonzero. Then
\begin{equation}
I+VG = \begin{pmatrix}
1 & v_1 G_{12} & v_1G_{13} & \cdots & \cdots &\cdots & v_1G_{1N} \\
v_2 G_{21} & 1 & v_2G_{23} & \cdots & \cdots&\cdots & v_2G_{2N} \\
\vdots & & \ddots & & & &  \vdots \\
v_s G_{s1} & \cdots & v_sG_{s,s-1} & 1 & \cdots & \cdots & v_sG_{sN} \\
0 & 0 & \cdots& 0  & 1 & 0 & 0  \\
\vdots & &  & &  & \ddots \\
0 & \cdots & \cdots & \cdots & \cdots &0 & 1 
\end{pmatrix} \ .
\end{equation}
It is clear that the maximum value required for the coherence calculation is given by choosing two columns $H_i, H_j$ where $1\le i,j\le s$.  If we let $\delta=\|VG\|_{\max}$, an upper bound on the coherence can be explicitly computed to be
\begin{equation}
\mu(I+VG) \le \frac{2\delta+(s-2)\delta^2}{1+(s-1)\delta} \ . 
\end{equation}
Combining this bound with Corollary \ref{cor:muAH} yields the following result.
\begin{corollary}
\label{cor:muAH2}
Let $A$ be an $(L\times N)$ matrix  with normalized columns.  Let $V$ be a diagonal $(N\times N)$ matrix, with only $s$ nonzero entries along its diagonal.  Let $G$ be an $(N\times N)$ matrix with zeros along its diagonal, and let  $\delta=\|VG\|_{\max}<1$. Then the mutual coherence of the product $A(I+VG)$ obeys the following bound:
 \begin{equation}
 \mu(A(I+VG))\le \frac{ \frac{2\delta+(s-2)\delta^2}{1+(s-1)\delta}+ (s+1)\mu(A)}{ \left|1 - (s+1)\mu(A) \right|} \ .
 \end{equation}
\end{corollary}

We note that analogous results can be derived when $M>1$, but will not be used in this paper.

\subsection{Application to second Born approximation ($M=2$)}
\label{sec:analysis}
We now apply Theorem~\ref{thm:coherence} to the case $M=2$, which corresponds to double scattering, also known as the second Born approximation. Here we directly apply Corollary \ref{cor:muAH2}.  We then obtain the following estimate on the coherence: 
\begin{equation}
\mu(\Phi_{\Bv_n})\le\mu(B^*)\mu(A(I+V_nG))\le\mu(B^*)\frac{ \frac{2\delta_n+(s-2)\delta_n^2}{1+(s-1)\delta_n}+ (s+1)\mu(A)}{ \left|1 - (s+1)\mu(A) \right|} \ ,
\end{equation}  
where we require the third term to be less than $1/(3s+1)$ for all iterations $n$ to guarantee convergence.  

The above analysis ignores the behavior of the error term $\|\Phi^*_{\Bv_n}\Be_n\|_\infty$ in Theorem \ref{thm:coherence}, where $\Be_n=\Phi(\Bv)-\Phi_{\Bv_n}\Bv+\Beps$.  Next, we investigate the behavior of error for $\Phi^*_{\Bv_n}$ as defined by~\cref{sec:scatt}.
To proceed, we note that the error term can be split into three parts as 
\begin{align}
\label{eq:3errs}
\|\Phi^*_{\Bv_n}(\Phi(\Bv)-\Phi_{\Bv_n}\Bv+\epsilon)\|_\infty = & \|\Phi^*_{\Bv_n}(\Phi(\Bv)-\Phi_{\Bv}\Bv+\Phi_{\Bv}\Bv-\Phi_{\Bv_n}\Bv+\epsilon)\|_\infty \\ \nonumber
\le & \|\Phi^*_{\Bv_n}(\Phi(\Bv)-\Phi_{\Bv}\Bv)\|_\infty+ \|\Phi^*_{\Bv_n}(\Phi_{\Bv}\Bv-\Phi_{\Bv_n}\Bv)\|_\infty \\ \nonumber
 & + \|\Phi^*_{\Bv_n}\Beps\|_\infty \ .
\end{align}
The first term measures the error of the linearization of the forward problem, the second term measures the difference between the linearizations at the true solution and the current iteration, and the last term is the error in the measurement system.  Note that this last term can account for the error if the underlying model is not $s$-sparse. 
Let us examine the first term and make use of the definitions of $\Phi$ and $\Phi_{\Bv}$.
We thus obtain
\begin{align*}
\|\Phi^*_{\Bv_n}(\Phi(\Bv)-\Phi_{\Bv}\Bv)\|_\infty = & \|\mathcal{D}\big[(I+V_n\Gamma)^*A^*A[(I- V\Gamma)^{-1}-(I+V\Gamma)]VBB^*]\|_\infty \\
= & \|\mathcal{D}\big[(I+ V_n\Gamma)^*A^*A[(V\Gamma )^2+\cdots]VBB^*\big]\|_\infty \\
\le & \|\mathcal{D}\big[(I+ V_n\Gamma)^*A^*A[(V\Gamma )^2]VBB^*\big]\|_\infty +\cdots \\
=& \|((I+ V_n\Gamma)^*A^*A(V\Gamma)^2)\circ \big(BB^*\big)^T \ \Bv \| _\infty +\cdots
\end{align*}
Each of the terms in the series is of the form $\|(Y^{(j)}\circ Z^T) \Bv\|_\infty$ for $j\ge2$, where 
\begin{equation}
Y^{(j)}=(I+ V_n\Gamma)^*A^*A(V\Gamma)^j \ , \qquad Z=BB^* \ .
\end{equation}
Letting $y^{(j)}=(Y^{(j)}\circ Z^T)\Bv$, we obtain the bound 
\begin{align}
|y^{(j)}_{k}|= & \left|\sum_{\ell=1}^N (Y^{(j)})_{k\ell}(Z)_{\ell k}(\Bv)_\ell\right| \\
\le & \|Y^{(j)}\|_{\max}\|Z\|_{1}\|x\|_{\infty} \ .
\end{align}
The largest entry of $Y^{(j)}$ can be further bounded by 
\begin{align}
Y^{(j)}_{ij} \le & \|A^*AV\Gamma\|_{\max}\|(I+V_n\Gamma)\|_1\|V\Gamma\|_1^{j-1} \\
\le & (1+\gamma_n)(1+(s-1)\mu(A))\gamma^{j-1}\delta \ ,
\end{align}
where $\delta=\|V\Gamma\|_{\rm max}$ and $\gamma=\|V\Gamma\|_1$, with analogous definitions for $\delta_n$ and $\gamma_n$.  
Because $B$ has ones along the diagonal and all off diagonal terms smaller than $\mu(B)$, we can bound the error term by 
\begin{align}
\|(Y^{(j)}\circ Z^T) \Bv\|_\infty\le (1+\gamma_n)(1+(s-1)\mu(A))(1+(s-1)\mu(B))\gamma^{j-1}\delta\|\Bv\|_\infty. 
\end{align}
To bound the total error, we sum the above geometric series, assuming $s\delta<1$, with the result
\begin{align}
\|\Phi^*_{\Bv_n}(\Phi(\Bv)-\Phi_{\Bv}\Bv)\|_\infty \le &(1+\gamma_n)(1+(s-1)\mu(A))(1+(k-1)\mu(B))\delta\|\Bv\|_\infty\sum_{j=2}^{\infty} \gamma^{j-1} \\
= & (1+\gamma_n)(1+(s-1)\mu(A))(1+(s-1)\mu(B)) \frac{\delta\gamma}{1-\gamma}\|\Bv\|_\infty \ .
\end{align}

\noindent Now turning to the second term in the error 
and using \eqref{eq:had}, we find that
\begin{align}
 &\|\mathcal{D}\big[A^*AV[(I+\Gamma V)-(I+\Gamma V_n)]\big]BB^*(I+\Gamma V_n)^*\|_\infty \\
 & \hspace{3cm} =  \|\mathcal{D}\big[A^*AV\Gamma (V-V_n)BB^*(I+\Gamma V_n)^*\big]\|_\infty \\
& \hspace{3cm} = \|(A^*AV\Gamma)\circ \big(BB^*(I+\Gamma V_n)^*\big)^T \ (\Bv-\Bv_n) \|_\infty
\end{align}
Once again we use $Y=A^*AV\Gamma$ and $Z=BB^*(I+\Gamma V_n)^*$, and thus obtain\begin{equation}
\|(Y\circ Z^T) (\Bv-\Bv_n)\|_\infty\le\|(Y\circ Z^T)\|_{\max}\|(\Bv-\Bv_n)\|_1  \ .
\end{equation}
We compute the max norm of $Y\circ Z^T$ by comparing the maximum diagonal and maximum off-diagonal entries.  On the diagonal, $Y_{jj}\le s \delta \mu(A)$, since the diagonal of $V\Gamma$ vanishes.  Likewise, $Z_{jj}\le1+s\delta_n \mu(B)$.  The off-diagonal calculations are identical to those needed for bounding the first term of the error.  Therefore, for $j\neq k$, we have
\begin{align}
&Y_{jk}\le \delta(1+(s-1)\mu(A)) \\
&Z_{jk}\le \mu(B)+\delta_n(1+(s-1)\mu(B)) \ .
\end{align}
Since $\mu(A),\mu(B),\delta,\delta_n\le 1$, we have 
\begin{equation}
\|\Phi^*_{\Bv_n}(\Phi_{\Bv}\Bv-\Phi_{\Bv_n}\Bv)\|_\infty \le s\delta\mu(A)(1+s\delta_n\mu(B))\|\Bv-\Bv_n\|_{1} \ .
\end{equation}

We can now restate Theorem \ref{thm:coherence} for the case of the second Born approximation.
\begin{theorem}\label{thm:nonlinearIHT2ndBorn}
Let $Y=A(I-V\Gamma)^{-1}VB+E$ and let $\{V_n\}$ be the sequence of diagonal matrices generated  by the nonlinear IHT iteration
\begin{equation}
V_{n+1}=H_s^\mathcal{D}\Big(\mathcal{D}\big(V_n+\tilde{A}_n^*\left(Y-\tilde{A}_nV_nB \right)B^*\big)\Big) \ ,
\end{equation}
where $\tilde{A}_n=A\left(I+V_n\Gamma\right)$.  Let $\gamma=\|\Gamma V\|_1<1$, $\delta=\|\Gamma V\|_{\max}$, and $\delta_n=\|\Gamma V_n\|_{\max}$.  If $V$ is $s$-sparse along its diagonal, then
 for all $n\ge 1$, 
\begin{align}
\label{eq:bound1}
\|\Bv_n-\Bv\|_1 &\le 
\rho^{(1)}_n\|\Bv_{n-1}-\Bv\|_1 \nonumber\\
   &+  \frac{\delta\gamma(1+\gamma_n)(3s+1)(1+(s-1)\mu(A))(1+(s-1)\mu(B^*))  }{(1-\rho^{(1)})(1-\gamma)}\|\Bv\|_\infty \nonumber \\
   &+ \|\Phi_{\Bv_n}\epsilon\|_\infty \ ,
\end{align}
where 
\begin{align}
\label{eq:mu1A}
&\rho^{(1)}_n=(3s+1)\left((\mu^{(1)}(A)\mu(B^*)+s\delta\mu(A)(1+s\delta_n \mu(B^*)\right) \nonumber \\
\text{ and     } \hspace{1cm} & \mu^{(1)}(A) = \min\left\{\frac{ \frac{2\delta_n+(s-2)\delta_n^2}{1+(s-1)\delta_n}+ (s+1)\mu(A)}{ \left|1 - (s+1)\mu(A) \right|},1\right\} \ .
\end{align}
\end{theorem}
Note that when the coefficient $\rho^{(1)}_n$ is less than one for all $n$, the above algorithm exhibits linear convergence.  
We also note that similar analysis can be performed for linear IHT under the first Born approximation.   Here, the second error term in Eq.~\eqref{eq:3errs} disappears, and the first error term acquires an additional term in the resulting sum.  
\begin{corollary}
Thus, we obtain the error estimate for linear IHT
\begin{align}
\label{eq:bound2}
\|\Bv_n-\Bv\|_1 \le &  
	\mu(A)\mu(B^*)(3s+1)\|\Bv_{n-1}-\Bv\|_1 \\
   &+  \frac{\delta(1+\gamma_n)(3s+1)(1+(s-1)\mu(A))(1+(s-1)\mu(B^*))  }{(1-\rho^{(1)})(1-\gamma)}\|\Bv\|_\infty + \|\Phi_{\Bv}\Beps\|_\infty \nonumber .
\end{align}
\end{corollary}

\subsection{Application to multiple scattering ($M=\infty$)}
We now consider the nonlinear version of the IHT algorithm.  As before, we assume without loss of generality that $\mu(B^*)\le\mu(A)$.  We note that the linear error term in Eq.~\eqref{eq:3errs} is not present.  Therefore, we only need to estimate the remaining terms, which are of the form
\begin{equation}
\label{eq:err2inf}
\|\Phi^*_{\Bv_n}(\Phi(\Bv)-\Phi_{\Bv}\Bv)\|_\infty =  \|\mathcal{D}\big[(I-\Gamma^*V_n^*)^{-1}A^*A[(I-V_n\Gamma)^{-1}-(I-V\Gamma)^{-1}]VBB^*\big]\|_\infty
\end{equation}
Next, we note that
\begin{equation}
\label{eq:res}
(I-V_n\Gamma)^{-1}-(I-V\Gamma)^{-1}=(I-V_n\Gamma)^{-1}(V_n-V)\Gamma(I-V\Gamma)^{-1} \ .
\end{equation}
Then, we use the Hadamard property \eqref{eq:had} of the diagonal matrix $(V_n-V)$ in \eqref{eq:res} to obtain
\begin{align}
\label{eq:fullnl}
&\|\mathcal{D}\big[(I-\Gamma^*V_n^*)^{-1}A^*A[(I-V_n\Gamma)^{-1}-(I-V\Gamma)^{-1}]VBB^*\big]\|_\infty \nonumber \\ & \hspace{2cm}=\| \big((I-\Gamma^*V_n^*)^{-1}A^*A (I-V_n\Gamma)^{-1}\big)\circ\big(\Gamma(I-V\Gamma)^{-1}VBB^*\big)^T (\Bv_n-\Bv)  \|_\infty
\nonumber \\ & \hspace{2cm}=\| \big((I-\Gamma^*V_n^*)^{-1}A^*A (I-V_n\Gamma)^{-1}\big)\circ\big(\Gamma V(I-\Gamma V)^{-1}BB^*\big)^T (\Bv_n-\Bv)  \|_\infty
\end{align}
We use the same analysis to bound \eqref{eq:fullnl} by $\|Y\|_{\max}\|Z\|_{\max}\|\Bv_n-\Bv\|_1$, where $Y=(I-\Gamma^*V_n^*)^{-1}A^*A (I-V_n\Gamma)^{-1}$ and $Z=\Gamma V(I-\Gamma V)^{-1}BB^*$. Since $A^*A$ and $BB^*$ are normalized, we have the bounds
\begin{align}
&\|Y\|_{\max}\le\left( \frac{1}{1-\gamma_n}\right)^2 \\ 
&\|Z\|_{\max}\le \frac{(1+(s-1)\mu(B^*))\delta}{1-\gamma} \ ,
\end{align}
which results in the final error estimate
\begin{equation}
\|\Phi^*_{\Bv_n}(\Phi(\Bv)-\Phi_{\Bv}\Bv)\|_\infty \le \left( \frac{1}{1-\gamma_n}\right)^2 \frac{(1+(s-1)\mu(B^*))\delta}{1-\gamma} \|\Bv_n-\Bv\|_1
\end{equation}
Putting this result into Theorem \ref{thm:coherence}, we obtain the following result for the full $T$-matrix IHT algorithm.  

\begin{theorem}\label{thm:fullTmatrix}
Let $Y=A(I-V\Gamma)^{-1}VB+E$ and let $\{V_n\}$ be the sequence of diagonal matrices generated  by the nonlinear IHT iteration
\begin{equation}
V_{n+1}=H_s^\mathcal{D}\Big(\mathcal{D}\big(V_n+\tilde{A}_n^*\left(Y-\tilde{A}_nV_nB \right)B^*\big)\Big) \ ,
\end{equation}
where $\tilde{A}_n=A\left(I-V_n\Gamma\right)^{-1}$.  Let $\gamma=\|\Gamma V\|_1<1$, $\gamma_n=\|\Gamma V_n\|_1$, $\delta=\|\Gamma V\|_{\max}$, and $\delta_n=\|\Gamma V_n\|_{\max}$. 
Assuming that $\mu(B^*)\le\mu(A)$, for all $n\ge 1$, then
\begin{align}
\label{eq:bound3}
	\|\Bv_n-\Bv\|_1 \le & \rho_n\|\Bv_{n-1}-\Bv\|_1+\|\Phi_{\Bv_n}\Beps\|_\infty \ ,
\end{align}
where
\begin{align*}
&\rho_n=(3s+1)\left(\mu(B^*)+\left( \frac{1}{1-\gamma_n}\right)^2 \frac{(1+(s-1)\mu(B^*))\delta}{1-\gamma} \right)\ .
\end{align*}
\end{theorem}

\section{Numerical simulations}
\label{sec:coherence}

In this section we investigate the linear and nonlinear IHT algorithms in numerical simulations. We also analyze the convergence of the algorithms as characterized by Theorems \ref{thm:nonlinearIHT2ndBorn} and \ref{thm:fullTmatrix}. The analysis is carried out in three parts.  First, we compare the coherence estimates from section \ref{sec:mut_coh} with the true coherence for nonlinear IHT.  Next, we compare the convergence guarantees in Theorems \ref{thm:nonlinearIHT2ndBorn} and \ref{thm:fullTmatrix} with numerical simulations.  Finally, we report simulations that explore the use of the nonlinear IHT algorithm in settings beyond which the convergence estimates in Theorems \ref{thm:nonlinearIHT2ndBorn} and \ref{thm:fullTmatrix} hold.

\subsection{Far Field Coherence Estimates}
\label{subsec:coherence}

The coherence estimates derived in Section \ref{sec:mut_coh} are quite general and can be specialized for specific experimental geometries. We recall the following coherence in bounds under the assumption that $\mu(B^*)\le\mu(A)$:
\begin{itemize}
\setlength\itemsep{0.5em}
\item Linear: $\mu(\Phi) \le \mu(A)\mu(B^*)$
\item Second Born: $\mu(\Phi_\Bv) \le \mu(A(I+V\Gamma))\mu(B)^* \le \mu^{(1)}(A)\mu(B^*) $
\item Fully nonlinear: $\mu(\Phi_\Bv) \le \mu(A(I-V\Gamma)^{-1})\mu(B)^* \le \mu(B^*) $
\end{itemize}
where $\mu^{(1)}(A)$ is given by \eqref{eq:mu1A}.
In this section, we consider the case of far-field scattering as described in
Section~\ref{sec:scatt}. We begin by computing the coherence of the matrix $A$, as defined by~\eqref{eq:disc_ex}.  Using~\eqref{eq:mu_def}, we find that
\begin{equation}
\label{eq:coh_ex1}
\mu(A) =  \frac{1}{N_d}
\max_{j\neq \ell}\left|\sum_{m=1}^{N_d} e^{ik \Hx_m \cdot (\Br_j-\Br_\ell)} \right|\end{equation}    
We assume that the measurement directions $\{\Hx_m\}$ are uniformly distributed on the unit sphere.  For $N_s$ sufficiently large, we can approximate the above sum by an integral:\begin{equation}
\label{eq:coh_int1}
\mu(A)\approx\max_{j\neq\ell} \left|\frac{1}{4\pi}\int_{S^2} e^{ik \Hx \cdot(\Br_j-\Br_\ell)}d\Hx\right| = \max_{j\neq\ell}\left|\frac{\sin(k|\Br_j-\Br_\ell|)}{k|\Br_j-\Br_\ell|}\right| = \left|\frac{\sin(kh)}{kh}\right| \ ,
\end{equation}
where $h$ is the spacing between the voxels. This result is compared to the coherence computed for a finite number of directions in 
Fig.~\ref{fig:coh1}.   As seen from Eq.~\eqref{eq:coh_ex1}, the coherence is a decreasing function in the number of directions of illumination. If the incident directions $\Hx$ are chosen in the same manner as the measurement directions $\Hx$, then $\mu(B^*)=\mu(A)$.   

\begin{figure}[t]
\centering
\includegraphics[width=0.5\linewidth]{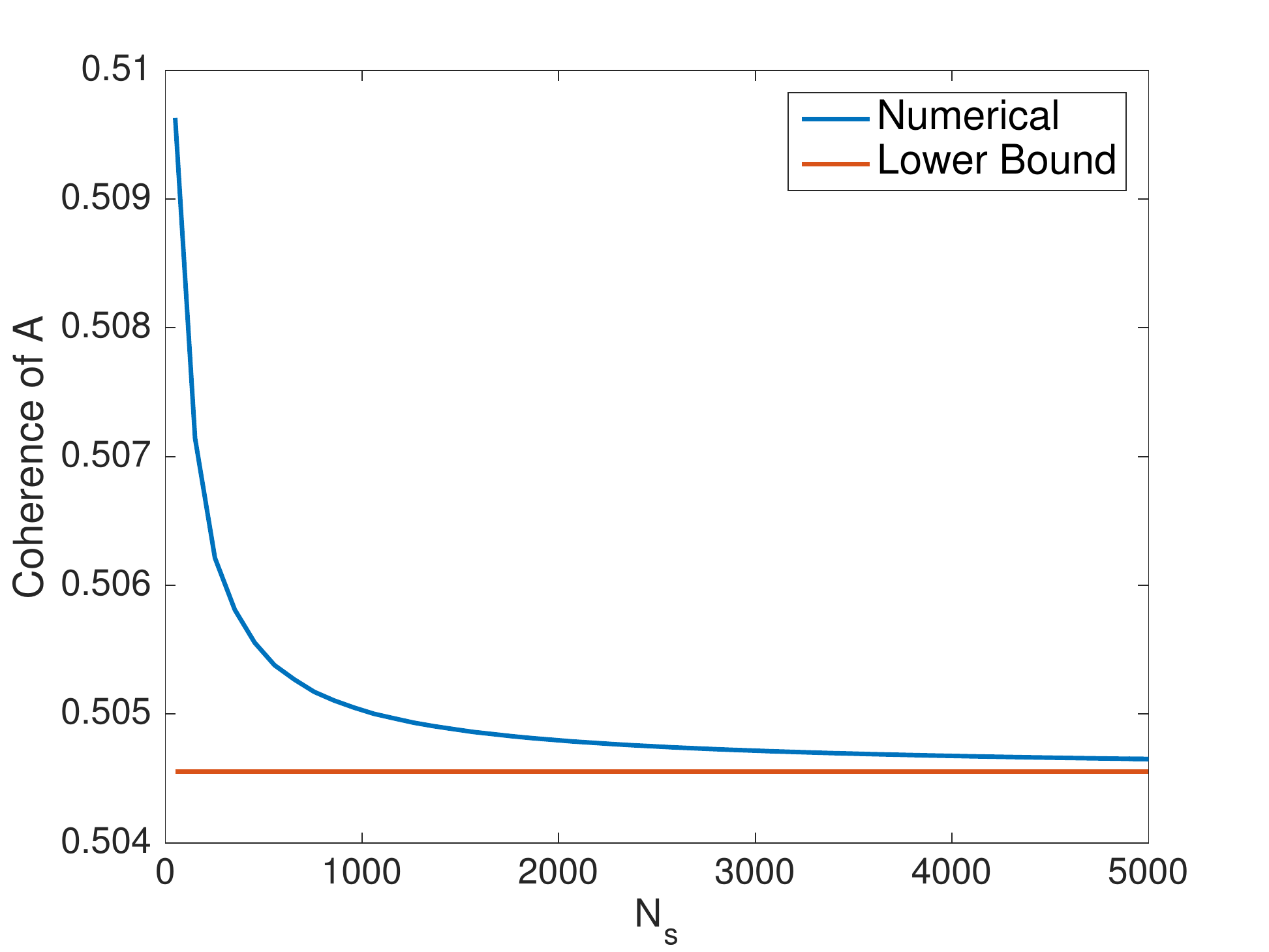}
\caption{Numerical computation of the coherence of the sensing matrix $A$ for varying number of incident directions $N_d$ with $kh=1.885$.  The region $\Omega$ was the unit cube $[0,1]^3$, discretized with $N=1000$ into a $10\times10\times10$ mesh. The red line corresponds to $\sin(kh)/(kh)\approx0.5045$.     
}
\label{fig:coh1}
\end{figure}

We now turn to the effects of scattering beyond the Born approximation on the coherence. We proceed by studying the experimental scenario of Fig.~\ref{fig:coh1}. In this setting, within the second Born approximation, \eqref{eq:mu1A} yields the bound $\mu(A(I+V\Gamma))\le1$ for all $s$ and $\delta_n$. This bound is not useful.  Likewise, in the previous analysis of the fully nonlinear algorithm, we use the same trivial bound $\mu(A(I-V\Gamma)^{-1})\le1$.  We look to improve these upper bounds for simple cases.  

\subsubsection{The case $s=1$} 
We consider the simplest case of a single scatterer ($s=1$) at the position $\Br_*$ with potential $\eta_0$.  Here $V\Gamma$ has only one nonzero row and thus $(I-V\Gamma)^{-1}=(I+V\Gamma)$. By~\eqref{eq:disc_ex} the entries of this matrix are given by
\begin{equation}
\label{eq:A(I+VG)}
\left[A(I+V\Gamma)\right]_{j\ell}=e^{ik\Hx_j\cdot\Br_\ell}+\frac{k^2h^3\eta_0}{|\Br_\ell-\Br_*|}e^{ik(|\Br_\ell-\Br_*|+\Hx_j\cdot\Br_*)} \ .
\end{equation}
The coherence  is now approximated for large values of $N_s$ by 
\begin{equation}
\label{eq:muAIVG_ex}
\mu(A(I+V\Gamma))\approx \max_{j\neq\ell}\frac{\int_{S^2}f(\Br_j,\Hx)\overline{f(\Br_\ell,\Hx)}d\Hx}{\sqrt{\int_{S^2}|  f(\Br_j,\Hx)d\Hx |^2\int_{S^2}|  f(\Br_\ell,\Hx)d\Hx |^2}} \ ,
\end{equation}
where, in accordance with Eq.~\eqref{eq:A(I+VG)}, the function $f$ is given by 
\begin{align}
f(\Br,\Hx) = e^{ik\Hx\cdot\Br}+\frac{k^2h^3\eta_0}{|\Br-\Br_*|}e^{ik(|\Br-\Br_*|+\Hx\cdot\Br_*)} \ .
\end{align}
Computing the above integrals, we obtain 
\begin{equation}
\label{eq:AIVG_one}
\mu(A(I+V\Gamma))=\max_{j\ne\ell}\left[D_1(\Br_j)D_1(\Br_\ell)\left(\left|\frac{\sin(k|\Br_j-\Br_\ell|)}{k|\Br_j-\Br_\ell|}\right| +D_2(\Br_j,\Br_\ell) \right)\right]\ ,
\end{equation}
where the functions $D_1$ and $D_2$ are defined as
\begin{align}
D_1(\Br)=&\left(1+\frac{kh^3\eta_0 \sin (2 |\Br-\Br_*| )}{k|\Br-\Br_*|^2 }+\frac{k^4h^6\eta_0^2}{|\Br_\ell-\Br_*|^2}\right)^{-1/2}  \nonumber \\
D_2(\Br_j,\Br_\ell)=&\frac{kh^3\eta_0 \left(\sin (k (|\Br_j-\Br_*|+|\Br_\ell-\Br_*|))+k^2h^3\eta_0 e^{i k (|\Br_j-\Br_*|-|\Br_\ell-\Br_*|)}\right)}{ |\Br_j-\Br_*||\Br_\ell-\Br_*|}\ .
\end{align} 
This formula makes evident that the coherence of $A(I+V\Gamma)$ is a scaled version of $\mu(A)$ with an additive term.  The sinc function (which gives $\mu(A)$) is the only term that depends on the distance between $\Br_j$ and $\Br_\ell$.  All other terms depend on the details of the scatterer. 

The dependence of the coherence on $|\Br_j-\Br_*|$ is shown in Fig.~\ref{fig:coh2} for several values of $\eta_0$.  We have verified numerically that the above maximum is obtained by setting $|\Br_j-\Br_\ell|=h$ and $|\Br_j-\Br_*|=|\Br_\ell-\Br_*|$. We see that the coherence is given by the maximum value of the curve, as indicated by the dots.  Interestingly, for stronger scattering ($\eta_0=0.2$) the maximum is achieved by taking points close to the scatterer (with the closest distance being $h$), while for weaker scattering, the maximum is obtained at approximately $2h$.  Thus stronger scattering increases the coherence and there is a trade off between coherence and nonlinearity, as further illustrated below.  
 \begin{figure}
\centering
\includegraphics[width=0.5\linewidth]{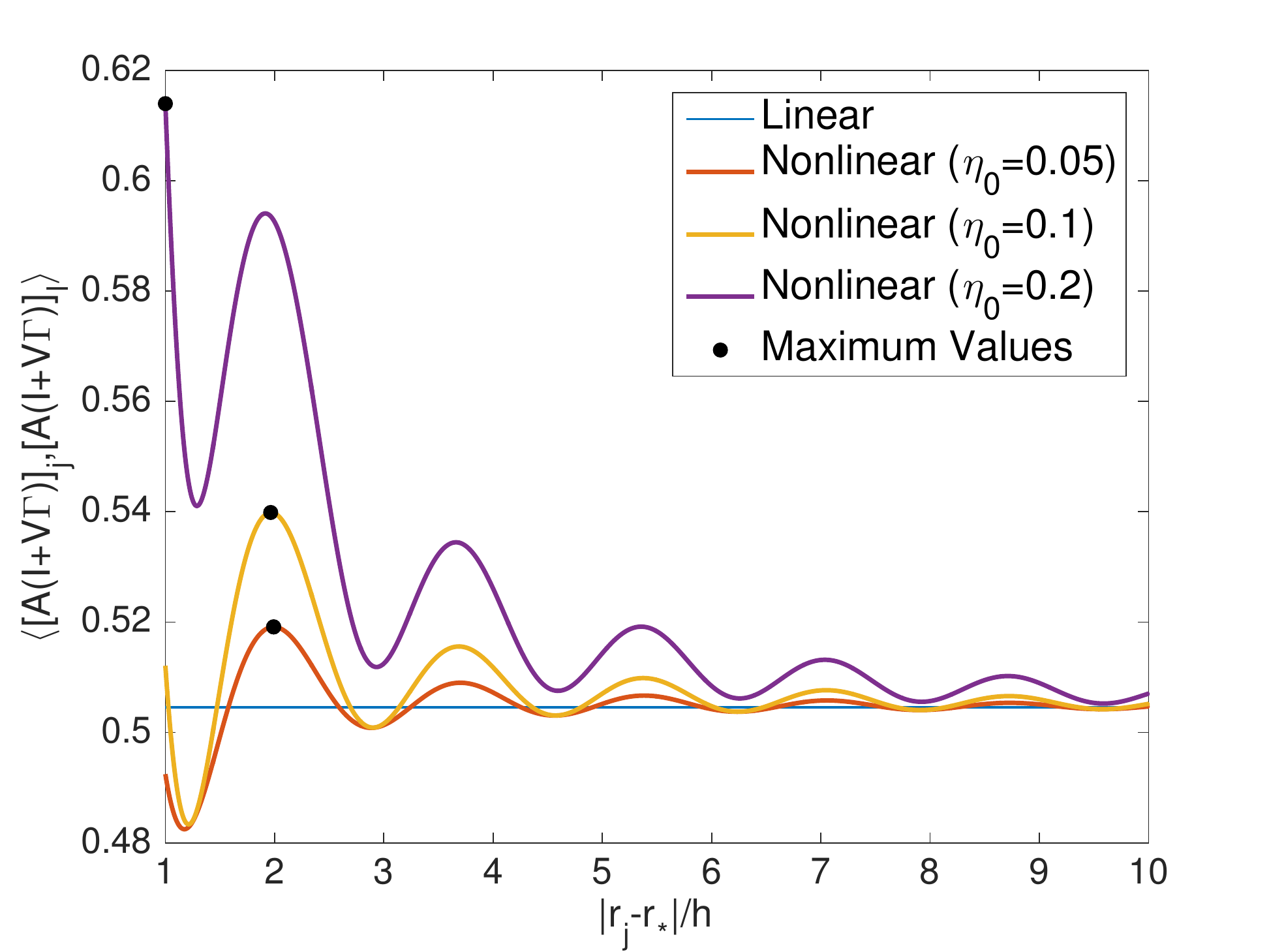}
\caption{Plots of the coherence \eqref{eq:AIVG_one} for $N_s=500$, $kh=1.885$ and $\eta_0=0.05, 0.1$, and 0.2.  The value $|\Br_j-\Br_\ell|=h$ is fixed, and $|\Br_j-\Br_*|=|\Br_\ell-\Br_*|$.  The coherence for each case is the maximum value of the appropriate plot, as indicated by the black dot.}
\label{fig:coh2}
\end{figure}

\subsubsection{The case $s>$1} 
For $s>1$, let $\mathcal{I}$ index the locations of the scatterers, with $|\mathcal{I}|=s$.    If each scatterer has the same potential $\eta_0$,  the coherence is given by Eq.~\eqref{eq:muAIVG_ex} with
\begin{equation}
\label{eq:f_def_s}
f(\Br,\Hx) = e^{ik\Hx\cdot\Br}+k^2h^3\eta_0\sum_{n\in\mathcal{I}}\frac{e^{ik(|\Br-\Br_n|+\Hx\cdot\Br_n)}}{|\Br-\Br_n|} \ ,
\end{equation} 
which is derived by computing the entries of the matrix $A(I+V\Gamma)$.  
The required integrals can still be computed analytically, but become quite cumbersome.  Thus, we will resort to numerical calculations.  Fig.~\ref{fig:coh3} displays the numerically computed coherences for varying numbers of scatterers. In each experiment, $s$ scatterers are uniformly randomly placed inside the unit cube, which has been discretized into a $10\times10\times10$ voxel grid, with $N_s=500$ and $kh=1.885$.   We compare the case when $\eta_0$ is fixed for all $s$, and the case when $\eta_0$ changes but $\|V\Gamma\|_1$ is fixed.  In the latter case, the coherence decreases as $s$ increases.   Moreover, there is a minimal difference between the coherence for the first and second Born approximations. 
We also compare the results of Figs.~\ref{fig:coh2} and ~\ref{fig:coh3}, where for $s=1$ and $\eta_0=0.1$, we obtain the analytically computed nonlinear coherence value of 0.540 and the numerically computed value of 0.539.   Overall this observation confirms the conservative nature of the mutual coherence estimates.  

\begin{figure}
\centering
\begin{subfigure}[b]{0.4\textwidth}
                \centering
                \includegraphics[width=\textwidth]{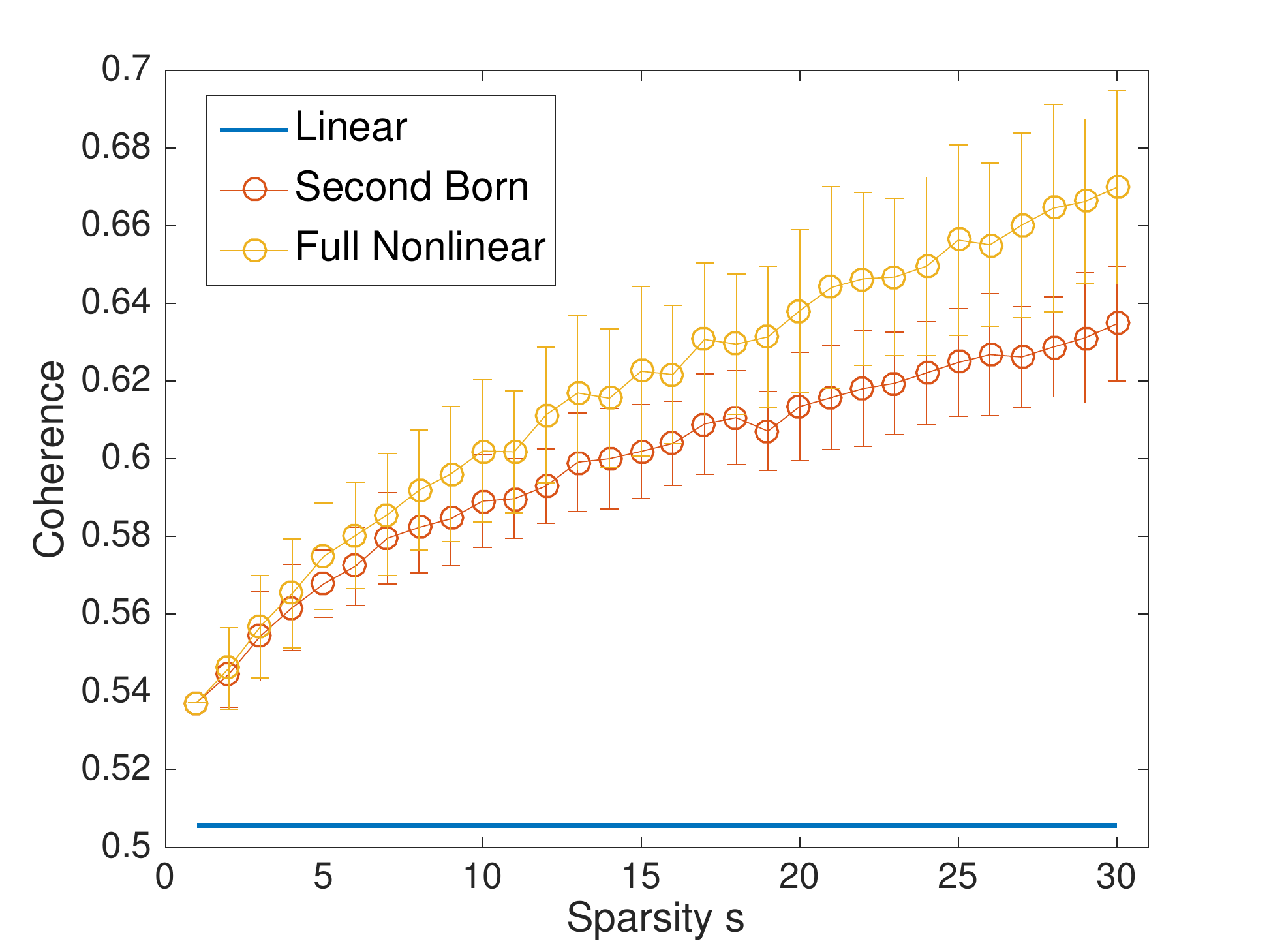}
                \caption{$\eta_0=0.1$}
        \end{subfigure}
\begin{subfigure}[b]{0.4\textwidth}
                \centering
                \includegraphics[width=\textwidth]{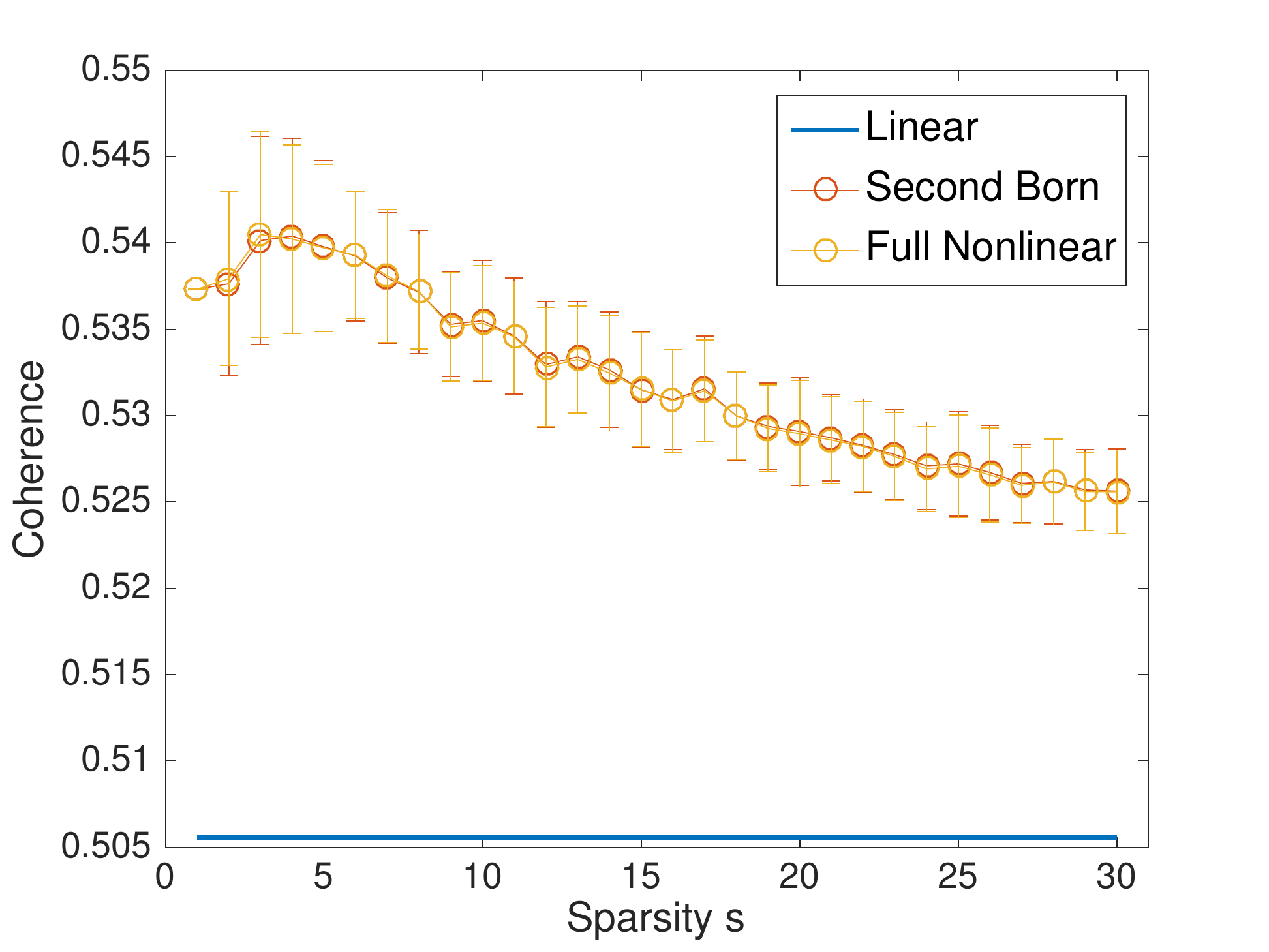}
                \caption{$\eta_0$ varies to fix $\|V\Gamma\|_1=0.3553$}
        \end{subfigure}
\caption{Plots of the numerically computed coherence for varying number of scatterers $s$ with fixed potential $\eta_0$.  The three curves represent the coherence of the linear matrix $\mu(A)$, the second Born approximation matrix $\mu(A(I+V\Gamma))$ and the fully nonlinear matrix $\mu(A(I-V\Gamma)^{-1})$.  For each level of sparsity, $s$ inhomogeneities were uniformly randomly placed in $\Omega$.  The plotted coherence values are the average over 100 realizations, with error bars signifying one standard deviation in each direction.  For the plot on the left, $\eta_0=0.1$ is fixed, while on the right hand, $\eta$ was chosen adaptively so that $\|V\Gamma\|_1=0.3553$ was fixed, a number corresponding to $\eta_0=0.1$ when $s=1$.  Note that for the analyses of both nonlinear algorithms we used the pessimistic bound $\mu\le1$.
}
\label{fig:coh3}
\end{figure}

\subsection{Comparison with convergence theorems}
\label{subsec:conv_comp}

We consider the same experimental setup as above, but at higher frequencies with $kh=1.885$.  Evidently, convergence of the nonlinear IHT algorithm is not guaranteed.  We also note that in the linear case with $\mu(A)=\mu(B^*)=0.505$, convergence is not guaranteed, since no choice of $s$ satisfies the condition $\rho=\mu(A)\mu(B^*)(3s+1)<1$.   In this section, we investigate some examples that shed some light on the theoretical convergence results.  We will see that the theory for the nonlinear IHT algorithm is quite conservative, and that even when the theory does apply, faster convergence than predicted is obtained in practice.   

Since the previous choice of parameters fell outside of the theoretical convergence regime, we consider two experiments with extreme choices of parameters for which the theoretical convergence guarantees hold.  Here the domain $\Omega$ is discretized into a 10$\times$10$\times$10 grid.  For each experiment, three voxels are chosen uniformly at random to have nonzero potential $\eta_0$.   The sources consist of 400 incident plane waves at uniformly spaced angles over the entire sphere, with the scattering amplitude being measured in the far field for the same angles.  Forward data is generated using the coupled-dipole method \cite{1973ApJ186705P}.  

For the first experiment,  $\Omega$ is a cube with sidelengths $4.75\lambda$, which makes the dimensionless parameter $kh=2.98$.  This high frequency setup results in coherence values of $\mu(A)=\mu(B)= 0.0525$.  The potential $\eta_0$ is set to $10^{-5}$.  After choosing the three scatterer locations at random (uniformly), we find that $\delta=\|\Gamma V\|_{\max}=0.0395$ and $\gamma=\|\Gamma V\|_{1}=0.0460$. For the second experiment, the cubic region is taken to be smaller with sides of length 4.55$\lambda$ with $kh=2.858$. We then obtain $\mu(A)=\mu(B)= 0.098$.  The potential has also been increased to $\eta_0=10^{-4}$, which gives $\delta=\|\Gamma V\|_{\max}=0.987$ and $\gamma=\|\Gamma V\|_{1}=0.1299$ after choosing the three scatterer locations uniformly at random.    
For both experiments, the theoretical error bounds $\|\Bv-\Bv_n\|_1$ for linear, second Born approximation and nonlinear IHT are shown in Fig.~\ref{fig:convplots1}. Also shown, for comparison, are numerical results obtained from the reconstruction algorithms under the same conditions with the thresholding parameter set to $s=3$. In Fig.~\ref{fig:convplots1} we plot the dependence of the $L^1$ error on the number of iterations of the reconstruction algorithms for both experiments. The results from numerical reconstructions are compared to 
Theorems 4.5 and 4.6. As may be expected, we find that the theoretical error bounds are conservative.

\begin{figure}[h]
\centering
\begin{subfigure}[b]{0.4\textwidth}
                \centering
                \includegraphics[width=\textwidth]{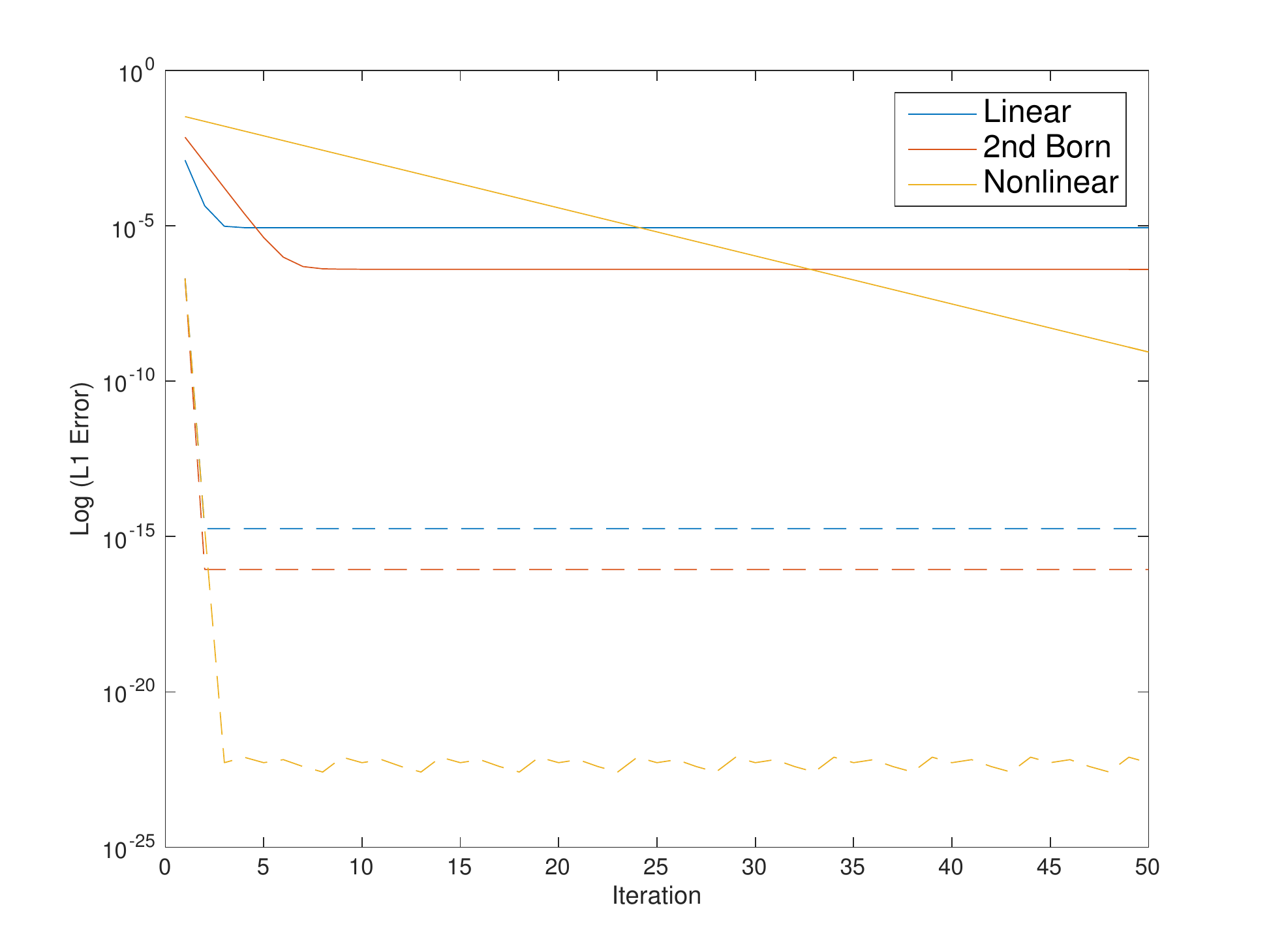}
                \caption{$\mu(A)=\mu(B^*)=0.0525$ \\ $\delta=0.0395$, $\gamma=0.046$, $s=3$}
        \end{subfigure}
\begin{subfigure}[b]{0.4\textwidth}
                \centering
                \includegraphics[width=\textwidth]{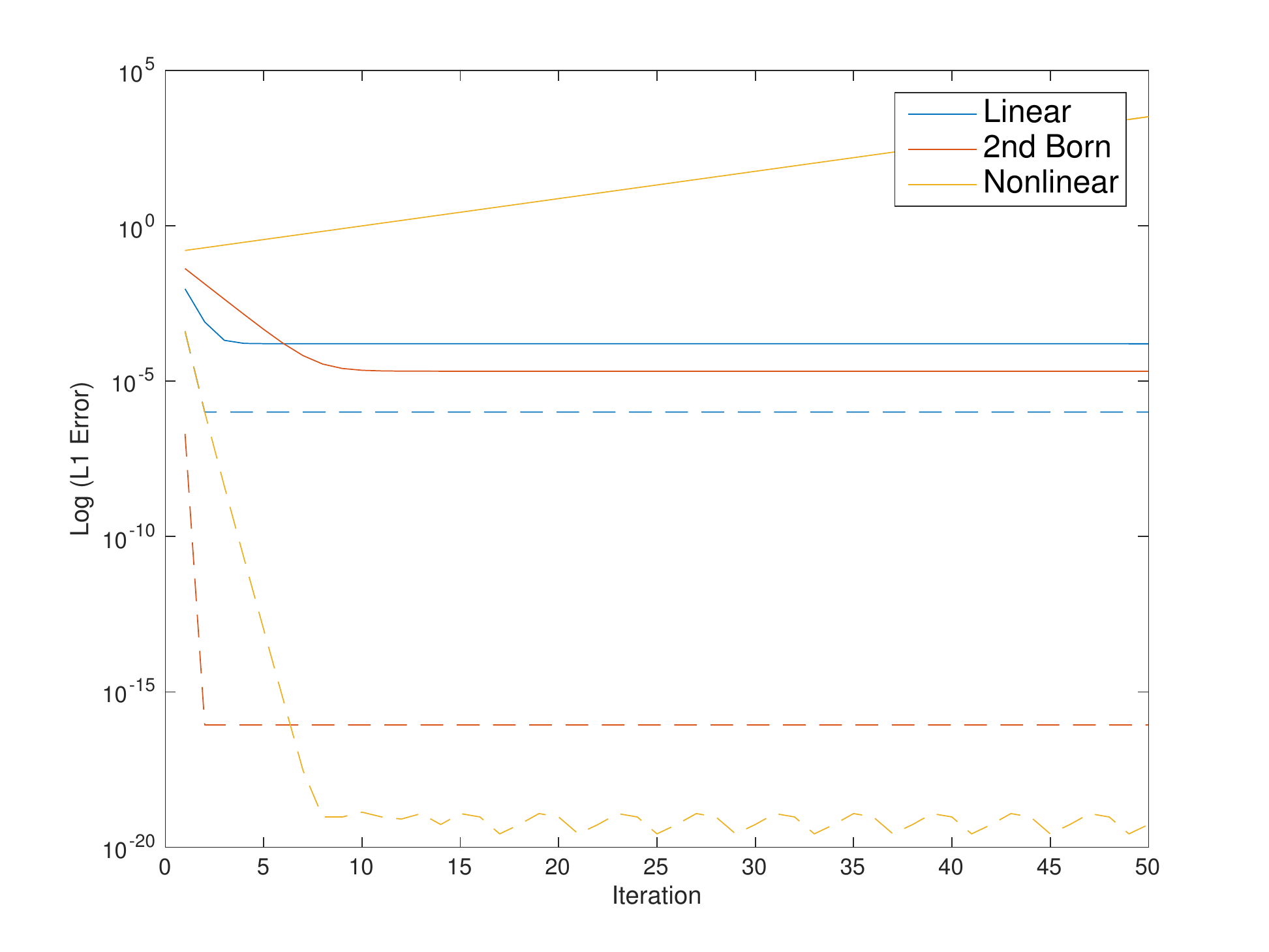}
                \caption{$\mu(A)=\mu(B^*)=0.098$ \\ $\delta=0.0987$, $\gamma=0.1299$, $s=3$}
        \end{subfigure}
\caption{Comparison of the error $\|\Bv-\Bv_n\|_1$ for numerical reconstructions and theory.  The left plot has $\mu(A)=\mu(B^*)=0.0525$, $\delta=0.0395$, $\gamma=0.046$, and $s=3$.  The right plot has $\mu(A)=\mu(B^*)=0.098$, $\delta=0.0987$, $\gamma=0.1299$, and $s=3$.  The solid lines indicate the theoretical estimates given by Theorems 4.5 and 4.6 and the dashed lines are the results from numerical simulations. Note that for strongernonlinearity (as shown on the right) the theory does not guarantee convergence, but the simulated reconstruction is within the convergence regime.} 
\label{fig:convplots1}
\end{figure}

\subsection{Numerical Reconstructions}
\label{sec:sims}

In this section we illustrate the use of the nonlinear IHT algorithm in settings where convergence cannot be expected.  
Two models are used to test the algorithm: two spherical scatterers of constant potential (model 1) and one larger spherical scatterer with radially varying potential (model 2).  The scatterers are contained in the cubic domain $\Omega$ with sides of length $5\lambda$. In model 1 the scatterers have potential $\eta_0$, radii $\lambda/2$, and are separated by a distance of $3\lambda/2$ between their centers. In model 2, the scatterer has radius $5\lambda/4$ and potential that decreases linearly from the value $\eta_0$ to 0 at its center. Tomographic slices for each model are shown in Fig.~\ref{fig:models}.  
 \begin{figure}[h!]
\centering
\includegraphics[width=0.65\linewidth,trim={0 4cm 0 4cm},clip]{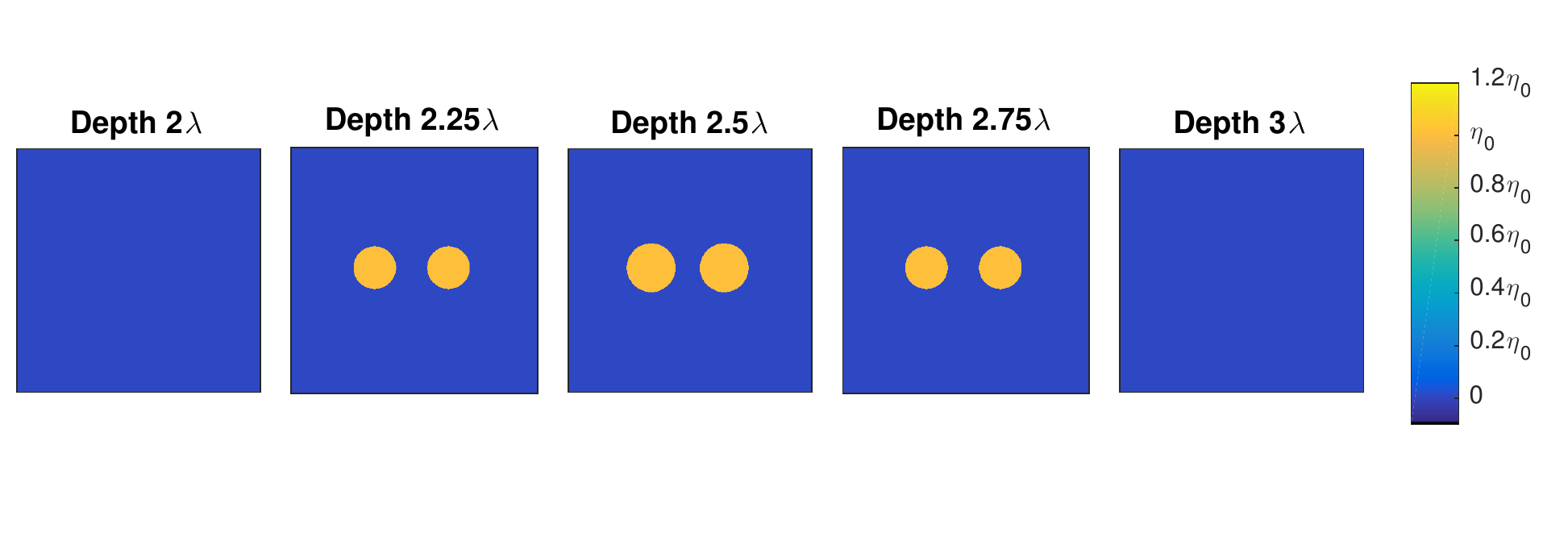}
\includegraphics[width=0.65\linewidth]{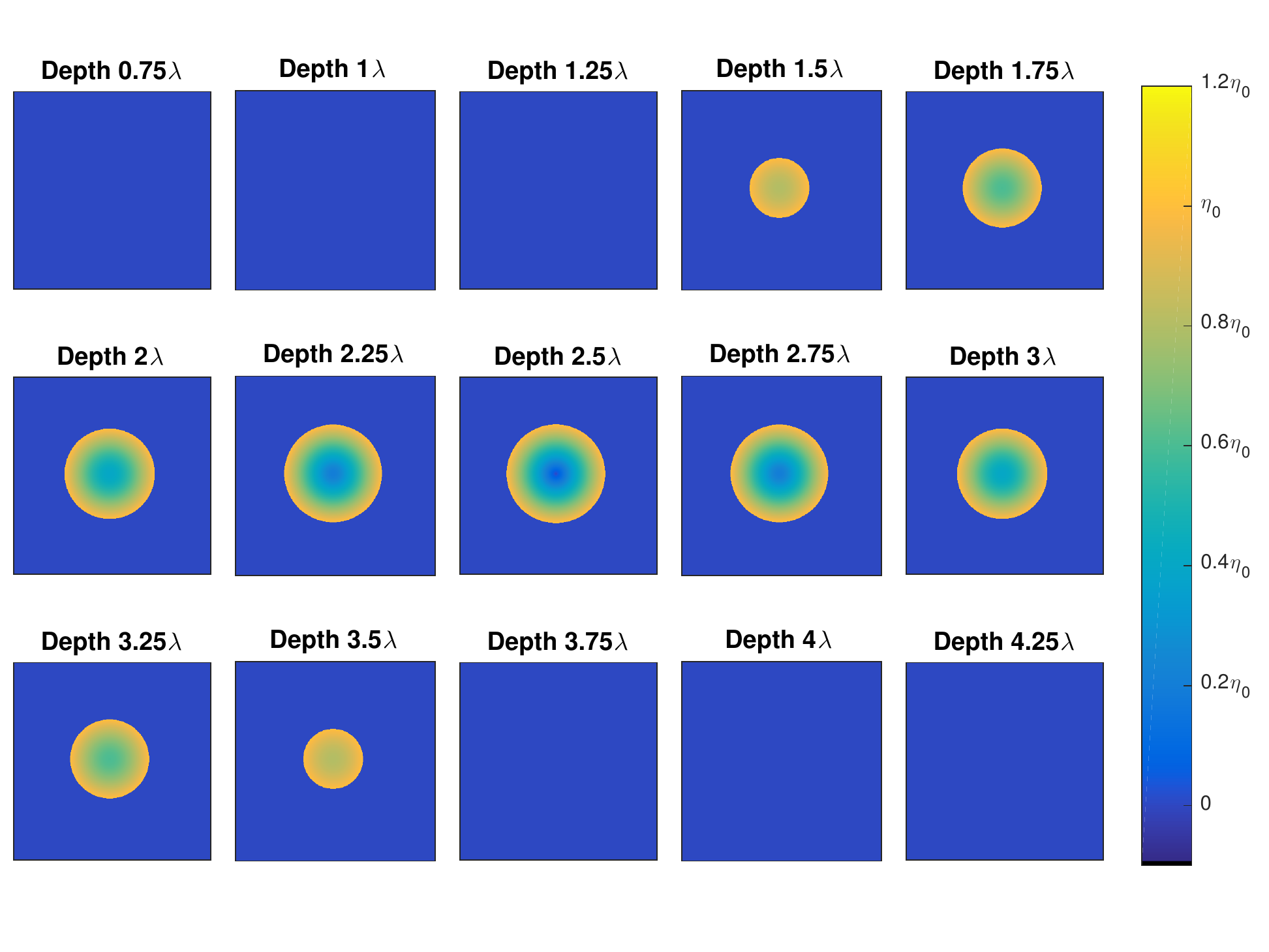}\\
\caption{Tomographic slices of model 1 (top) and model 2 (bottom). The field of view in each  slice is $5\lambda \times 5 \lambda$.}
\label{fig:models}
\end{figure}
Forward data was generated using the coupled-dipole method \cite{1973ApJ186705P}, using a $25\times25\times25$ discretization. The incident and measurement directions were chosen to be spaced uniformly on half of the unit sphere, with $N_s=N_d=225$ and $kh=1.2566$.  Gaussian white noise was added to the data at a level of 1\%. Here the coherence is given by$\mu(A)=\mu(B)=0.764$. 

Reconstructions were conducted using a discretization of $\Omega$ into a $21\times21\times21$ cubic mesh, thereby avoiding so-called inverse crime.  Four versions of the IHT algorithm were investigated: linear ($M=1$), second Born ($M=2$), third Born ($M=3$), and fully nonlinear ($M=\infty$). Each algorithm was run for 100 iterations.  In all cases, the iterations are initialized by $V_0=0$. For model 1, the threshold level $s$ for the thresholding operator $H_s$ was set to 230/9,261 = 2.48\%.   Note that the threshold limit in $H_s$ is not equal to the actual sparsity $s$, which is  0.84\% sparse.   While not explained by the theoretical analysis, this level of sparsity leads to better performance, especially in the presence of noise.  Reconstructions of the central slices (corresponding to a depth of $2.5\lambda$) of model 1 using the linear and fully nonlinear algorithms 
are displayed in Table~\ref{tbl:model1}. At the weakest level of scattering, both algorithms perform nearly identically.  However, as the strength of scattering is increased, the linear reconstruction degrades, while the nonlinear algorithm reconstructs the scatterers quite well.  At the highest level of scattering, both algorithms fail.  Convergence plots for all algorithms are also displayed.  For each iteration, the error $Y_{err}$ measures the relative error between the true data $Y$ and data that would be generated by the current reconstruction $V^{(\rm rec)}$.  Consistent with with Eq.~\eqref{eq:disc_forward}, the error is defined as   
\begin{equation}
\label{eq:Yerr}
Y_{\rm err}^2=\frac{\sum_{j=1}^{N_d}\sum_{\ell=1}^{N_s} \left|Y_{j\ell}-\left[A(I-V^{(\rm rec)}\Gamma)^{-1}V^{(\rm rec)}B\right]_{j\ell}\right|^2}{\sum_{j=1}^{N_d}\sum_{\ell=1}^{N_s}\left|Y_{j\ell}\right|^2} \ .
\end{equation}

        \begin{table}[h!]
        \centering
        \begin{tabular}{|c|c|c|c|}
           \toprule
            potential & Linear IHT & Nonlinear IHT & Error  \\
            \midrule
            $\eta_0=0.01$ & \includegraphics[width=0.25\textwidth]{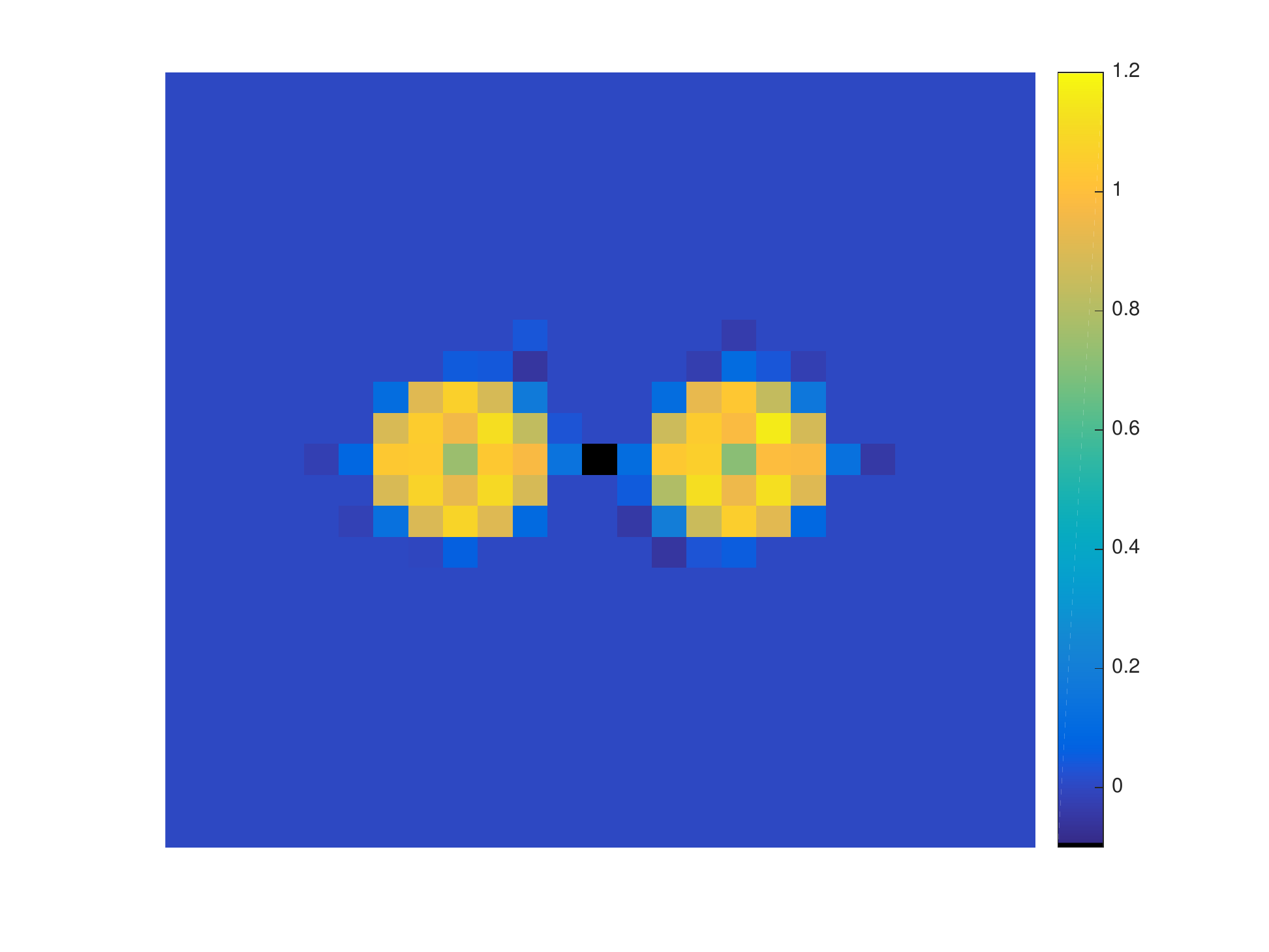} & \includegraphics[width=0.25\textwidth]{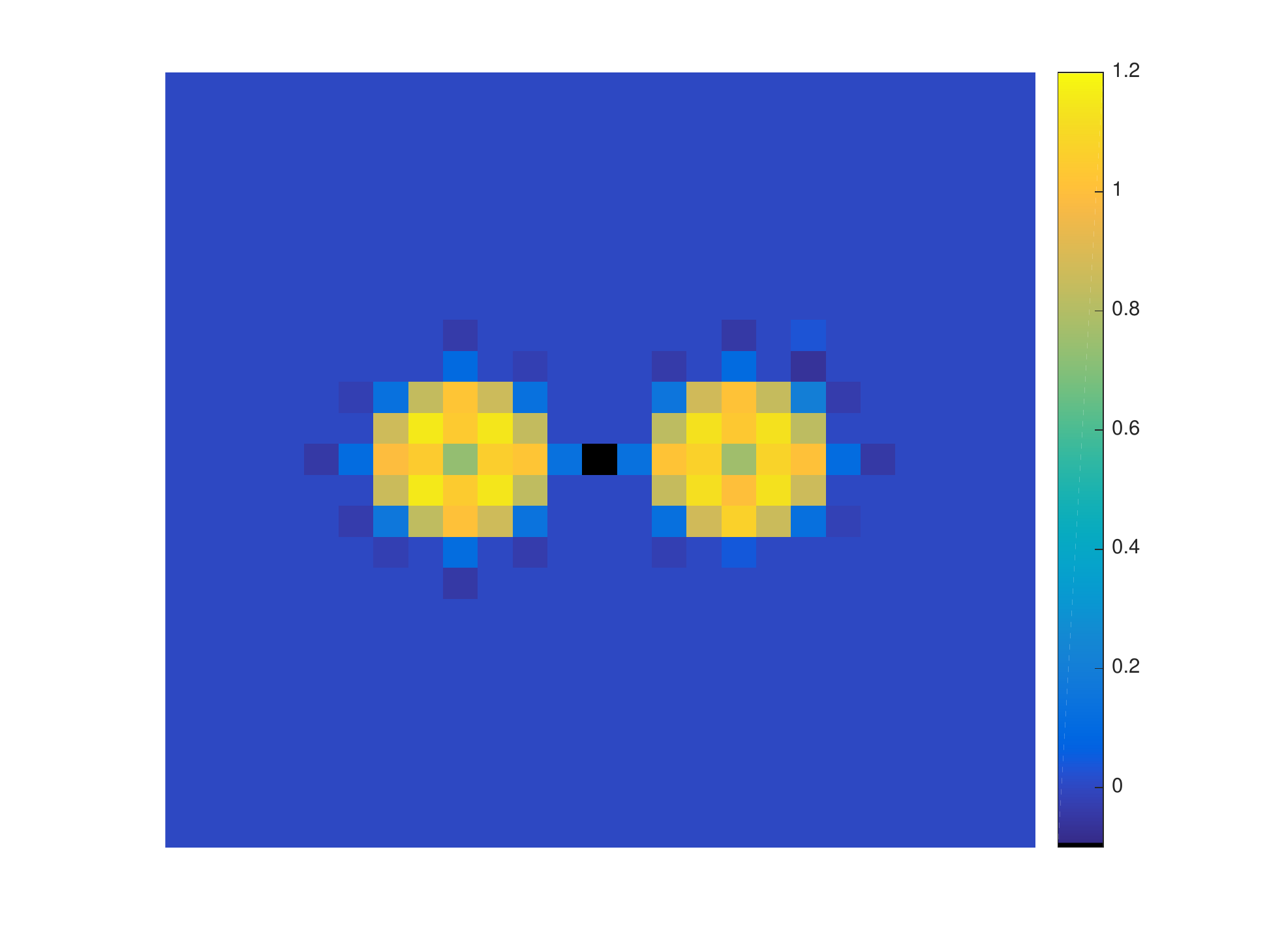} & \includegraphics[width=0.25\textwidth]{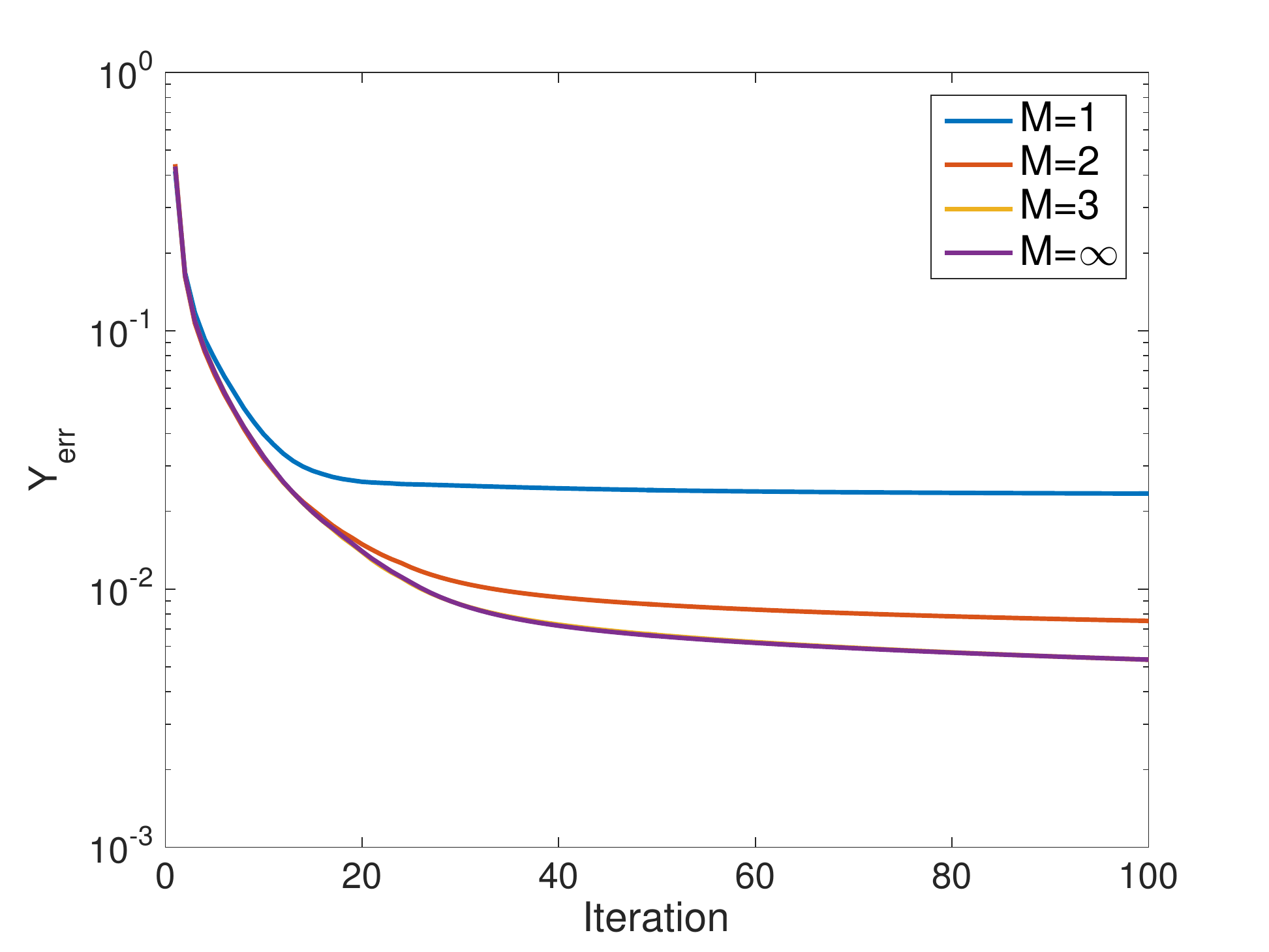} \\ \midrule 
                       $\eta_0=0.06$ & \includegraphics[width=0.25\textwidth]{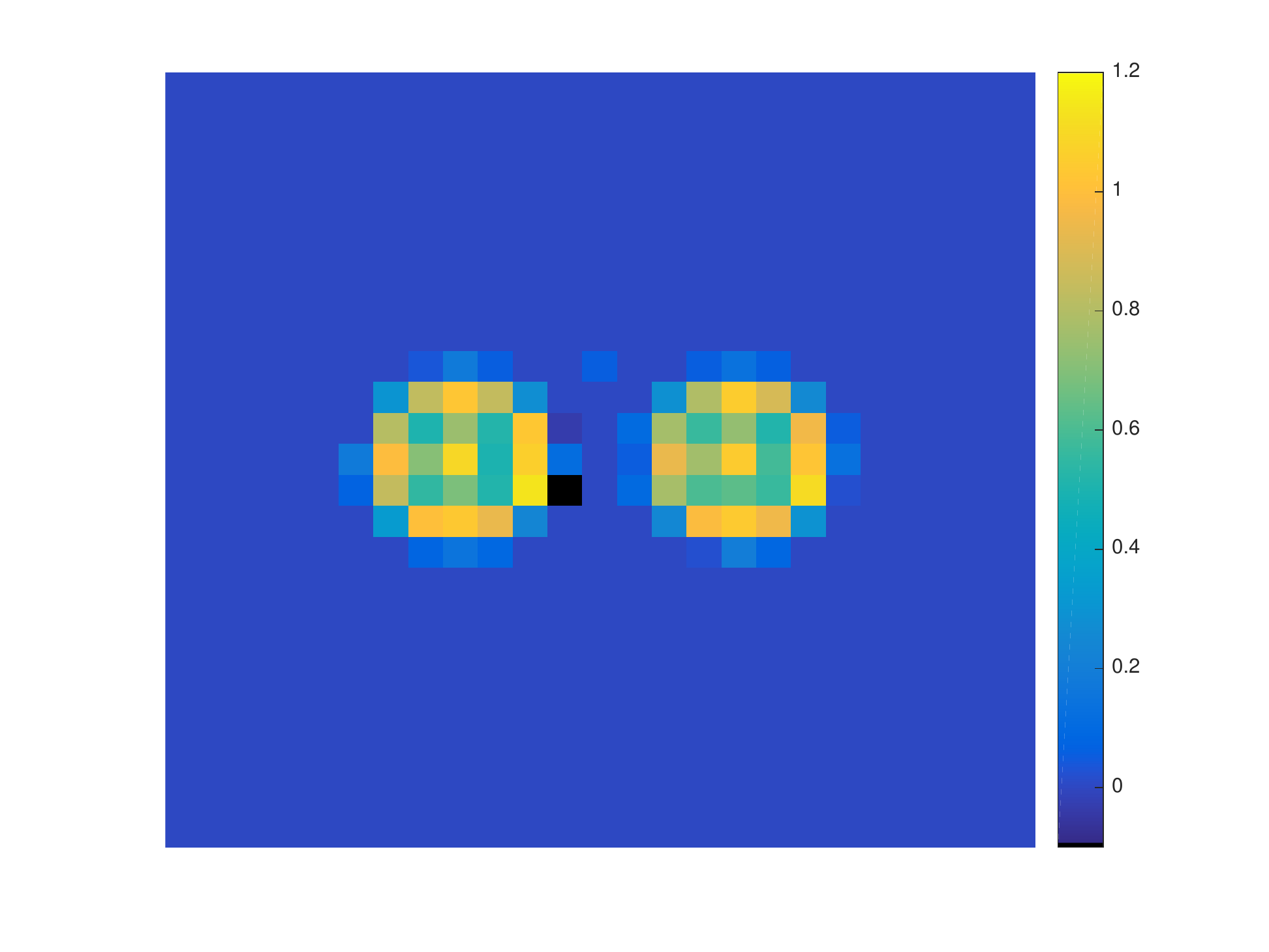} & \includegraphics[width=0.25\textwidth]{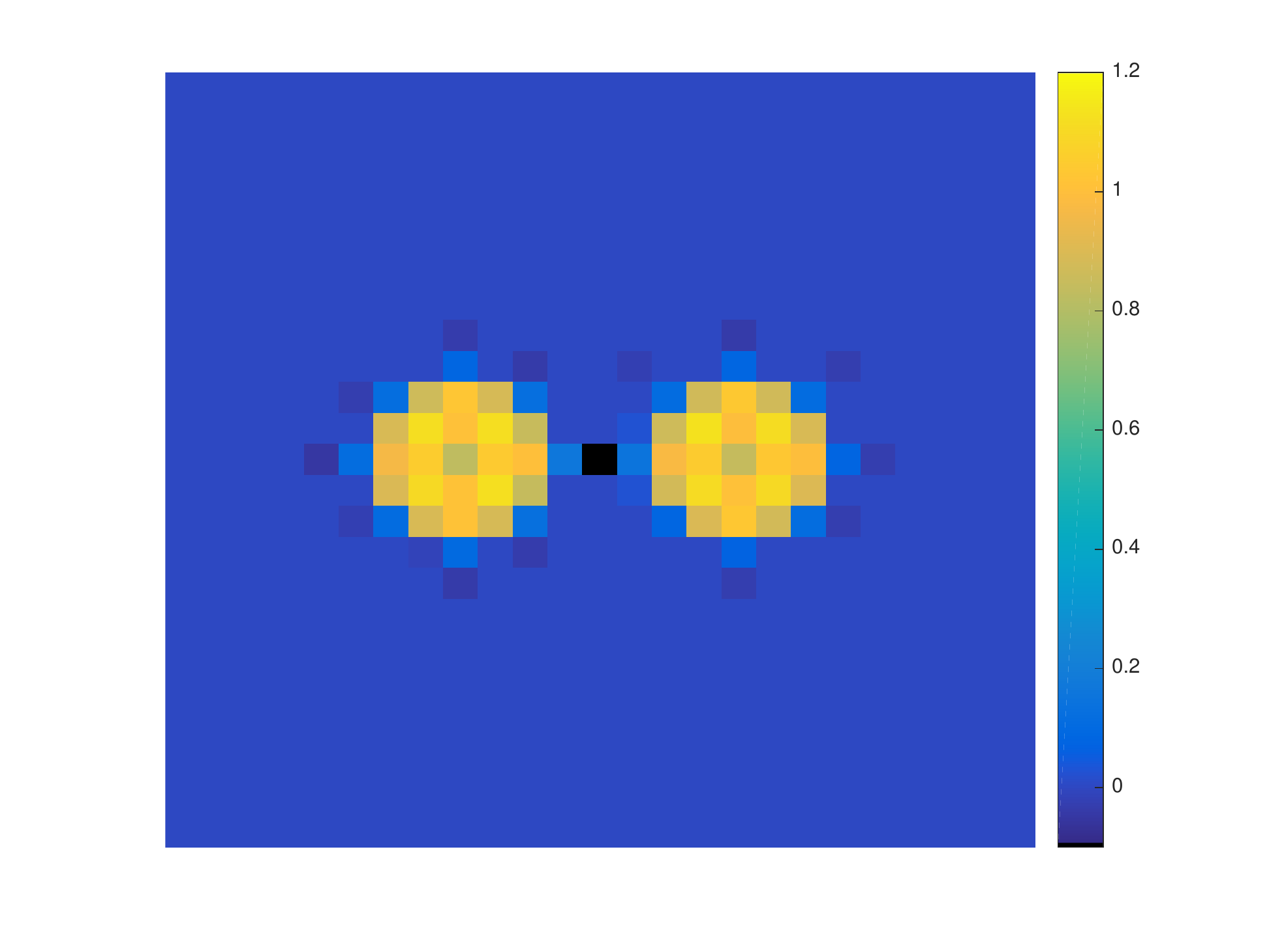}  & \includegraphics[width=0.25\textwidth]{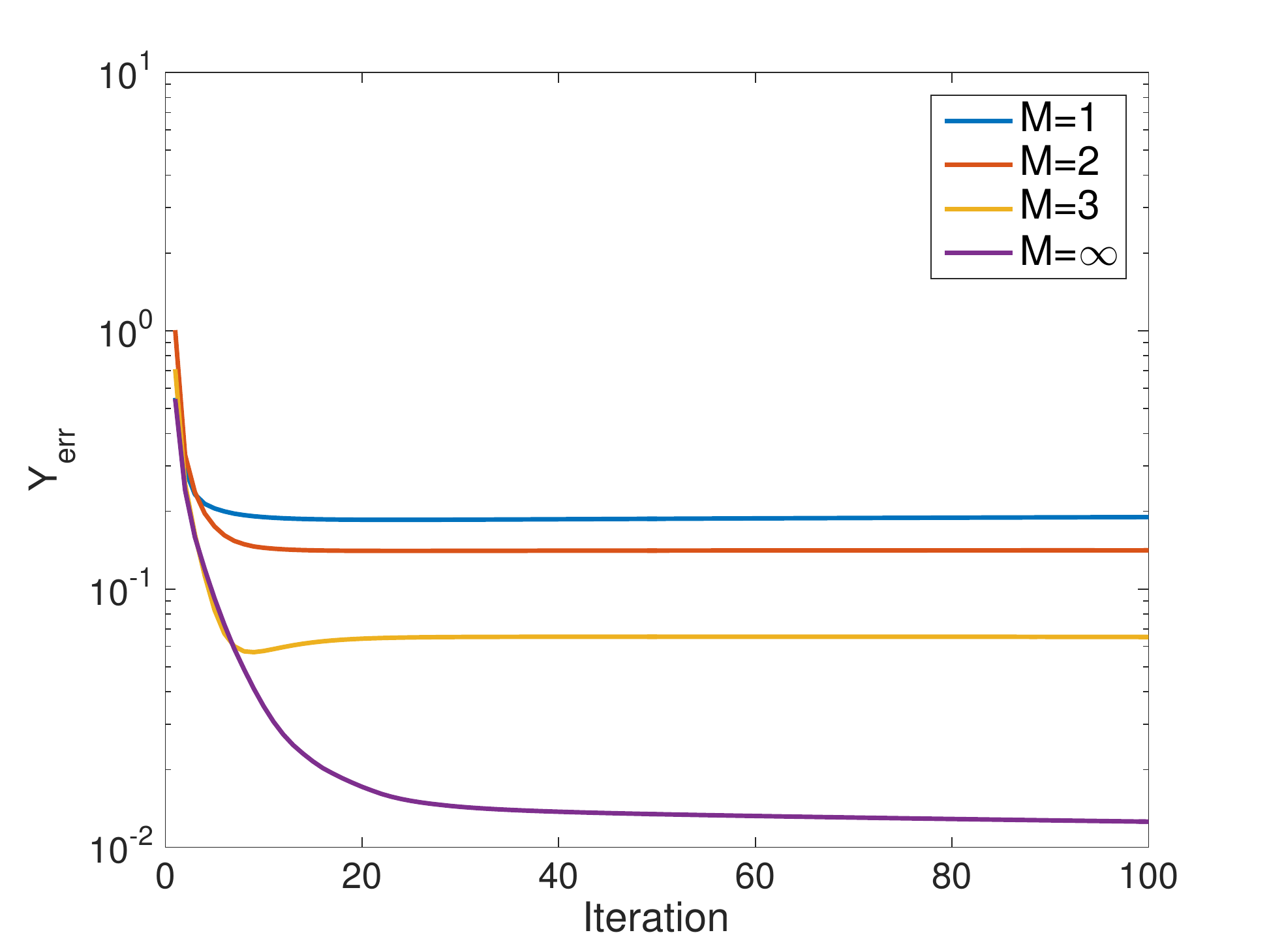}\\ \midrule
                                                                  $\eta_0=0.1$ & \includegraphics[width=0.25\textwidth]{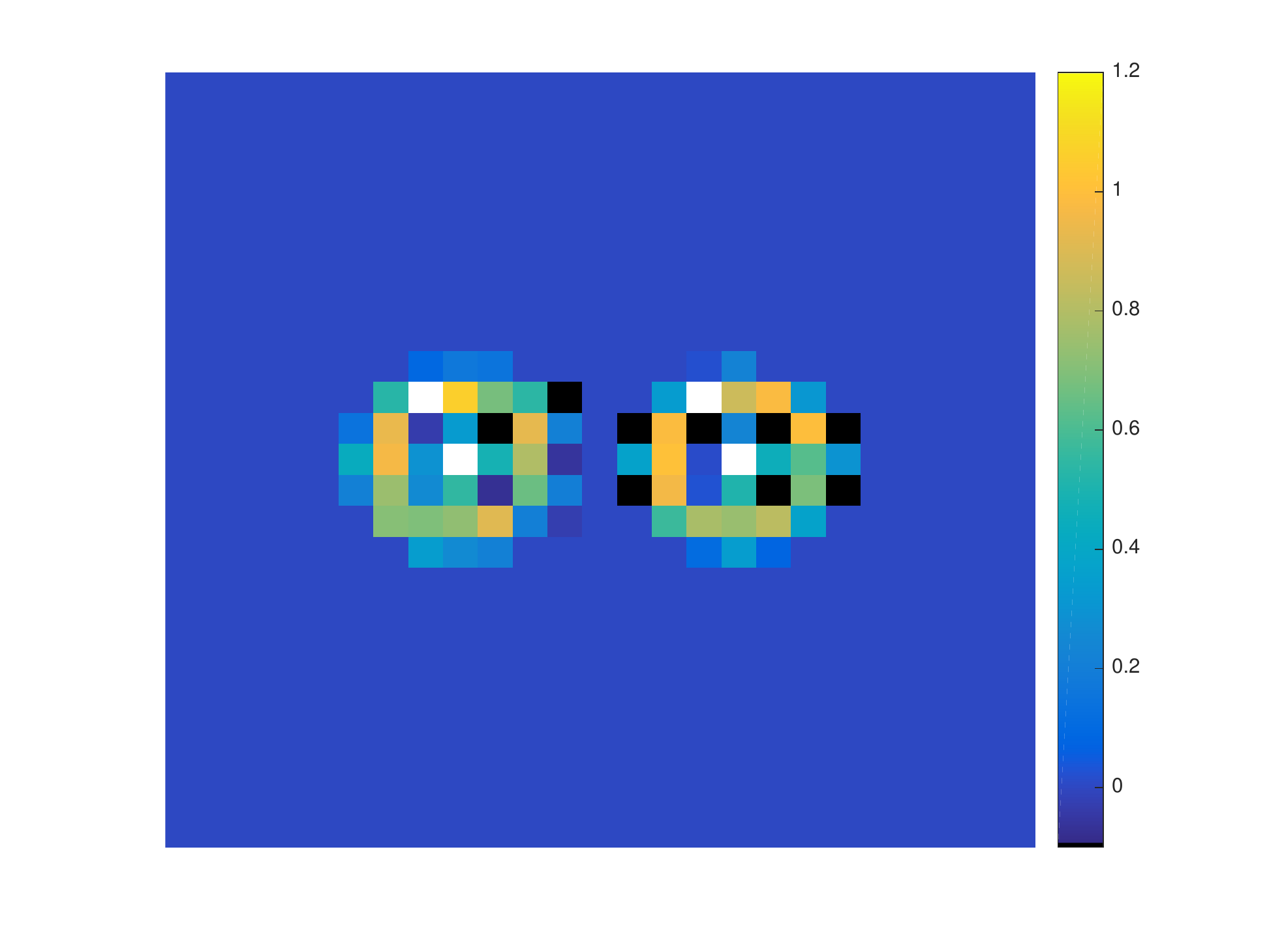} & \includegraphics[width=0.25\textwidth]{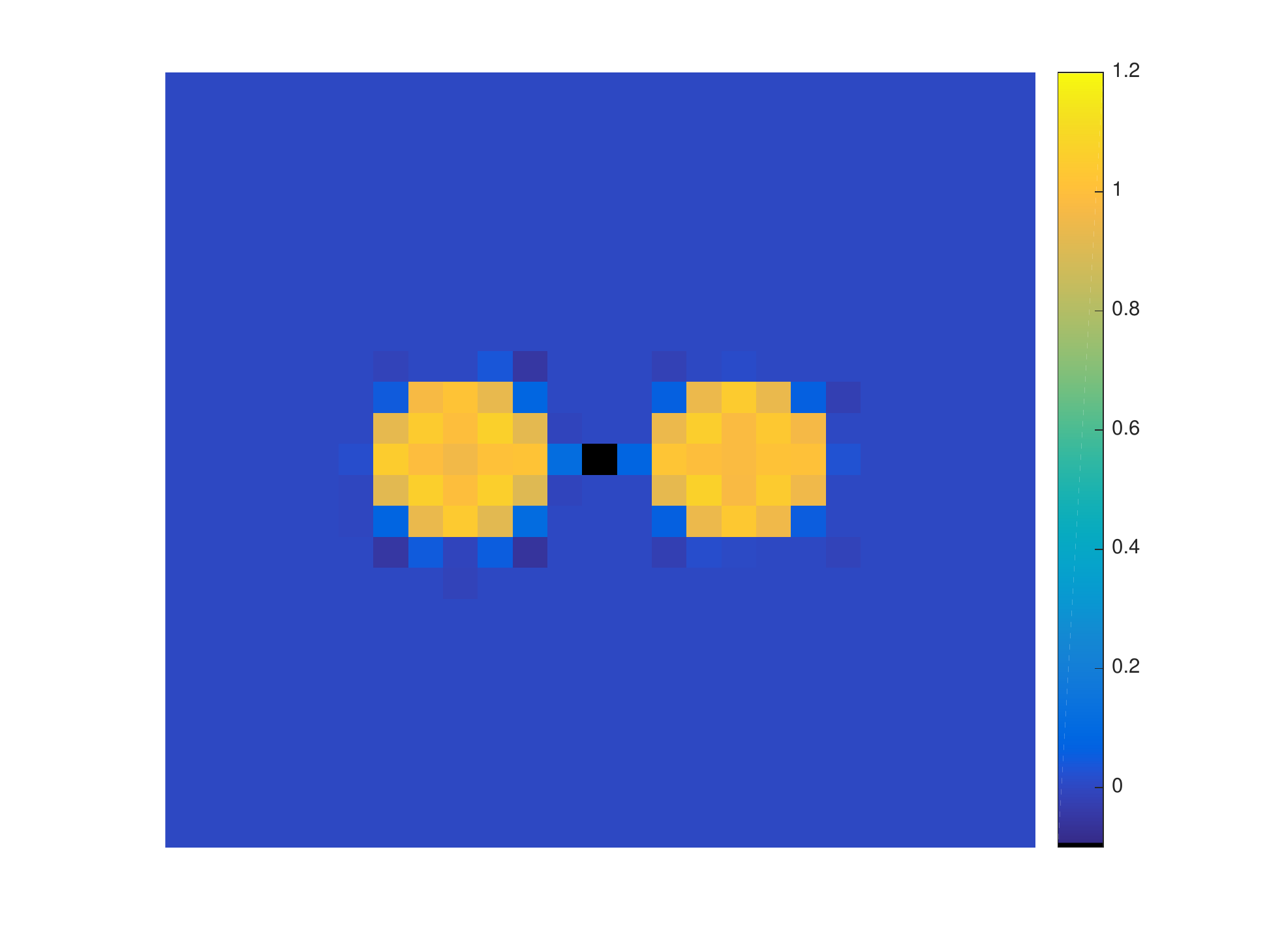} &\includegraphics[width=0.25\textwidth]{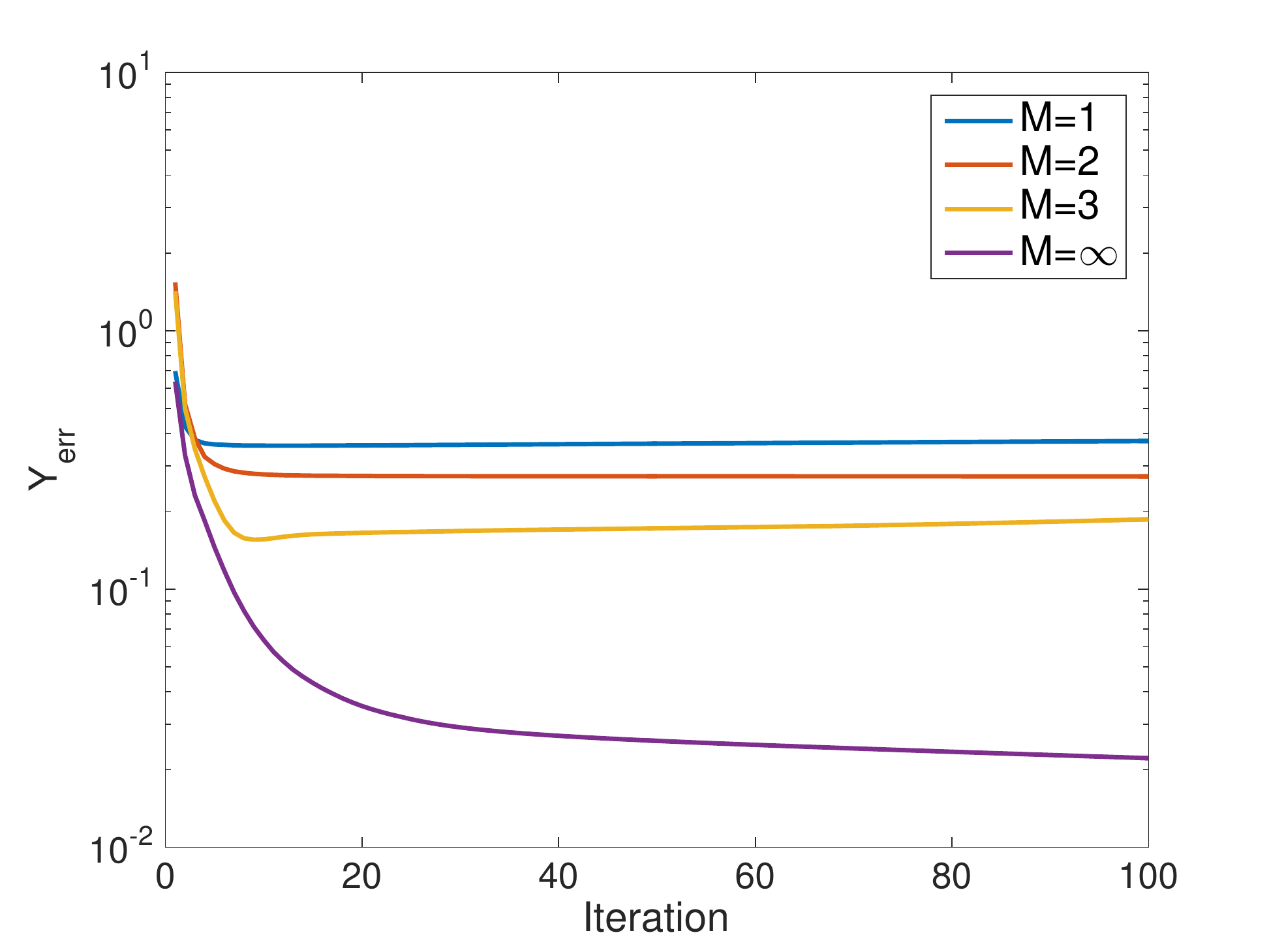}  \\ \midrule
                                               $\eta_0=0.4$ & \includegraphics[width=0.25\textwidth]{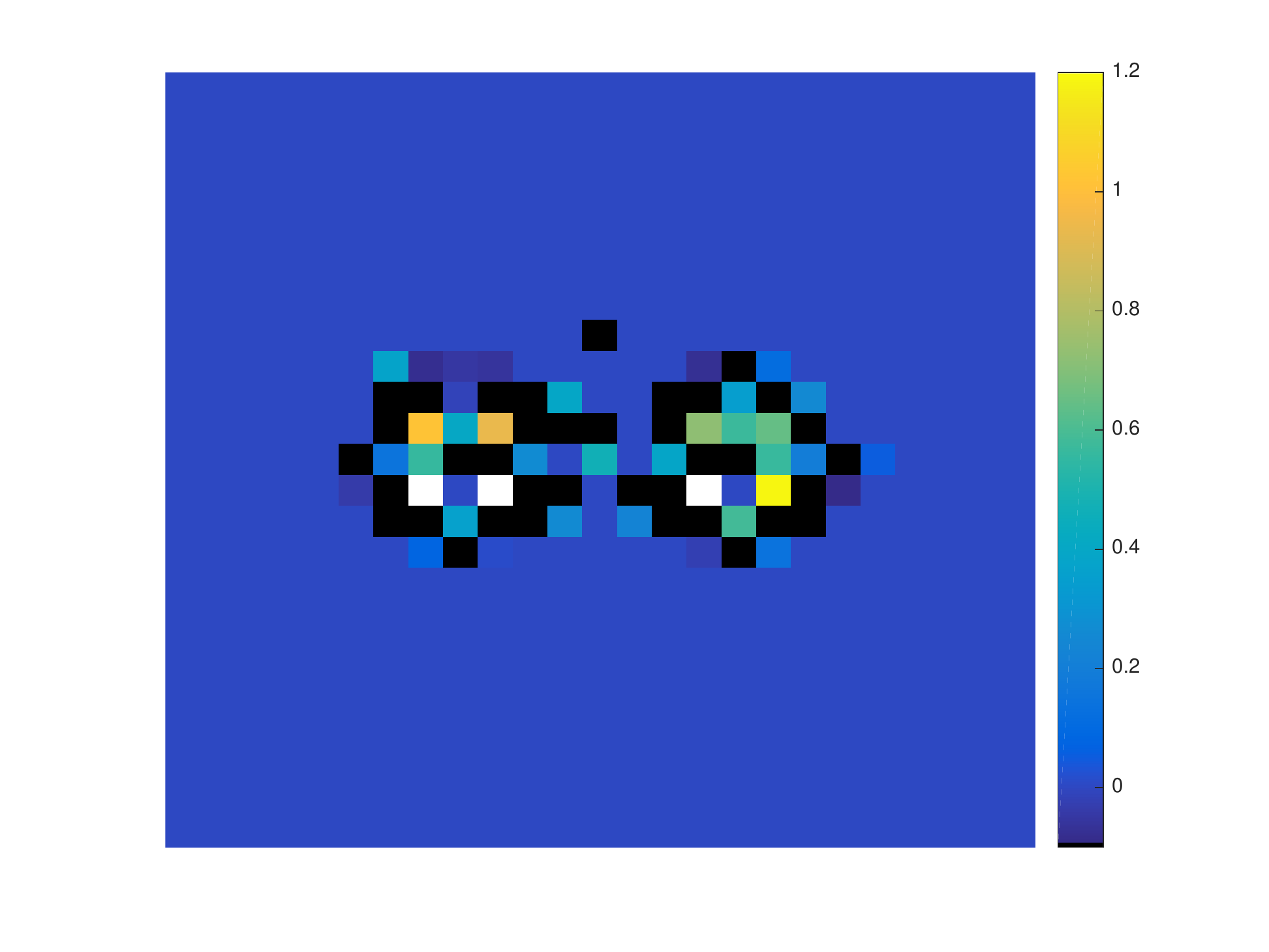} & \includegraphics[width=0.25\textwidth]{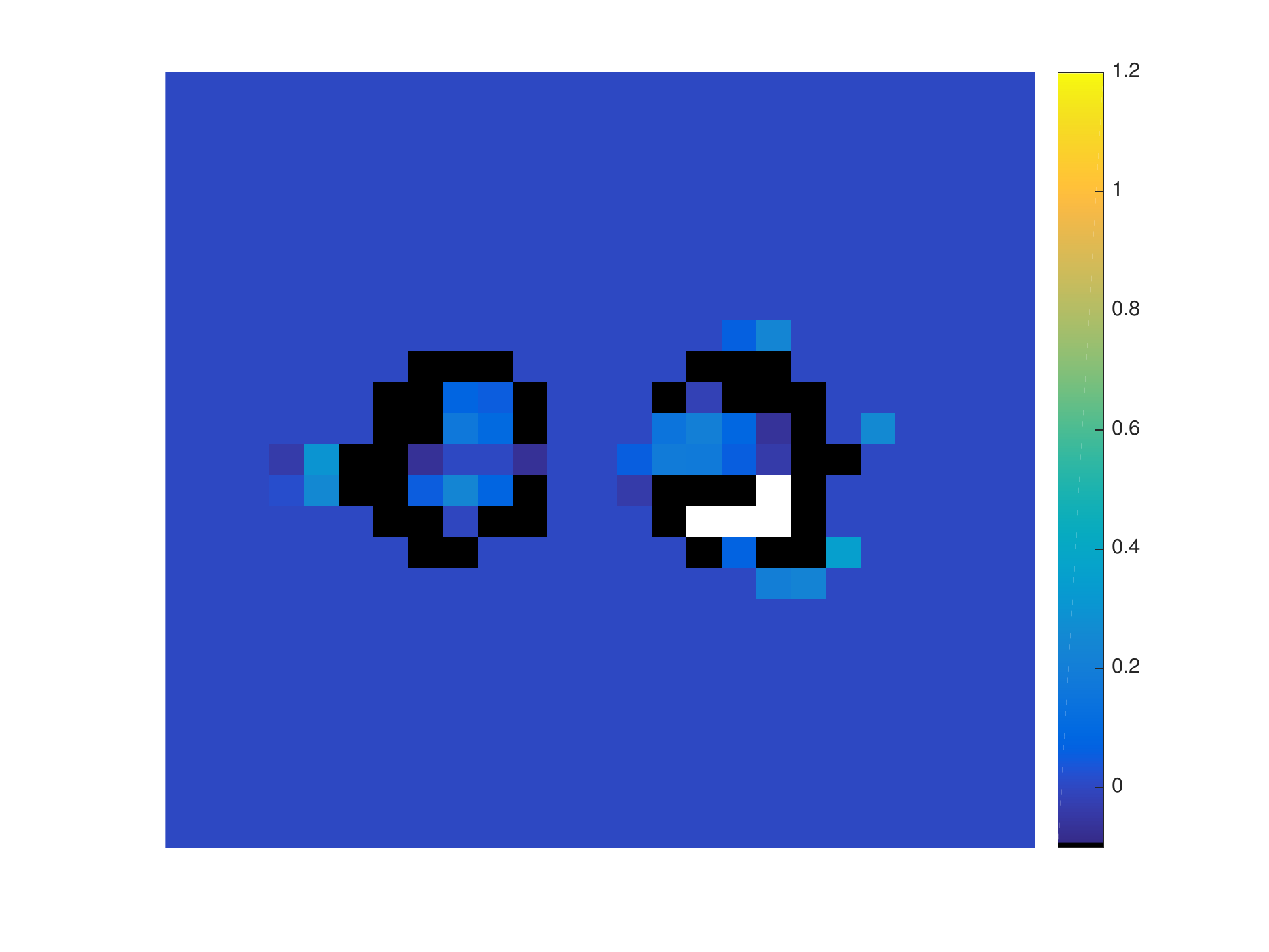} &
 \includegraphics[width=0.25\textwidth]{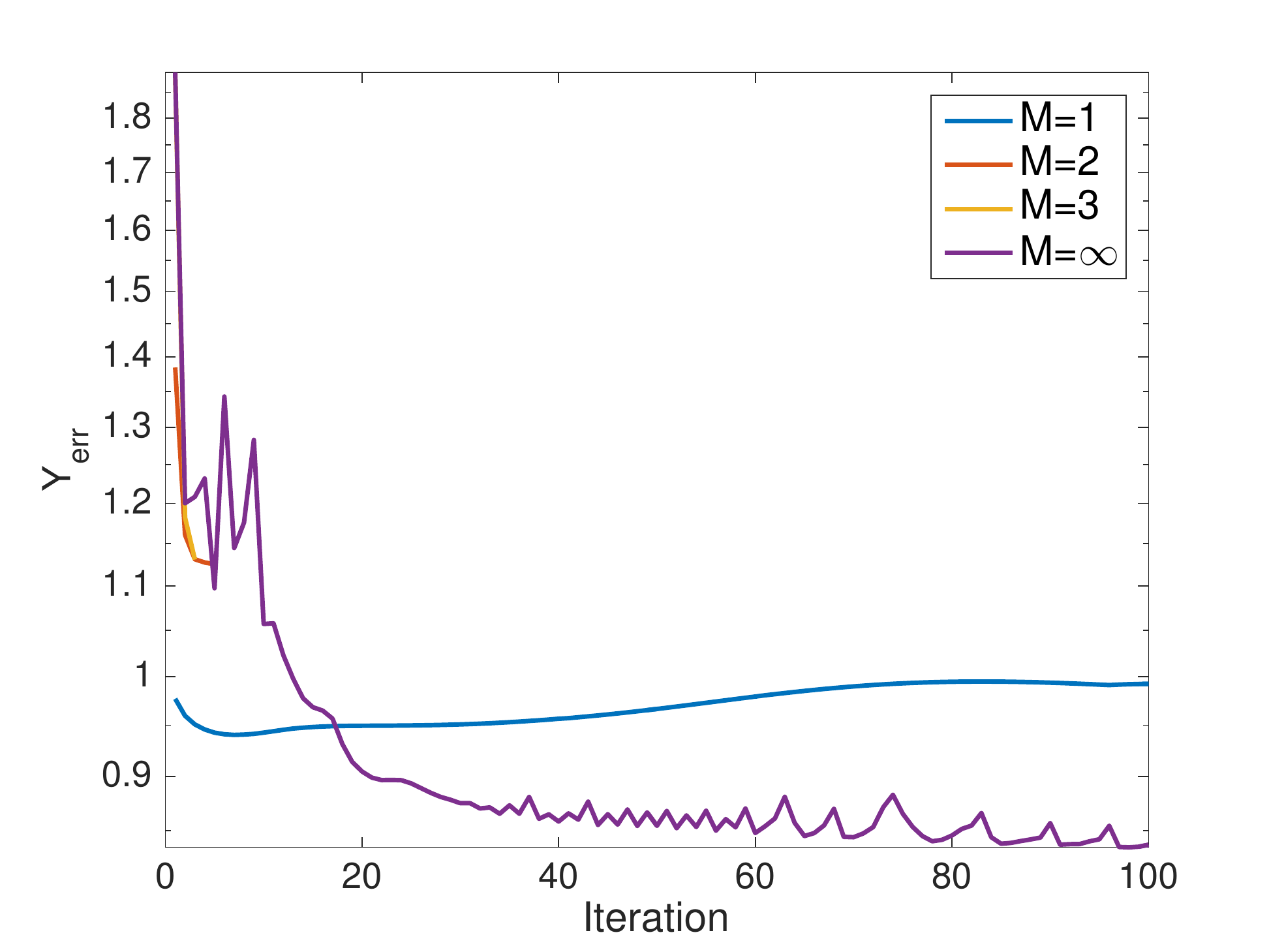} \\
          \bottomrule
        \end{tabular}
        \bigskip
        \caption{The central slice of the reconstructions of model 1 for varying susceptibilities.  A comparison of the linear IHT and the fully nonlinear IHT algorithms is shown.  The last column plots the error $Y_{\rm err}$. The field of view in each image is $5\lambda \times 5 \lambda$.
}
        \label{tbl:model1}
    \end{table}

        \begin{table}[h!]
        \centering
        \begin{tabular}{|c|c|c|c|}
           \toprule
            & $M=1$ & $M=2$ & $M=3$  \\
            \midrule
            $\eta_0=0.06$ & \includegraphics[width=0.27\textwidth]{3L1-eps-converted-to.pdf} & \includegraphics[width=0.27\textwidth]{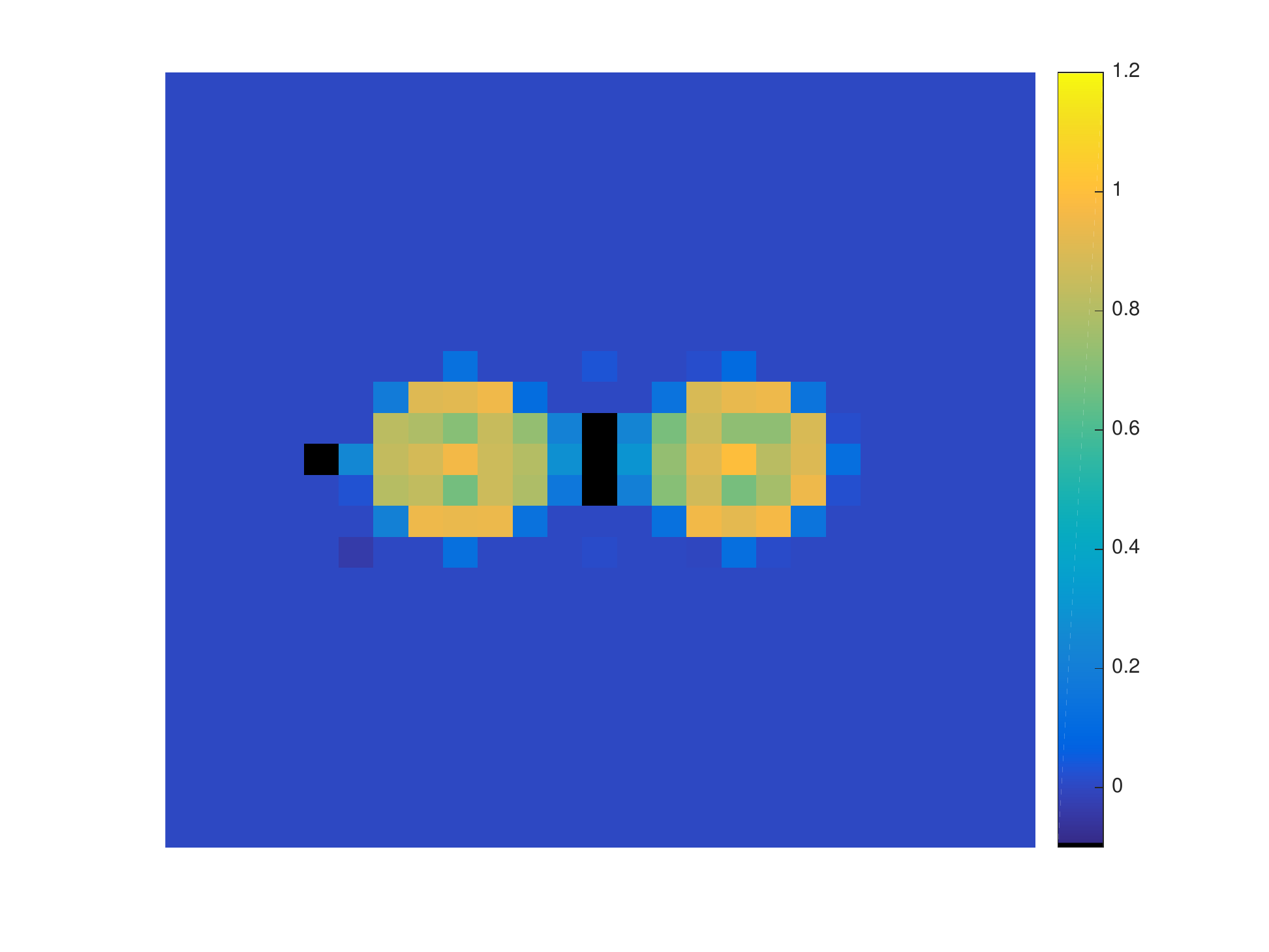} & \includegraphics[width=0.27\textwidth]{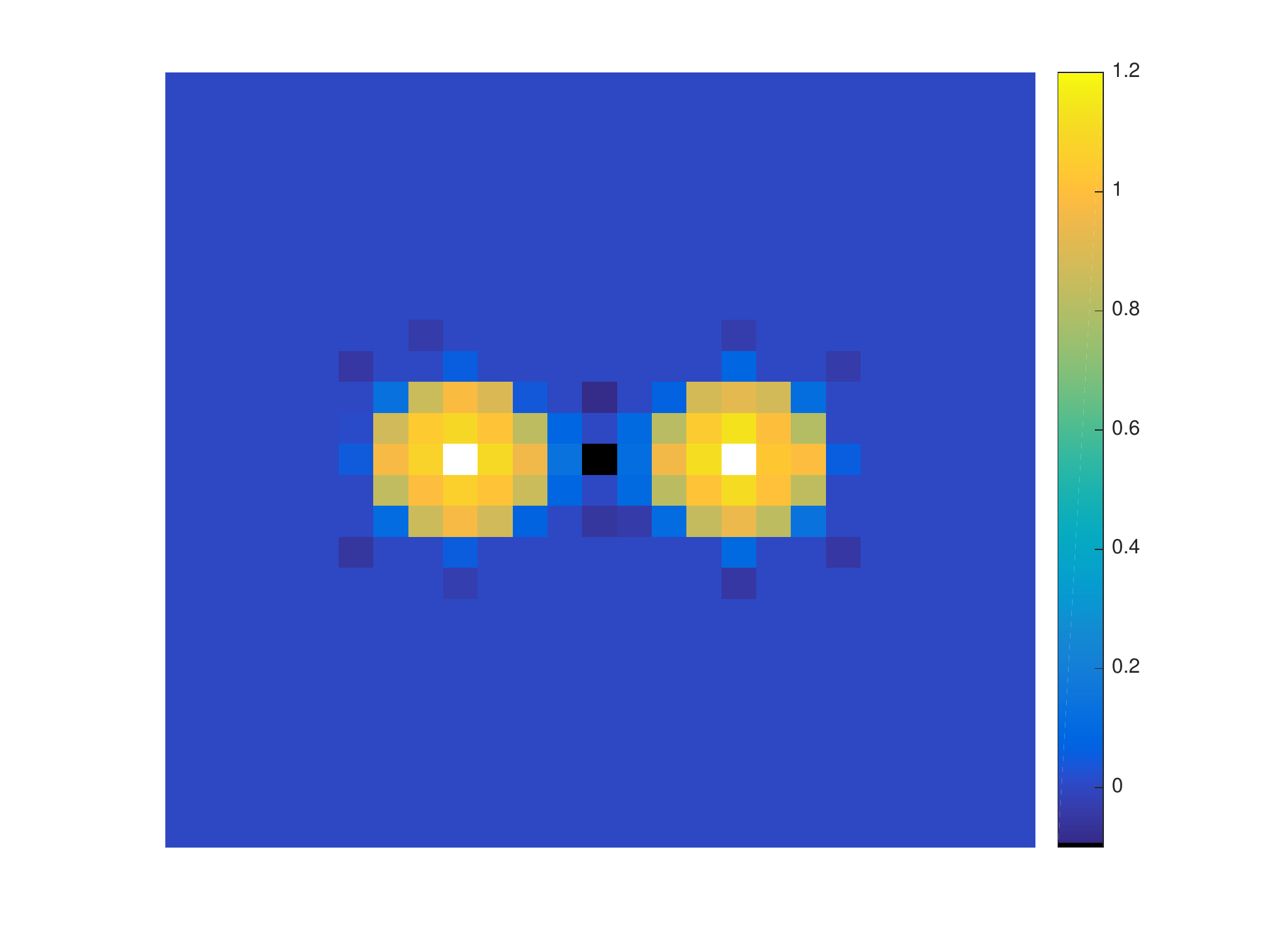}\\  \midrule
            $\eta_0=0.1$ & \includegraphics[width=0.27\textwidth]{2L1-eps-converted-to.pdf} & \includegraphics[width=0.27\textwidth]{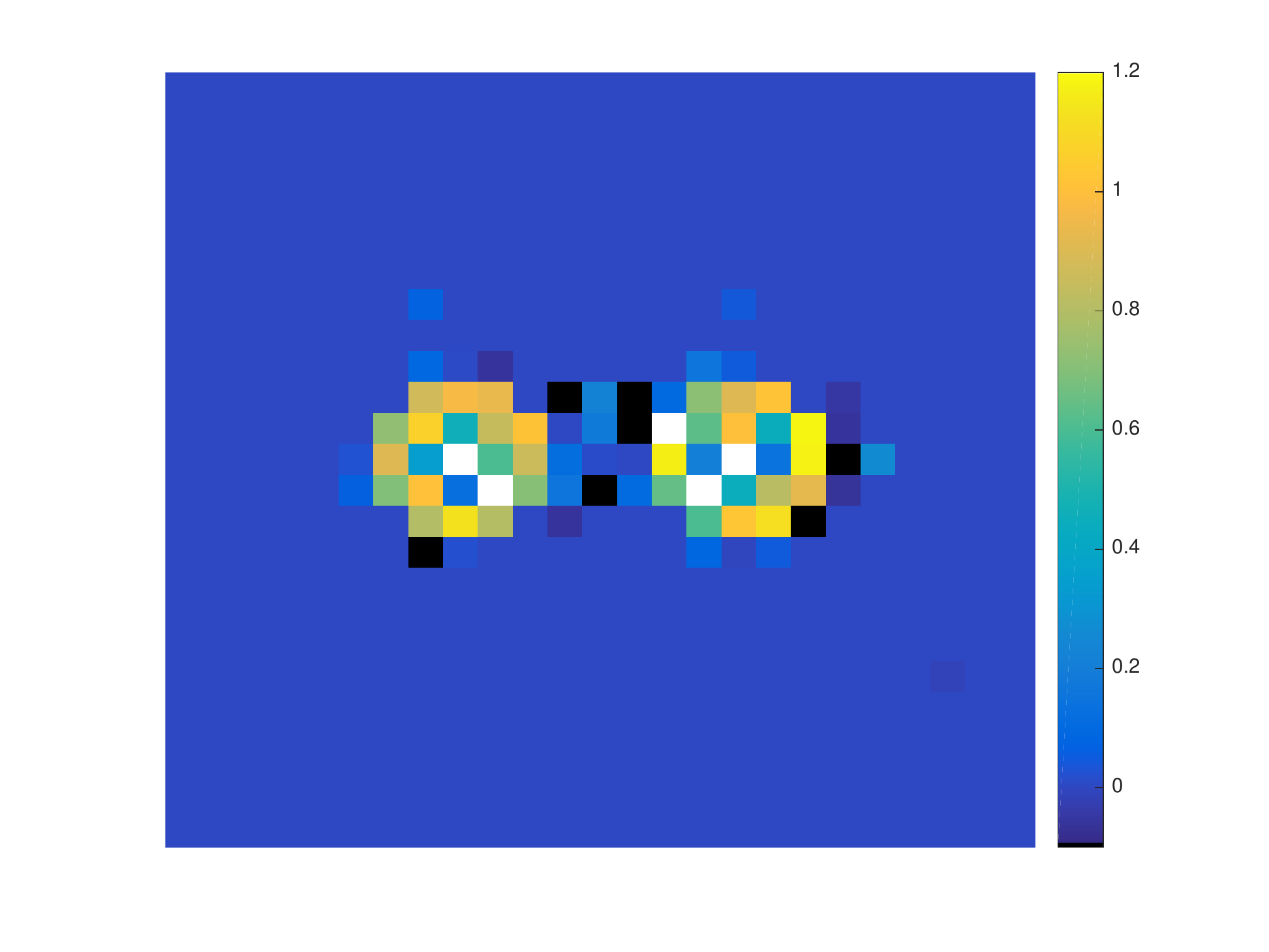} & \includegraphics[width=0.27\textwidth]{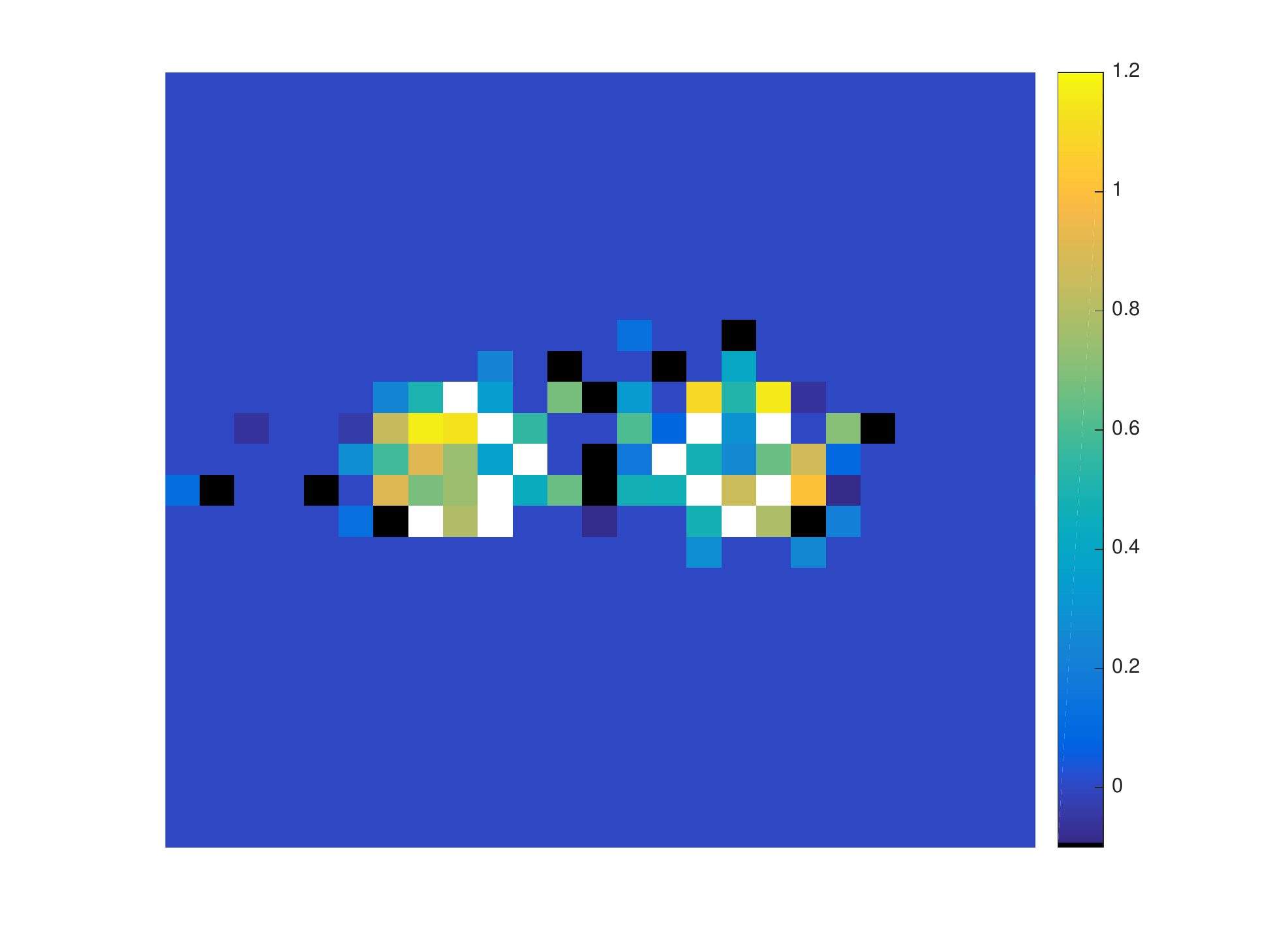} \\
          \bottomrule
        \end{tabular}
        \bigskip
        \caption{
The central slice of the reconstructions for model 1 with $\eta_0$=0.06 and 0.1. The reconstructions were conducted with the linear, second and third Born approximations. The Born series converges for the top row and diverges for the bottom row. Note that in both cases, the fully nonlinear IHT algorithm converges. The field of view in each image is $5\lambda \times 5 \lambda$.   
}
        \label{tbl:model1-2}
    \end{table}  

Next we consider the reconstruction of model 2 using the same parameters as for model 1 with $\|V\Gamma\|_2=0.9813$. This model is substantially less sparse, with 6.5\% sparsity.  A thresholding limit of 1,800 was chosen for all reconstructions. Reconstructions of the central slices of model 2 using the linear and fully nonlinear algorithms 
are displayed in Table~\ref{tbl:model1-2}. Tomographic reconstructions are also shown in Fig.~\ref{fig:rec2_1}. Central slices of the reconstructions for varying number of terms in the Born series are shown in Fig. 8. The convergence of the algorithms is illustrated in Fig. 9. We note that as in model 1, the linear reconstruction algorithm finds the support well and higher orders of scattering recover the interior susceptibilities.

\begin{figure}[t]
\centering
\includegraphics[width=0.65\linewidth]{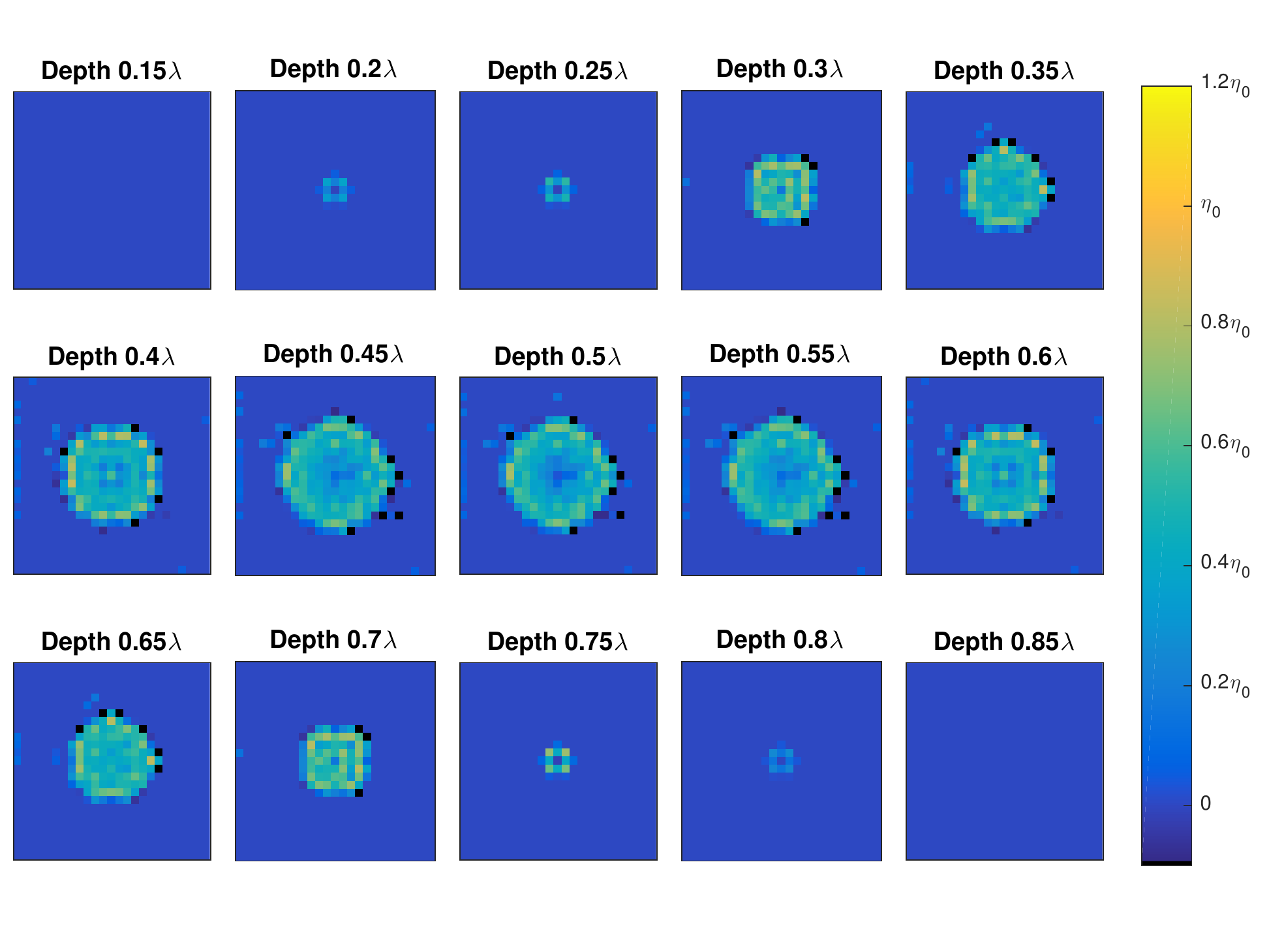}
\includegraphics[width=0.65\linewidth]{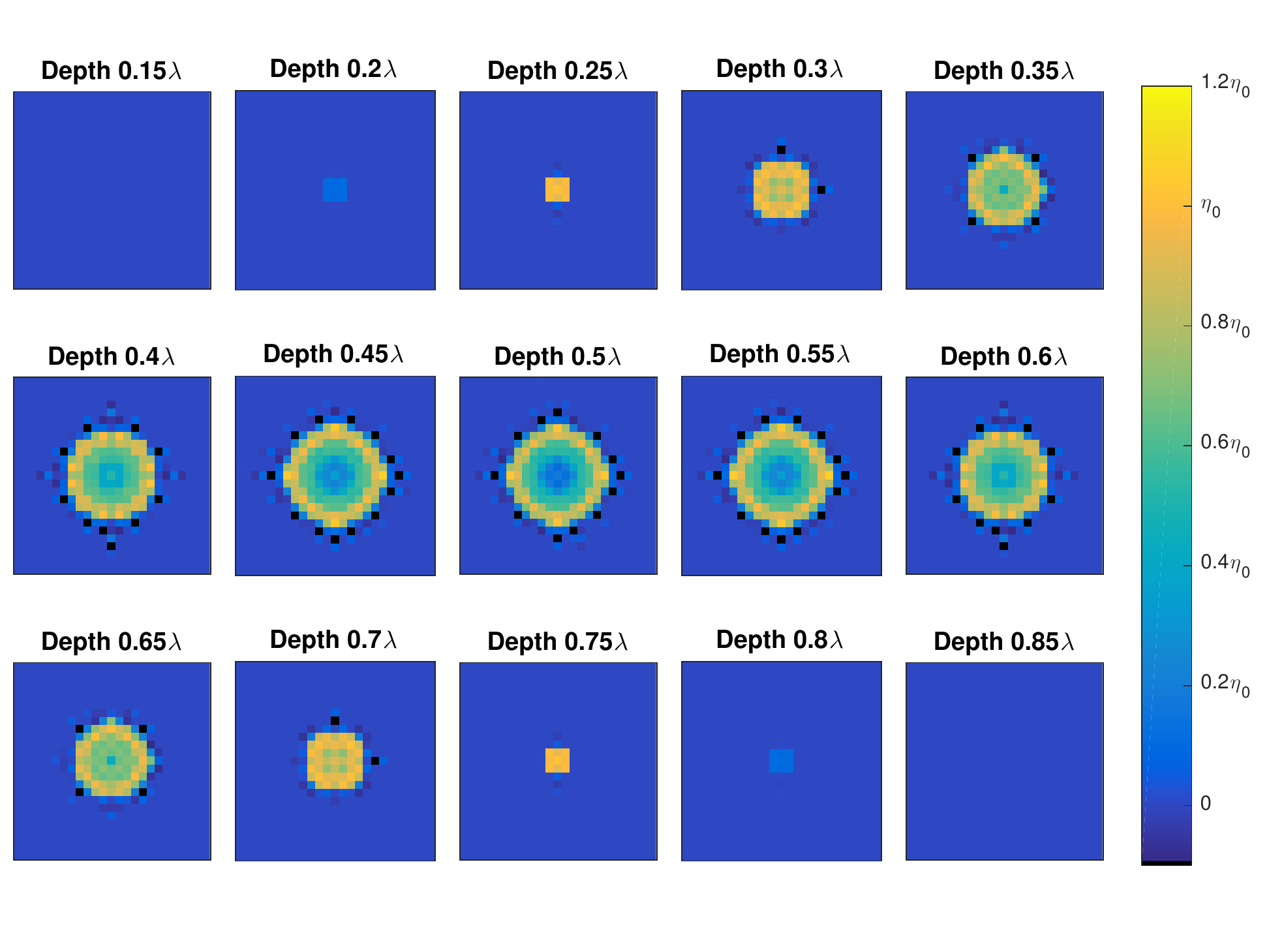}
\caption{Tomographic reconstructions of Model 2 with $\eta_0=0.09$. A comparison of the linear (top) and fully nonlinear IHT (bottom) algorithms is shown. The field of view in each slice is $5\lambda \times 5 \lambda$.
}
\label{fig:rec2_1}
\end{figure}

\begin{figure}[t]
\centering
$M=1$ \hspace{3.2cm} $M=3$ \hspace{3.2cm} $M=5$ \\
\includegraphics[width=0.3\linewidth]{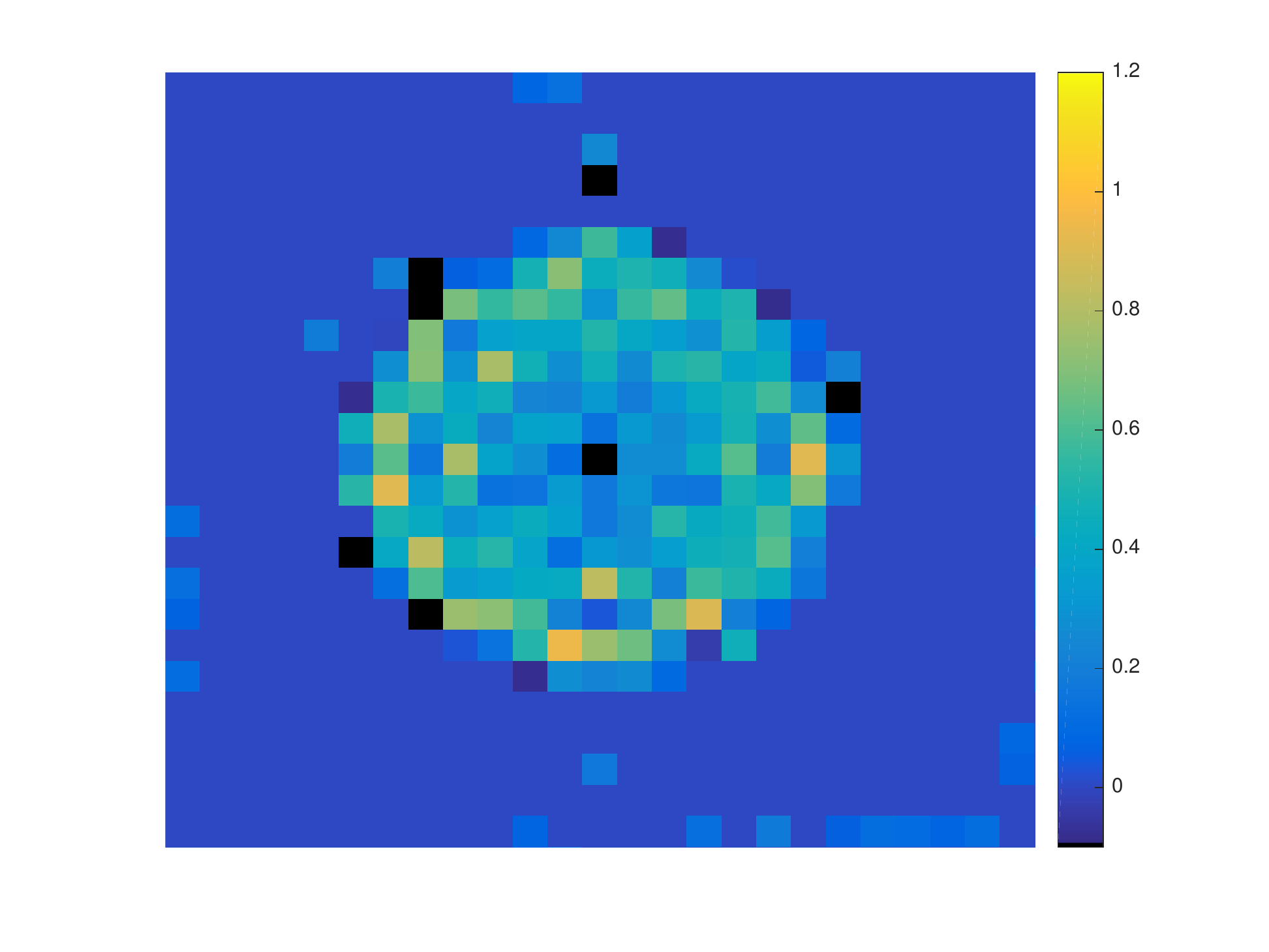}
\includegraphics[width=0.3\linewidth]{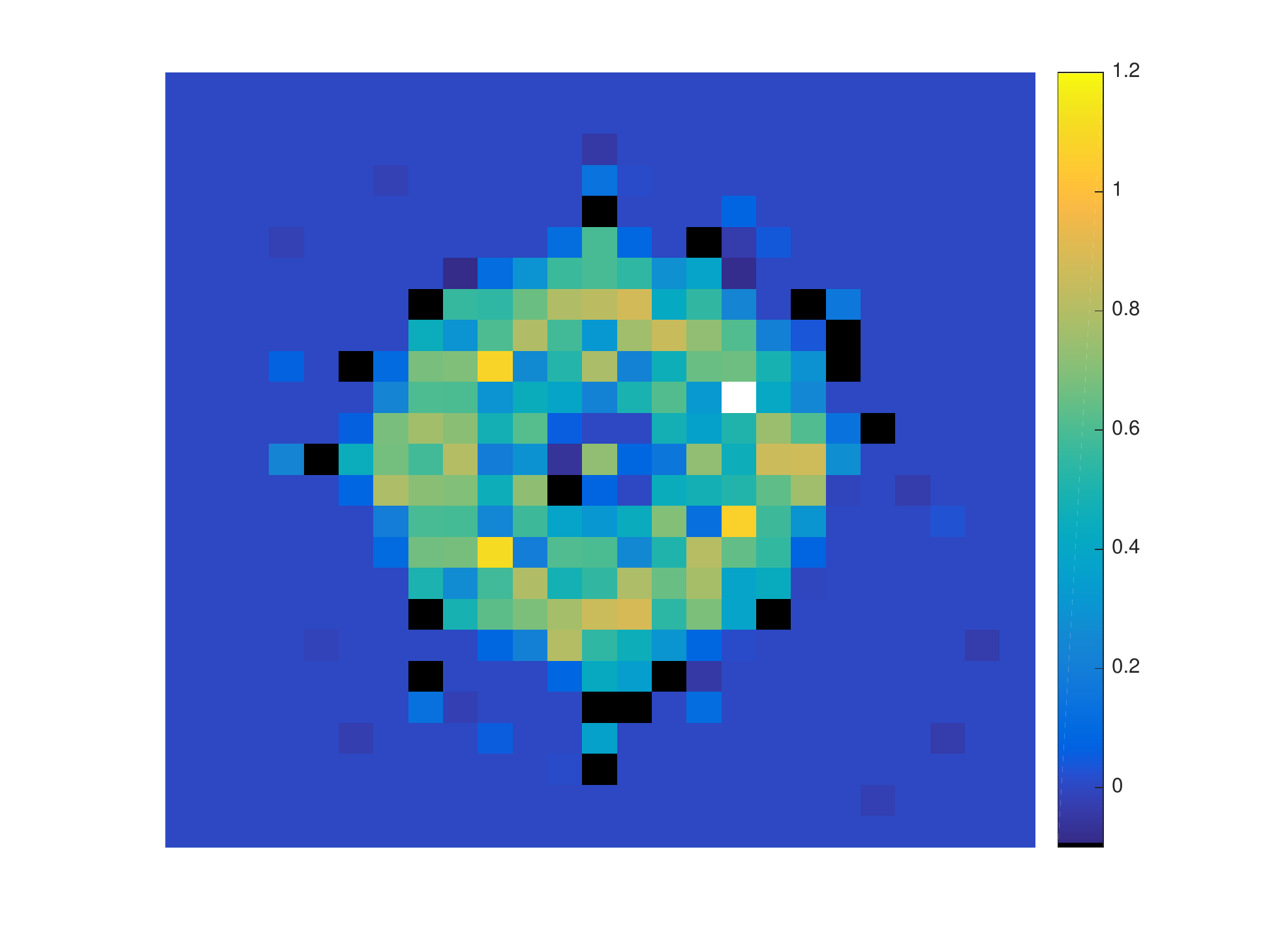} \includegraphics[width=0.3\linewidth]{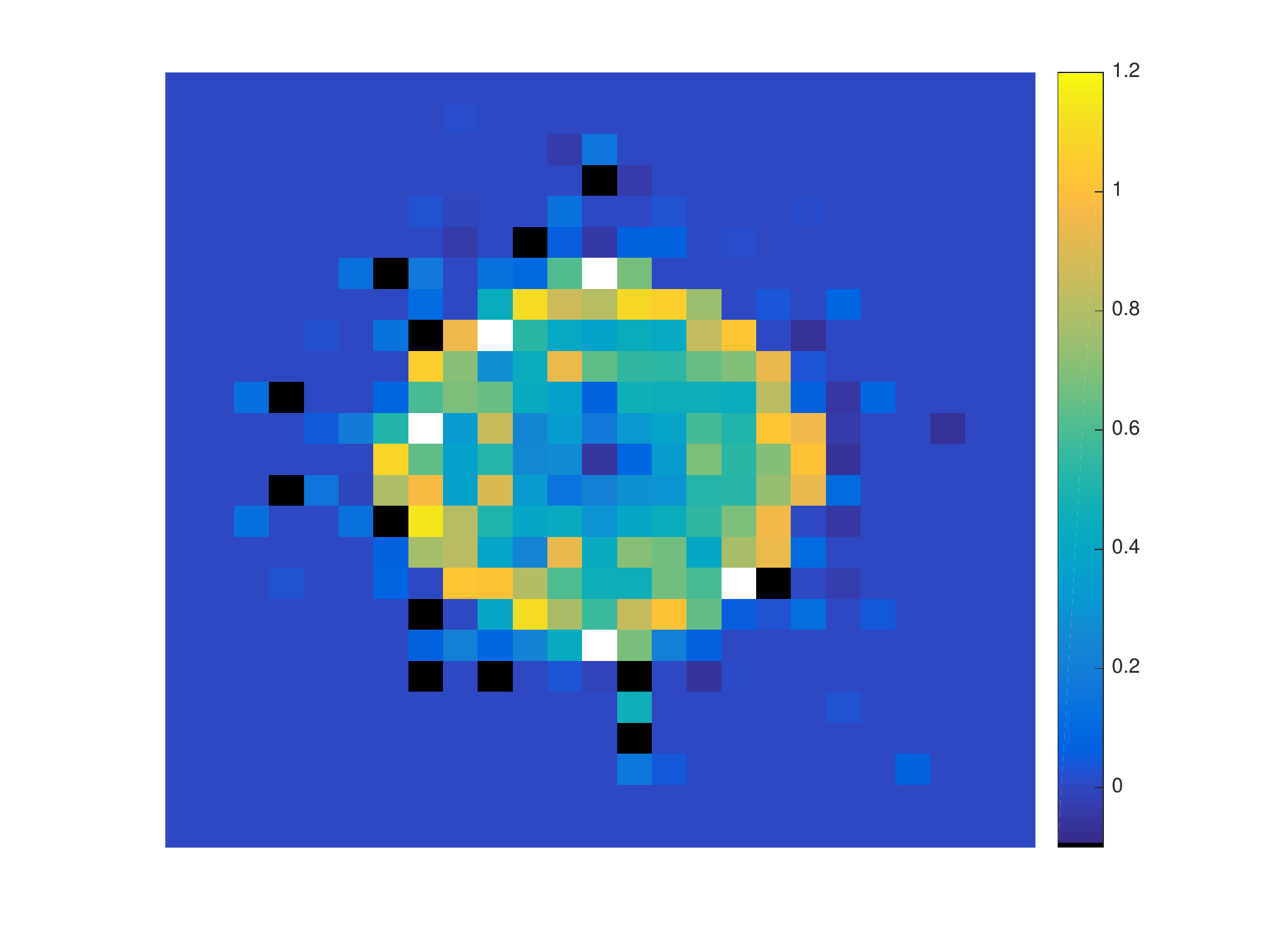} \\
$M=7$ \hspace{3.2cm} $M=9$ \hspace{3.2cm} $M=\infty$ \\
\includegraphics[width=0.3\linewidth]{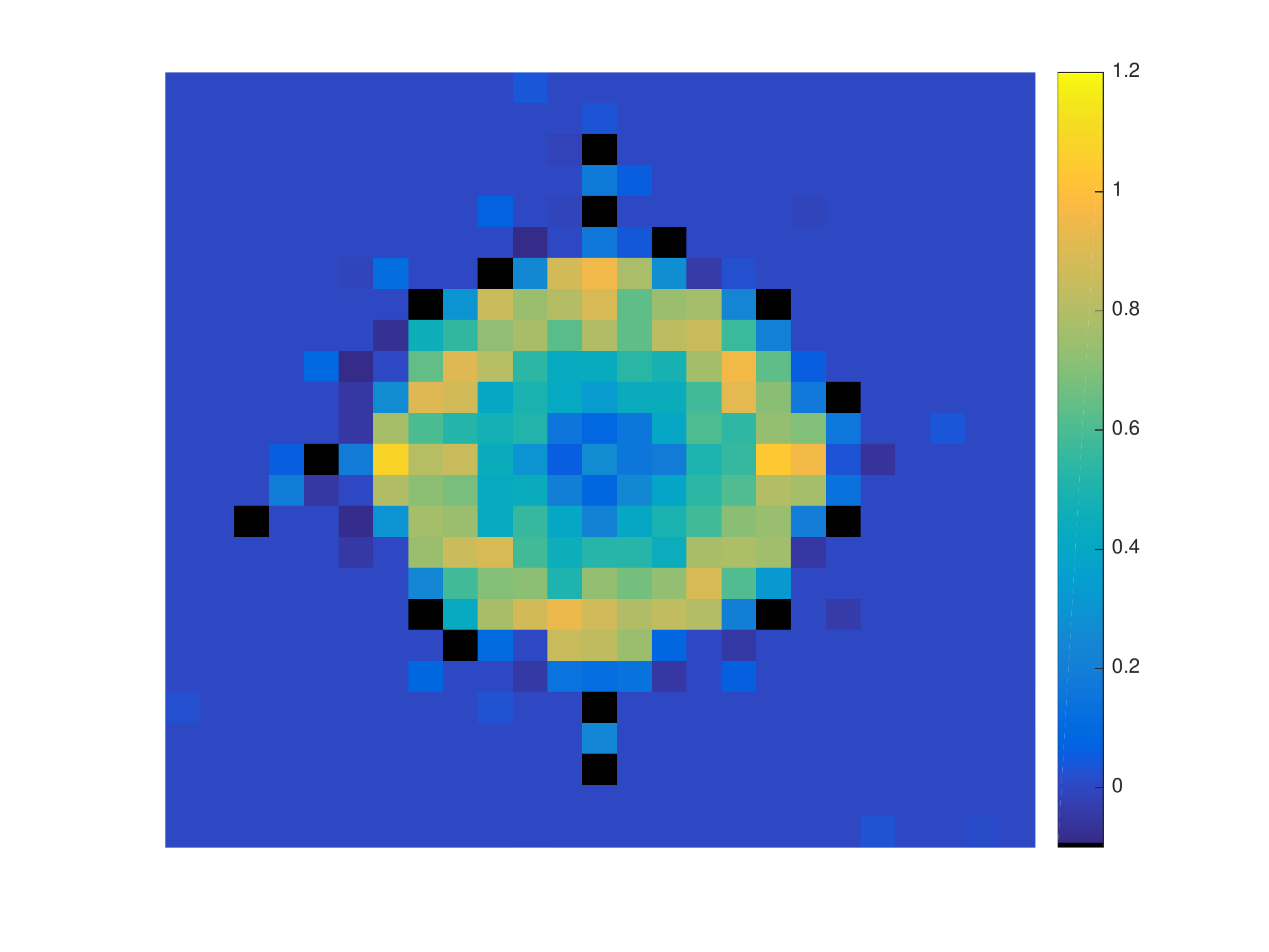}
\includegraphics[width=0.3\linewidth]{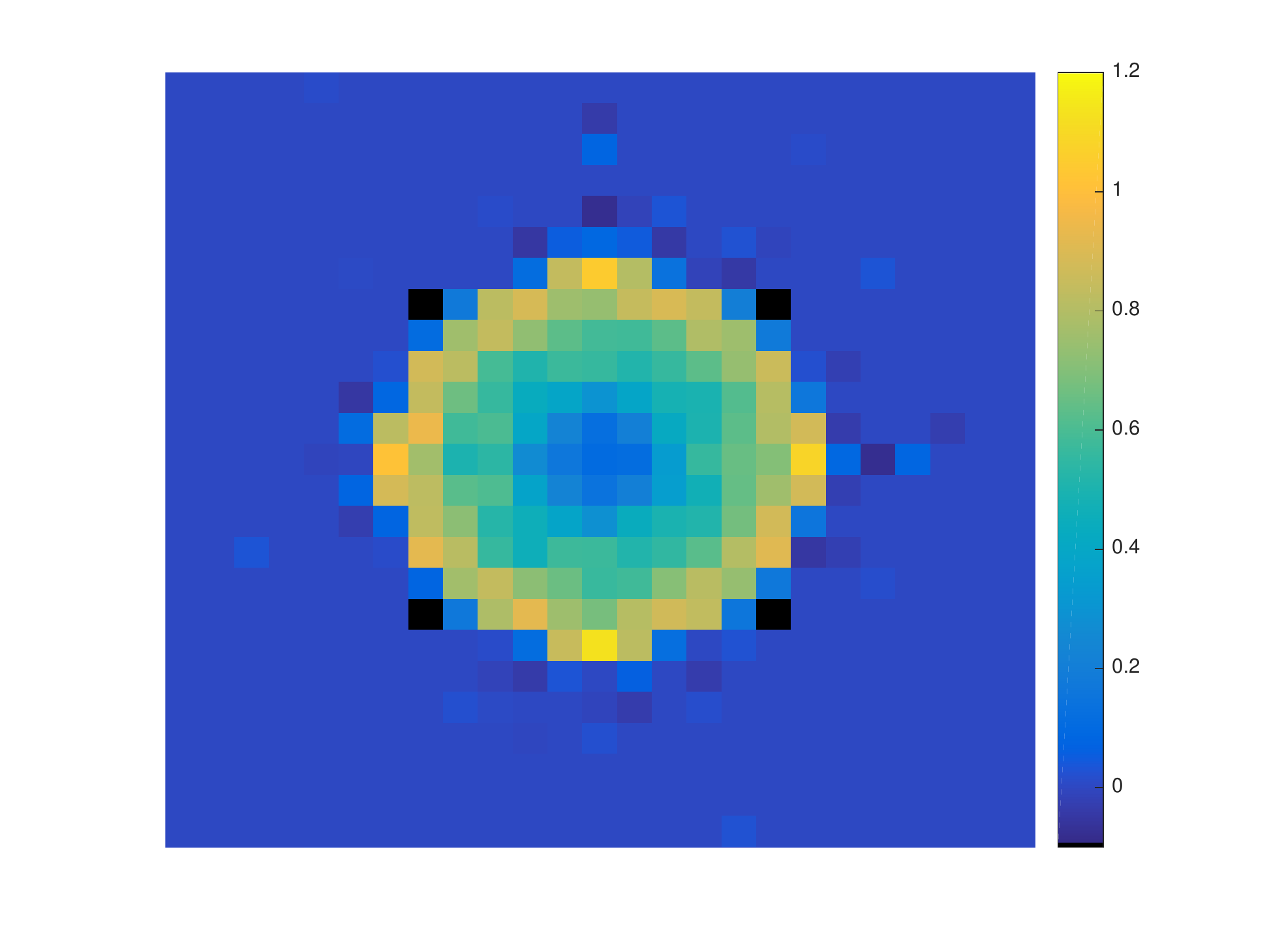}
\includegraphics[width=0.3\linewidth]{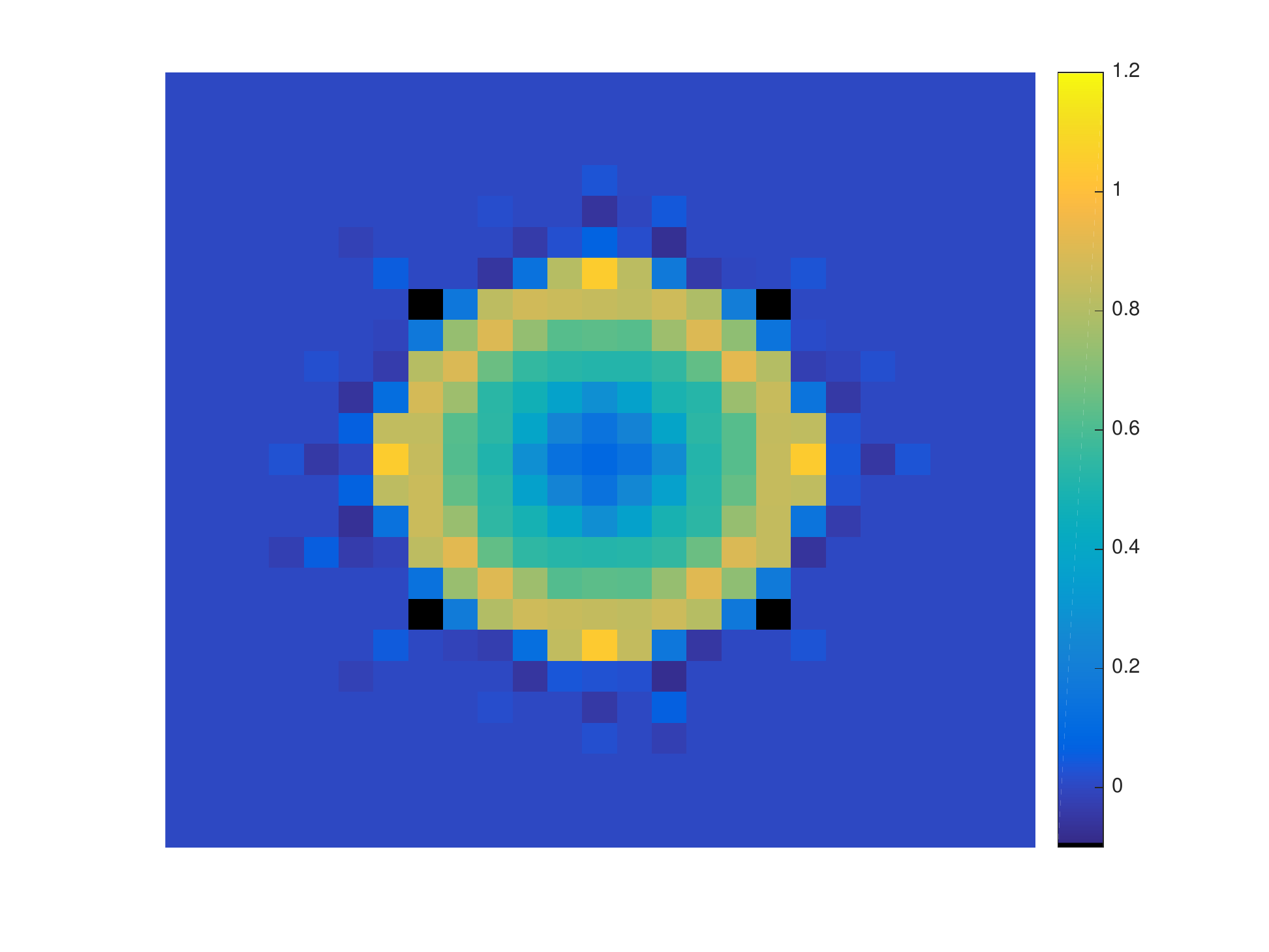}
\caption{The central slice of the reconstructions for Model 2 for $\eta_0=0.09$. A comparison of the nonlinear IHT for varying number of terms added to the linearization ($M$) is shown.  Here, the Born series converges naturally leading to improved reconstructions as more terms are added. The field of view is $5\lambda \times 5 \lambda$.
}
\label{fig:rec2}
\end{figure}

\begin{figure}[t]
\centering
\includegraphics[width=0.55\linewidth]{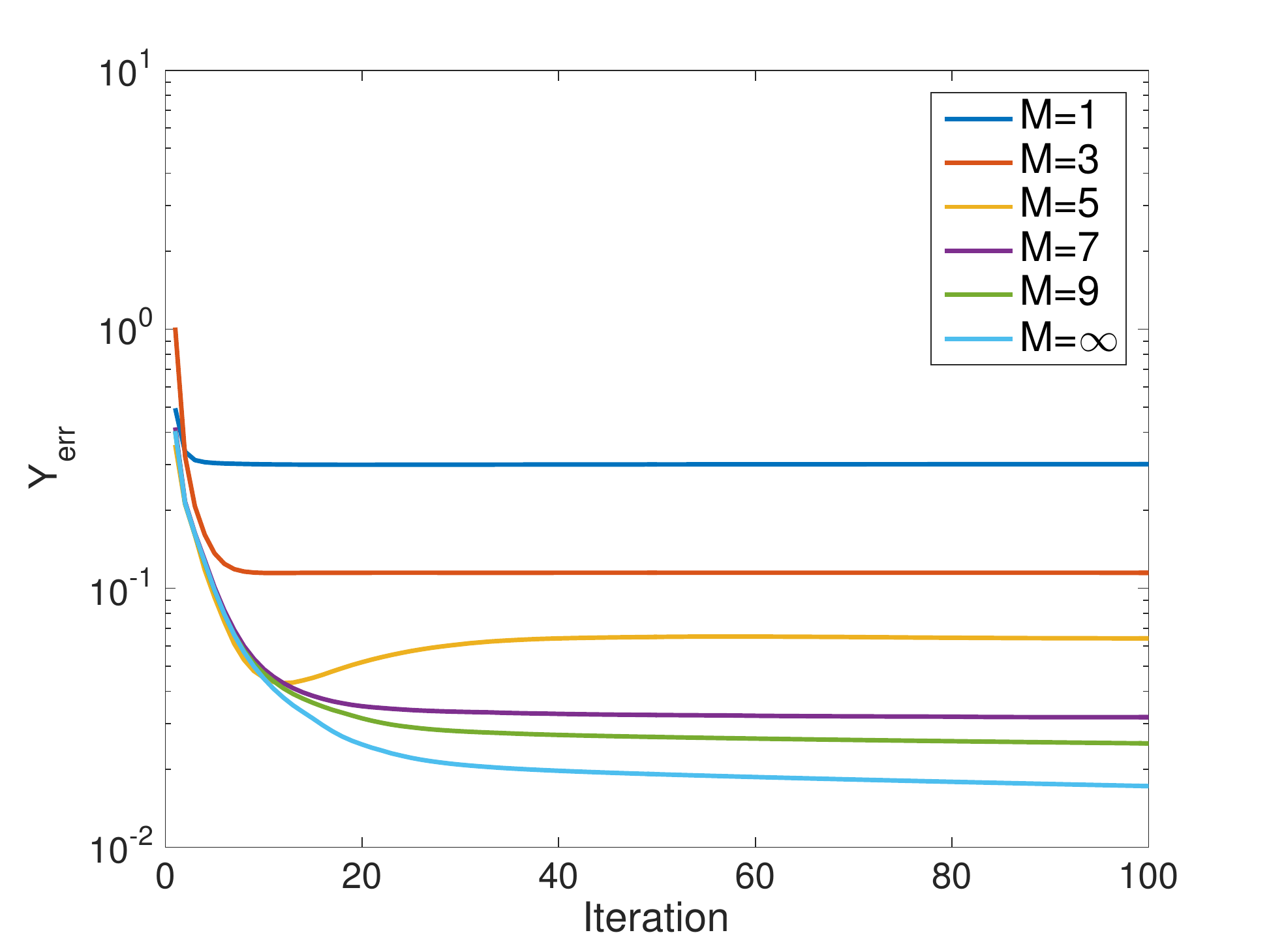}
\caption{Error $Y_{\rm err}$ for Model 2, corresponding to the images in Fig.~\ref{fig:rec2}.  
}
\label{fig:err2}
\end{figure}

Using the same experimental setup ($N_s=N_d=225$), $\Omega$ was discretized into 1,000 voxels, where $s$ voxels were chosen uniformly at random to have potential $\eta_0$.  50 iterations of the IHT algorithms were performed for $M=1,2,3,4,\infty$.  Each experiment was performed with 500 realizations with the hard thresholding limit set to the sparsity level $s$.  We define a successful reconstruction as {\it exactly} recovering the support of all $s$ voxels.   Fig.~\ref{fig:suc_plots} shows the success rate for two choices of $\eta_0$ with varying levels of sparsity.
\begin{figure}[t]
\centering
\begin{subfigure}[b]{0.45\textwidth}
                \centering
                \includegraphics[width=\textwidth]{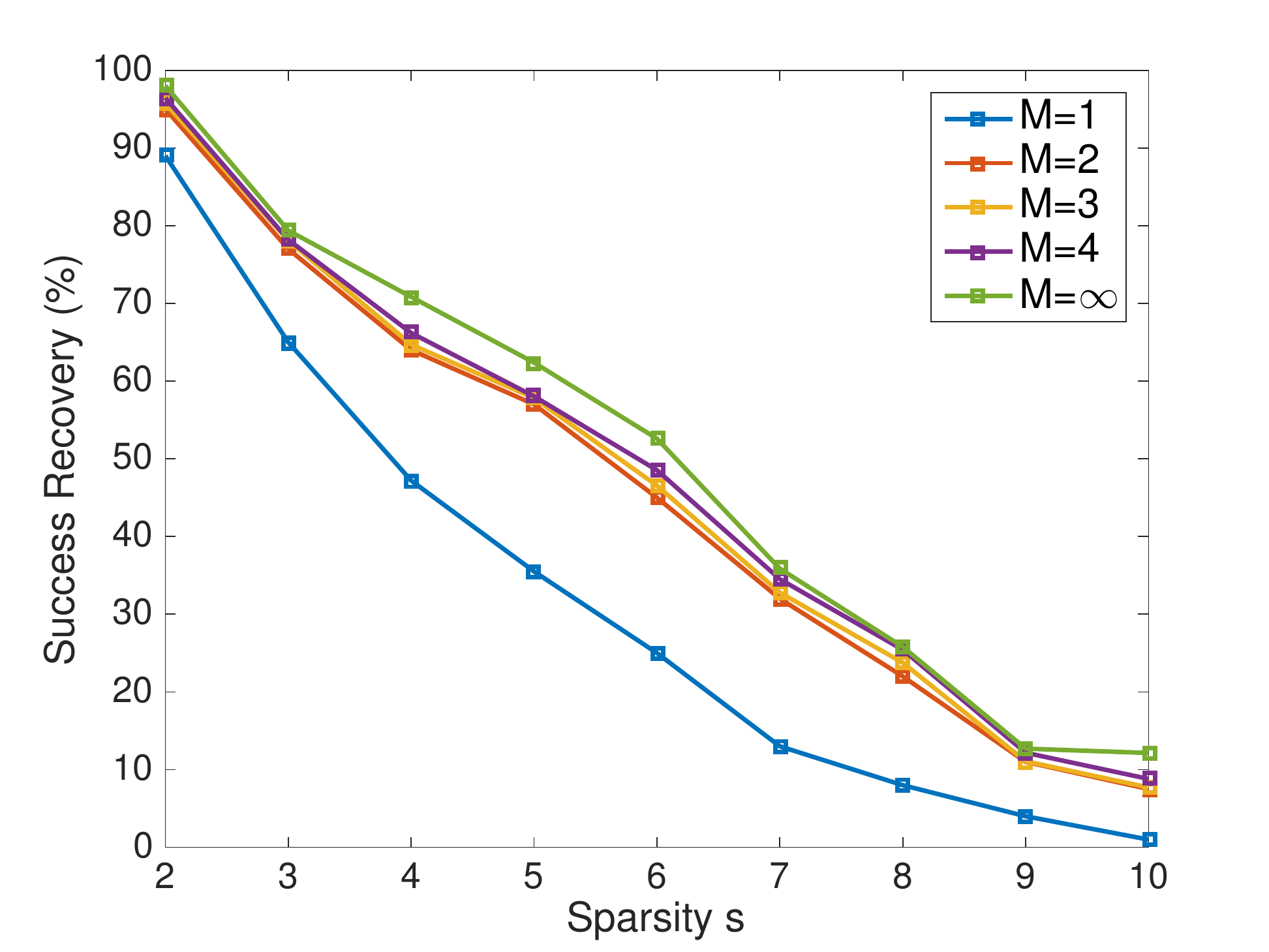}
                \caption{$\eta_0=0.05$}
        \end{subfigure}
\begin{subfigure}[b]{0.45\textwidth}
                \centering
                \includegraphics[width=\textwidth]{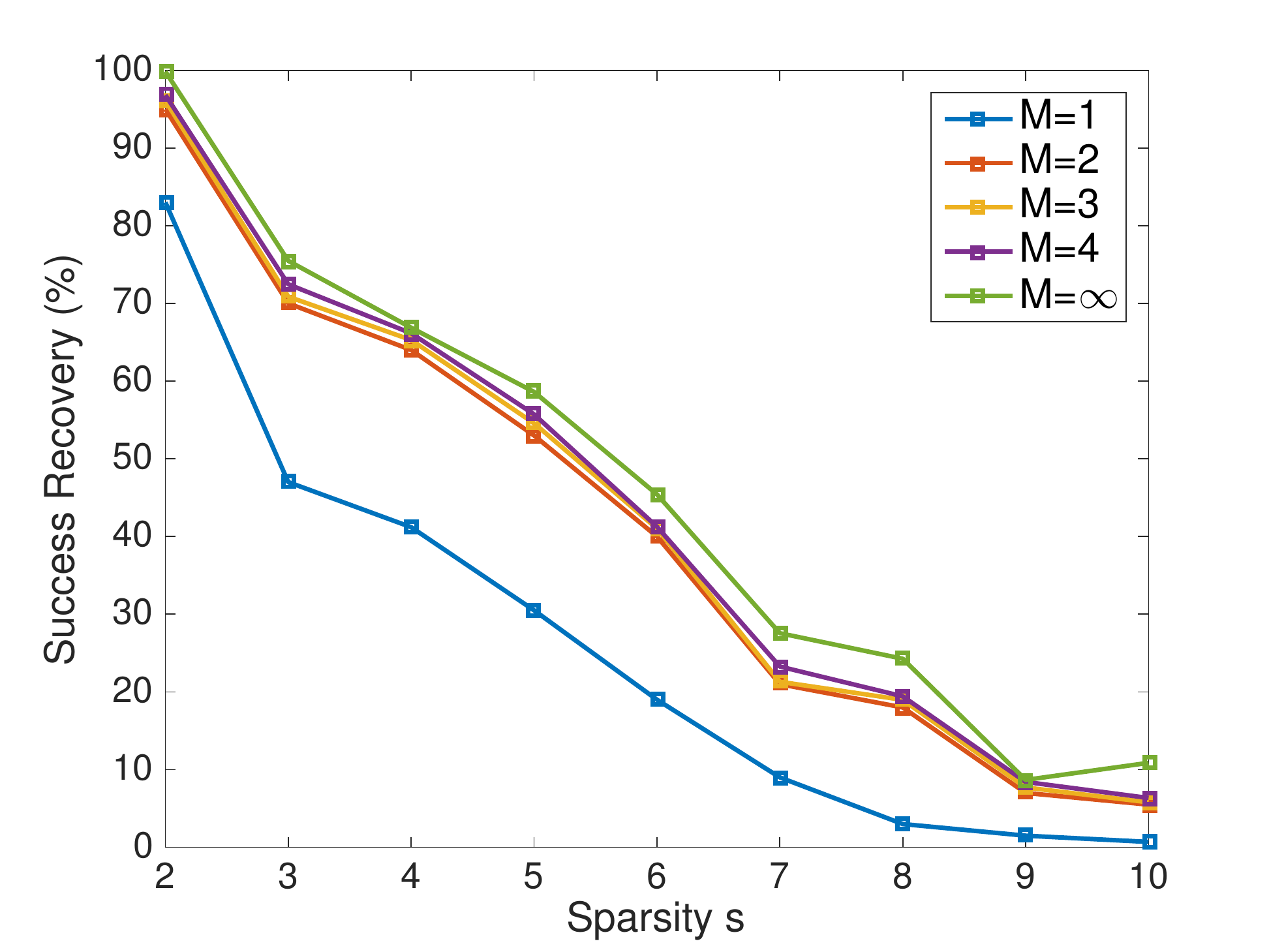}
                \caption{$\eta_0=0.1$}
        \end{subfigure}
\caption{Success rate of support recovery for various forms of the nonlinear iterative hard thresholding algorithm as a function of sparsity.  The left plot has $\eta_0=0.05$ and the right plot has $\eta=0.1$.
}
\label{fig:suc_plots}
\end{figure}
We note that since the targets are quite sparse, the second Born approximation accounts for the majority of the improvement of the nonlinear reconstruction over linear IHT.  

\section{Discussion}

We have considered the application of nonlinear IHT to inverse scattering. The convergence and error of the method was analyzed by means of coherence estimates and compared to numerical simulations. It was found that in the multiple-scattering regime, the performance of nonlinear IHT is superior to linear IHT. Moreover, numerical evidence suggests that the coherence estimates we have obtained are relatively conservative. Future work will be concerned with average case coherence estimates, which may be anticipated to lead to a more accurate characterization of the convergence of the nonlinear IHT algorithm. It would also be of interest to explore the role of near-field measurements in improving the resolution of reconstructed images. Other areas of future interest include the extension of our results to the setting of the Bremmer series, which is well adapted to the study of layered media and geophysical applications.

\appendix
\section{Proof of Theorem \ref{thm:coherence}} 
\label{apx:proof}

The proof follows the general ideas in~\cite{wang}.  
First, we assume without loss of generality that the columns of the sensing matrix are normalized so that their $\ell_2$ norms are unity. Then, we use the following lemmas:

\begin{lemma}
\label{lem:coherence}
Let $A\in\mathbb{C}^{m\times n}$ be normalized such that $\|A_i\|_2=1$ for all $i=1,\dots,n$.  Then,
\begin{equation}
\mu(A)= \sup_{x\in\mathbb{R}^n\setminus\{0\}} \frac{\|(I-A^*A)\Bx\|_\infty}{\|\Bx\|_1}
\end{equation}
\end{lemma}

\begin{lemma}
Let $\Bz_{n+1}=\Bx_{n}+\Phi^*_{\Bx_n}(\By-\Phi_{\Bx_n} \Bx_{n})$.  
Let $I^*$ be the index set of nonzero entries of $\Bx$, $I^n$ be the index set of the $s$ largest entries of $\Bz_n$, and  $I^n_+$ be the index set of the $s+1$ largest entries of $\Bz_n$.  Let $S_n=I^n_+\bigcup I^{n-1}\bigcup I^*$.  Let $\mathbf{w}_n$ be the subvector of $\Bz_n$ with components restricted to $S_n$. Then, the following inequality holds
\begin{equation}
\label{eq:ineq1}
\|\Bx_n-\Bx\|_1\le (3s+1)\|\mathbf{w}_n-\Bx\|_\infty \ . 
\end{equation}
\end{lemma}
\noindent The proofs of these lemmas can be found in \cite{wang}.

Using $\Bz_{n+1}=\Bx_n+\Phi^*_{\Bx_n}(\By-\Phi_{\Bx_n} \Bx_{n})$, we find that
\begin{align}
\Bz_{n+1}-\Bx = (I-\Phi^*_{\Bx_n}\Phi_{\Bx_n})(\Bx_n-\Bx)+\Phi^*_{\Bx_n}\big(\Phi(\Bx)-\Phi_{\Bx_n}\Bx+\Beps\big).
\end{align}
Let $A=I-\Phi^*_{\Bx_n}\Phi_{\Bx_n}$, where by hypothesis the linearized matrix $\Phi_{\Bx_n}$ has normalized columns.  Then,
\begin{align}
\|\Bz_{n+1}-\Bx\|_\infty= & \|(I-\Phi^*_{\Bx_n}\Phi_{\Bx_n})(\Bx_n-\Bx)+\Phi^*_{\Bx_n}\big(\Phi(\Bx)-\Phi_{\Bx_n}\Bx+\Beps\big)\|_\infty \\
\le & \|(I-\Phi^*_{\Bx_n}\Phi_{\Bx_n})(\Bx_n-\Bx)\|_\infty + \|\Phi^*_{\Bx_n}\big(\Phi(\Bx)-\Phi_{\Bx_n}\Bx+\Beps\big)\|_\infty \\
\le & \mu_0 \|\Bx_n-\Bx\|_1+\|\Phi^*_{\Bx_n}\Be_n\|_\infty \ ,
\end{align}
where the last inequality follows from Lemma \ref{lem:coherence}.  The above inequality holds for any subvector of $\Bz_{n+1}-\Bx$, so we have
\begin{equation}
\label{eq:ineq2}
\|\mathbf{w}_{n+1}-\Bx\|_\infty \le  \mu \|\Bx_n-\Bx\|_1+\|\Phi^*_{\Bx_n}\Be_n\|_\infty \ .
\end{equation}
By combining the inequalities in Eqs.~\eqref{eq:ineq1} and \eqref{eq:ineq2}, we obtain
\begin{equation}
\label{eq:ineq3}
\|\Bx_{n}-\Bx\|_1 \le \rho \|\Bx_{n-1}-\Bx\|_1+(3s+1)\|\Phi^*_{\Bx_n}\Be_n\|_\infty \ .
\end{equation} 
The theorem follows by iterating the above estimate. \hfill \qed

\section{RIP Analysis of Nonlinear IHT for Scattering} 
\label{apx:rip}

Here we prove a similar result to Theorem \ref{thm:fullTmatrix} using RIP rather than coherence estimates.  We first state a related result, which has been slightly simplified for our purposes.

\begin{theorem}\label{thm:B}\cite{blu_nliht}
Let $\By=\Phi(\Bx_*)$  with $\Bx_*$ $s$-sparse and let $\{\Bx_n\}$ be the sequence generated by the iteration 
\begin{equation}
\Bx_{n+1}=H_s\left(\Bx_{n}+\Phi^*_{\Bx_n}(\By-\Phi_{\Bx_n} \Bx_{n})\right) \ ,
\end{equation}
where $\Phi_{\Bx_n}$ is the linearization of $\Phi$ at $\Bx_n$.  Assume that for all $n\ge 1$, the linearizations $\Phi_{\Bx_n}$ satisfy the RIP for all $s$-sparse $\Bz_1$ and $\Bz_2$
\begin{equation}
\alpha\|\Bz_1-\Bz_2\|^2_2 \le \|\Phi_{\Bx_n}(\Bz_1-\Bz_2)\|^2_2 \le \beta\|\Bz_1-\Bz_2\|^2_2 \ .
\end{equation}
Further assume that $\Phi$ and $\Phi_{\Bx_n}$ satisfy
\begin{equation}
\|\Phi(\Bx_*)-\Phi(\Bx_n)-\Phi_{\Bx_n}(\Bx_*-\Bx_n)\|^2_2\le C\|\Bx_*-\Bx_n\|^2_2 \ .  
\end{equation}
Then, the following inequality holds for all $n\ge 1$
\begin{equation}
\|\Bx_*-\Bx_{n+1}\|_2^2 \le 2\left(\frac{\beta}{\alpha}-1+\frac{4}{\alpha}C \right) \|\Bx_*-\Bx_{n}\|_2^2 \ .
\end{equation}•
\end{theorem}
We note that this theorem is conceptually similar to Theorem \ref{thm:coherence}, since we require the linearizations to obey a common RIP bound, as well as a bound on the error of the linearizations.  

We will assume that we are given RIP bounds for the sensing matrices $A$ and $B$. That is, we let $A$ and $B$ satisfy RIP bounds of order $s$ with constants $\delta_s^A$ and $\delta_s^B$, respectively.  By invoking the results of \cite{8171765}, the RIP constant for the linear matrix $\Phi$ defined by \eqref{eq:Kdef} is given by $\delta_s^2$, where $\delta_s=\max\{\delta_s^A,\delta_s^B\}$.  This result holds for linearized matrices of the form \eqref{eq:nlihtisp}. Now, given RIP bounds for the linear sensing matrices, we obtain a bound on the RIP constant for the linearizations.  Again, without loss of generality, we will use  \eqref{eq:nlihtisp} for the nonlinear IHT, and attach the extra factors in the linearization to the matrix $A$.  Let us first consider the linearization using the second Born approximation.  Results for the higher order terms follow immediately.  Next, we derive an RIP type bound for the matrix $A(I+VG)$, where $V$ is a diagonal matrix that is $s$-sparse. Using the triangle inequality, we have
\begin{equation}
(1-\|VG\|_2)^2\|\Bx\|_2^2 \le \|(I+VG)\Bx\|_2^2 \le (1+\|VG\|_2)^2\|\Bx\|_2^2 \ .
\end{equation}
Making use of this result, we obtain the inequalities
\begin{align}
(1-\|VG\|_2)^2(1-\delta_{2s}^A)\|\Bx\|_2^2 \le \|A(I+VG)\Bx\|^2_2 \le (1+\|VG\|_2)^2(1+\delta_{2s}^A)\|\Bx\|_2^2 \ ,
\end{align}
where we have introduced the RIP constant for $A$ of order $2s$.  Since $\Bx$ is an $s$-sparse vector and $V$ is $s$-sparse along its diagonal, the vector $(I+VG)\Bx$ is at most $2s$-sparse.  For higher order terms, $\|VG\|_2$ is replaced by the $\sum_{j=1}^M \|VG\|_2^j$.  Thus, for any order $M$, we have 
\begin{equation}
\label{eq:RIP_Ms}
(1-\delta_{2s}^A)\left(1-\sum_{j=1}^M \|VG\|_2^j\right)^2\|\Bx\|_2^2 \le \left\|\left(A\sum_{j=0}^M(VG)^j\right)\Bx\right\|_2^2 \le (1+\delta_{2s}^A)\left(1+\sum_{j=1}^M \|VG\|_2^j\right)^2\|\Bx\|_2^2 \ .
\end{equation}
This result gives us the constants $\alpha$ and $\beta$ in Theorem \ref{thm:B}.  In particular, by squaring the RIP constant, we have 
\begin{align}
&\alpha = 1-\left( (1-\delta_{2s})\left(1-\sum_{j=0}^M(VG)^j\right) ^2  -1\right)^2 \ , \\ 
 & \beta = 1+\left( (1+\delta_{2s})\left(1+\sum_{j=0}^M(VG)^j\right) ^2  -1\right)^2 \ .
\end{align}

Now, we bound the error term 
\begin{equation}
\|\Phi(\Bx_*)-\Phi(\Bx_n)-\Phi_{\Bx_n}(\Bx_*-\Bx_n)\| \ .  
\end{equation}
Introducing matrix notation, this term becomes 
\begin{equation}
\label{eq:B_err}
\|A(I-VG)^{-1}VB-A(I-V_nG)^{-1}VB\| \ ,
\end{equation}
which can be rewritten as
\begin{equation}
\|A(I-V_nG)^{-1}(V_n-V)(I-VG)^{-1}VB\| \ .
\end{equation}
Treating the ``vector'' as the diagonal matrix $V_n-V$, to compute an RIP bound, we are interested in the two matrices $A(I-V_nG)^{-1}$ and $\left((I-VG)^{-1}VB\right)^*$.  The RIP bound for the first matrix follows from Eq.~\eqref{eq:RIP_Ms}:
\begin{equation}
(1+\delta_{2s}^A)\left(\frac{1-2\|VG\|_2}{1-\|VG\|_2}\right)^2 \|\Bx\|_2^2 \le \|A(I-V_nG)^{-1}\Bx\|_2^2 \le (1+\delta_{2s}^A)\frac{1}{(1-\|VG\|_2)^2} \|\Bx\|_2^2  \ . 
\end{equation}
Similarly, for the second matrix, we have 
\begin{equation}
(1-\delta_{2s}^B)\left(\frac{m(1-2\|VG\|_2)}{1-\|VG\|_2}\right)^2 \|\Bx\|_2^2 \le \|B^*V((I-VG)^{-1})^*\Bx\|_2^2 \le (1+\delta_{2s}^B)\frac{\|V\|_{\infty}^2}{(1-\|VG\|_2)^2} \|\Bx\|_2^2  \ ,
\end{equation}
where $m$ is the smallest (in absolute value) nonzero entry of $V$. Once again using the results of \cite{8171765}, the RIP constant is given by 
\begin{align}
&\|A(I-V_nG)^{-1}(V_n-V)(I-VG)^{-1}VB\|^2_2 \le \left[1+\left(\frac{1+\delta_{2s}}{(1-\|VG\|_2)^2} -1\right)^2 \right]\|\Bv\|_\infty^2 \|\Bv-\Bv_n\|_2^2 , \\
\nonumber
\end{align}
which provides the constant $C$ in the theorem.  We can now summarize our results in the following theorem.

\begin{theorem}
Let $\By=\Phi(\Bv)$ be the scattering model, where $\Phi(\Bv)={\rm vec}(AV(I-\Gamma V)^{-1}B)$.  Let  $\gamma=\|\Gamma V\|_2<1/2$.  Suppose $A$ and $B$ satisfy the RIP on $2s$-sparse vectors with constants $\delta_{2s}^A$ and $\delta_{2s}^B$, respectively. Define $\delta_{2s}=\max\{\delta_{2s}^A,\delta_{2s}^B  \}$ and 
\begin{align}
 & \alpha = 1-\left( (1-\delta_{2s})\left(\frac{1-2\gamma}{1-\gamma} \right)^2  -1\right)^2 \ , \\ 
 & \beta = 1+\left(\frac{1+\delta_{2s}}{(1-\gamma)^2} -1\right)^2  ,\ \\
 & C = \beta\|\Bv\|_\infty^2 \ .
\end{align}
Then the nonlinear IHT algorithm converges if 
\begin{equation}
\frac{2}{3}(1+4\|\Bv\|_\infty^2)< \frac{\alpha}{\beta} \ .
\end{equation}
\end{theorem}

\bibliographystyle{siamplain}
\bibliography{references2}

\begin{thebibliography}{10}

\bibitem{blu_aiht}
{\sc T.~Blumensath}, {\em Accelerated iterative hard thresholding}, Signal
  Processing, 92 (2012), pp.~752 -- 756.

\bibitem{blu_nliht}
{\sc T.~Blumensath}, {\em Compressed sensing with nonlinear observations and
  related nonlinear optimization problems}, IEEE Transactions on Information
  Theory, 59 (2013), pp.~3466--3474.

\bibitem{blu_iht}
{\sc T.~Blumensath and M.~E. Davies}, {\em Iterative thresholding for sparse
  approximations}, Journal of Fourier Analysis and Applications, 14 (2008),
  pp.~629--654.

\bibitem{blu_iht2}
{\sc T.~Blumensath and M.~E. Davies}, {\em Iterative hard thresholding for
  compressed sensing}, Applied and Computational Harmonic Analysis, 27 (2009),
  pp.~265 -- 274.

\bibitem{blu_niht}
{\sc T.~Blumensath and M.~E. Davies}, {\em Normalized iterative hard
  thresholding: Guaranteed stability and performance}, IEEE Journal of Selected
  Topics in Signal Processing, 4 (2010), pp.~298--309.

\bibitem{Borcea_2015}
{\sc L.~Borcea and I.~Kocyigit}, {\em Resolution analysis of imaging with
  $\ell\_1$ optimization}, SIAM Journal on Imaging Sciences, 8 (2015),
  pp.~3015--3050.

\bibitem{Born99a}
{\sc M.~Born and E.~Wolf}, {\em Principles of Optics (7th Ed)}, Cambridge
  University Press, 1999.

\bibitem{cakoni_colton}
{\sc F.~Cakoni and D.~Colton}, {\em Qualitative Methods in Inverse Scattering
  Theory: An Introduction}, Interaction of Mechanics and Mathematics, Springer
  Berlin Heidelberg, 2005.

\bibitem{cakoni_colton_again}
{\sc F.~Cakoni and D.~Colton}, {\em A Qualitative Approach to Inverse
  Scattering Theory}, Applied Mathematical Sciences, Springer US, 2013.

\bibitem{CandesTao2005}
{\sc E.~J. Cand\`es and T.~Tao}, {\em Decoding by linear programming}, IEEE
  Transactions on Information Theory, 51 (2005), pp.~4203--4215.

\bibitem{chadan}
{\sc K.~Chadan and P.~C. Sabatier}, {\em Inverse Problems in Quantum Scattering
  Theory}, Springer, 1st~ed., 1989.

\bibitem{colton_kress}
{\sc D.~Colton and R.~Kress}, {\em Integral equation methods in scattering
  theory}, Pure and applied mathematics, Wiley, 1983.

\bibitem{colton_kress_again}
{\sc D.~Colton and R.~Kress}, {\em Inverse Acoustic and Electromagnetic
  Scattering Theory}, Applied Mathematical Sciences, Springer Berlin
  Heidelberg, 1997.

\bibitem{Fanngiang_Remote}
{\sc A.~Fannjiang, T.~Strohmer, and P.~Yan}, {\em Compressed remote sensing of
  sparse objects}, SIAM Journal on Imaging Sciences, 3 (2010), pp.~595--618.

\bibitem{0266-5611-26-3-035008}
{\sc A.~C. Fannjiang}, {\em Compressive inverse scattering: I. high-frequency
  simo/miso and mimo measurements}, Inverse Problems, 26 (2010), p.~035008.

\bibitem{0266-5611-27-3-035013}
{\sc A.~C. Fannjiang}, {\em The music algorithm for sparse objects: a
  compressed sensing analysis}, Inverse Problems, 27 (2011), p.~035013.

\bibitem{isakov}
{\sc V.~Isakov}, {\em Inverse Problems for Partial Differential Equations},
  Applied Mathematical Sciences, Springer New York, 1997.

\bibitem{0266-5611-33-6-060301}
{\sc B.~Jin, P.~Maaß, and O.~Scherzer}, {\em Sparsity regularization in
  inverse problems}, Inverse Problems, 33 (2017), p.~060301.

\bibitem{doi:10.1137/1.9781611970944}
{\sc C.~Kelley}, {\em Iterative Methods for Linear and Nonlinear Equations},
  Society for Industrial and Applied Mathematics, 1995.

\bibitem{8171765}
{\sc S.~Khanna and C.~R. Murthy}, {\em On the restricted isometry of the
  columnwise khatri-rao product}, IEEE Transactions on Signal Processing, 66
  (2018), pp.~1170--1183.

\bibitem{10.2307/43693726}
{\sc K.~Kilgore, S.~Moskow, and J.~C. Schotland}, {\em Inverse born series for
  scalar waves}, Journal of Computational Mathematics, 30 (2012), pp.~601--614.

\bibitem{kirch}
{\sc A.~Kirsch}, {\em An Introduction to the Mathematical Theory of Inverse
  Problems}, Springer-Verlag, Berlin, Heidelberg, 1996.

\bibitem{5728925}
{\sc O.~Lee, J.~M. Kim, Y.~Bresler, and J.~C. Ye}, {\em Compressive diffuse
  optical tomography: Noniterative exact reconstruction using joint sparsity},
  IEEE Transactions on Medical Imaging, 30 (2011), pp.~1129--1142.

\bibitem{Tmatrix1}
{\sc H.~W. Levinson and V.~A. Markel}, {\em Solution of the nonlinear inverse
  scattering problem by $t$-matrix completion. i. theory}, Phys. Rev. E, 94
  (2016), p.~043317.

\bibitem{hadamard}
{\sc S.~Liu and O.~TRENKLER}, {\em Hadamard, khatri-rao, kronecker and other
  matrix products}, International Journal of Information and Systems Sciences,
  4 (2008).

\bibitem{NguyenShin2013}
{\sc T.~L.~N. Nguyen and Y.~Shin}, {\em Deterministic sensing matrices in
  compressive sensing: A survey}, The Scientific World Journal, 2013 (2013).

\bibitem{1973ApJ186705P}
{\sc E.~M. {Purcell} and C.~R. {Pennypacker}}, {\em {Scattering and Absorption
  of Light by Nonspherical Dielectric Grains}}, apj, 186 (1973), pp.~705--714.

\bibitem{RudelsonVershynin2008}
{\sc M.~Rudelson and R.~Vershynin}, {\em On sparse reconstruction from fourier
  and gaussian measurements}, Communications on Pure and Applied Mathematics,
  61 (2008), pp.~1025--1045.

\bibitem{wang}
{\sc Y.~Wang, J.~Zeng, Z.~Peng, X.~Chang, and Z.~Xu}, {\em Linear convergence
  of adaptively iterative thresholding algorithms for compressed sensing.},
  IEEE Trans. Signal Processing, 63 (2015), pp.~2957--2971.

\bibitem{PMC2858419}
{\sc D.~W. Winters, B.~D. Van~Veen, and S.~C. Hagness}, {\em The music
  algorithm for sparse objects: a compressed sensing analysis}, IEEE
  transactions on antennas and propagation, 58 (2010), pp.~145--154.

\end{thebibliography}
\end{document}